\documentclass[11pt,reqno]{amsart}
\usepackage{fullpage}
\newcommand{\orcidlogo}{{\includegraphics[width=\fontcharht\font`l]{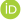}}}

\newcommand{\Addresses}{% additional braces for segregating \footnotesize
\setlength{\parindent}{0pt}

{
\bigskip
\footnotesize

    Dimitri Bytchenkoff \par\nopagebreak
    \textsc{Faculty of Mathematics, University of Vienna, 
Oskar-Morgenstern-Platz 1, 1090 Vienna, Austria; Acoustics Research Institute, Austrian Academy of Sciences, Dominika- nerbastei 16, 1010 Vienna, Austria; Université de Lorraine, CNRS, LEMTA, 54000 Nancy, France}\par\nopagebreak
    \textit{E-mail address}: \texttt{dimitri.bytchenkoff@univie.ac.at}\par\nopagebreak
    \href{https://orcid.org/0000-0001-5673-3260}{ 
   \orcidlogo\,0000-0001-5673-3260}
   
}

{
\bigskip
\footnotesize

    Arvin Lamando\par\nopagebreak

     \textsc{Institute of Mathematics, University of the Philippines, Diliman, 1101 Quezon City, Philippines}\par\nopagebreak
    \textit{E-mail address}: 
    \texttt{alamando@math.upd.edu.ph}
    \par\nopagebreak
    \href{https://orcid.org/0009-0001-2969-4857}{ 
   \orcidlogo\,0009-0001-2969-4857}
}

{
\bigskip
\footnotesize

    Franz Luef \par\nopagebreak
    \textsc{Department of Mathematical Sciences, NTNU Trondheim, Alfred Getz' vei 1, 7491 Trondheim, Norway}\par\nopagebreak
    \textit{E-mail address}: \texttt{franz.luef@ntnu.no}
    \par\nopagebreak
    \href{https://orcid.org/0000-0001-7413-8350}{ 
   \orcidlogo\,0000-0001-7413-8350}
}}
\title{Banach-Gelfand Triples with Kernel Theorems via Equivalence Bimodules}
\author{Dimitri Bytchenkoff, Arvin Lamando, and Franz Luef}
\date{}
\newlength{\bibitemsep}
\newlength{\bibparskip}
\let\oldthebibliography\thebibliography
\renewcommand\thebibliography[1]{%
  \oldthebibliography{#1}%
  \setlength{\parskip}{\bibitemsep}%
  \setlength{\itemsep}{\bibparskip}%
}
\usepackage[T1]{fontenc}
\usepackage{adjustbox}
\usepackage{hyphenat}
\usepackage{listings}
\usepackage{enumitem}
\usepackage{setspace}
\usepackage[dvipsnames]{xcolor}
\usepackage[hypertexnames=false, colorlinks, citecolor=BrickRed,linkcolor=MidnightBlue, urlcolor=BrickRed]{hyperref}
\usepackage{acro}
\usepackage{amsthm}
\usepackage{subcaption}
\usepackage{pifont}
\usepackage{cancel}
\usepackage{leftindex}
\usepackage[normalem]{ulem}
\usepackage{tikz}
\usepackage{amsmath,amssymb, mathtools}
\usepackage{mathrsfs}
\usepackage[scr=boondox,  % heavily sloped
            cal=esstix]   % slightly sloped
           {mathalpha}
\usepackage{multicol}
\usepackage{tikz}
\usetikzlibrary{matrix,arrows,decorations.pathmorphing}
\usepackage{tikz-cd}
\tikzset{commutative diagrams/.cd}
\usepackage{tikz}
\usetikzlibrary{matrix,arrows,decorations.pathmorphing}
\usepackage{tikz-cd}
\tikzset{commutative diagrams/.cd}
\usepackage{tikz}
\usetikzlibrary{
  calc,
  decorations.pathmorphing,
  shapes,
  arrows.meta
}

\newtheorem{exmp}{Example}[section]

\newtheoremstyle{mystyle}
{\topsep}% measure of space to leave above the theorem. E.g.: 3pt
{\topsep}% measure of space to leave below the theorem. E.g.: 3pt
{\itshape}% name of font to use in the body of the theorem
{0pt}% measure of space to indent
{\bfseries\scshape}% name of head font
{.}% punctuation between head and body
{ }% space after theorem head; " " = normal interword space
{\thmname{#1}\thmnumber{ #2}\thmnote{ (#3)}}

\newtheorem{proposition}[exmp]{Proposition}

\newtheorem{lemma}[exmp]{Lemma}

\newtheorem{theorem}[exmp]{Theorem}

\newtheorem{corollary}[exmp]{Corollary}

\theoremstyle{definition}
\newtheorem{definition}[exmp]{Definition}

\newtheorem{remark}[exmp]{Remark}

\newtheorem{example}[exmp]{Example}

\theoremstyle{remark}

\theoremstyle{definition}

\def\hmath$#1${\texorpdfstring{{\rmfamily\textit{#1}}}{#1}}

\newcommand{\mycomment}[1]{}
\newcommand{\op}[1]{\operatorname{#1}}

\newcommand{\lin}[3]{\leftindex_{#1}{\left \langle #2,#3 \right\rangle}}
\newcommand{\rin}[3]{\left\langle #2,#3 \right\rangle_{#1}}

\newcommand{\comment}[1]{}

\begin{document}
\begin{abstract}
    Hilbert C*-modules whose coefficient C*-algebra is equipped with a finite faithful trace generate Banach Gelfand triples. This provides a natural setting for the interpretation of Hilbert C*-modules as function spaces. In particular, when the modules are finitely generated and projective, their external tensor products interact well with projective tensor products. This allowed us to derive abstract kernel theorems for Hilbert C*-modules. As an application, we apply our results to Heisenberg modules and obtain new kernel theorems for time-frequency analysis.
\end{abstract}

\maketitle

\section{Introduction}
Gelfand triples arose as part of an ongoing effort to formalise quantum mechanics, also known as the rigged Hilbert spaces, introduced in \cite{GeKo55}, with subsequent developments published in \cite{GeIz64}. The idea is to rig, or rather, to extend a Hilbert space $\mathcal{H}$ to include generalised vectors that lie outside the original Hilbert space, but can nevertheless be seen as a functional on some dense subspace, say $X$, inside $\mathcal{H}.$ In general, there is a topological vector space $X$ that is densely embedded in $\mathcal{H}$ and $\mathcal{H}$ is, in its turn, weak-$*$ densely embedded in the continuous dual space $X'$, giving the eponymous \emph{Gelfand triple} $$X\hookrightarrow \mathcal{H} \hookrightarrow X'.$$ 

A classical example of a Gelfand triple is $\mathscr{S}(\mathbb{R}^d)\hookrightarrow L^2(\mathbb{R}^d)\hookrightarrow \mathscr{S}'(\mathbb{R}^d)$, where $\mathscr{S}(\mathbb{R}^d)$ is the Schwartz test function space, known to be densely embedded in the square-integrable functions $L^2(\mathbb{R}^d)$. Calculations involving vectors that are known to be non-square integrable, such as the Dirac delta functions for example, are now placed in a more rigorous framework as they are instead known to be distributions inside the dual space $\mathscr{S}'(\mathbb{R}^d)$ of tempered distributions \cite{Sc50-I-II}. In a related direction, there is the so-called \emph{Schwartz kernel theorem} \cite{Sc50-1,Ho03}, which says that every  continuous linear operator $T:\mathscr{S}(\mathbb{R}^d)\to \mathscr{S}'(\mathbb{R}^m)$ has a unique $\emph{kernel}$ $\kappa(T)\in \mathscr{S}'(\mathbb{R}^d\times \mathbb{R}^m)$ satisfying
\begin{align*}
    \rin{\mathscr{S}(\mathbb{R}^m),\mathscr{S}'(\mathbb{R}^m)}{g}{Tf} = \rin{\mathscr{S}(\mathbb{R}^d\times \mathbb{R}^m),\mathscr{S}'(\mathbb{R}^d\times \mathbb{R}^m)}{f\otimes g}{\kappa(T)}, \qquad \forall (f,g)\in \mathscr{S}(\mathbb{R}^d)\times \mathscr{S}(\mathbb{R}^m),
\end{align*}
where the tensor $f\otimes g$ is identified with the pointwise multiplication of $f$ and $g$. One of Grothendieck's early works is related to the abstract 
study of tensors of topological vector spaces from which one may derive the Schwartz kernel theorem \cite{Gr95-1}. As it turns out, tensors of topological vector spaces are quite subtle, and one of the central challenges is to know when the so-called \emph{projective tensor product} is embedded, i.e. mapped injectively and continuously, in certain spaces of interest. For the purposes of this paper, we are again brought back to the notion of Gelfand triples, as we would like to find out when these projective tensor products can be densely embedded in tensors of Hilbert C*-modules.

To get closer to our main motivation, we are interested in extending the Schwartz kernel theorem to a somewhat exotic function space that appears in time-frequency analysis. To start with, we have the Feichtinger algebra $\mathbf{S}_0(\mathbb{R}^d)$, introduced by Feichtinger in \cite{Fe81}. It enjoys several properties very similar to the Schwartz space (see \cite{Ja18} for a recent survey). However, perhaps the most striking difference between $\mathbf{S}_0(\mathbb{R}^d)$ and $\mathscr{S}(\mathbb{R}^d)$ is that the former is a Banach space, while the latter is only a Fr\'echet space. Furthermore, $\mathbf{S}_0(\mathbb{R}^d)$ also generates a Gelfand triple of Banach spaces \cite{CoFeLu08, Fe09, Fe18-3, Fe24-7}: $$\mathbf{S}_0(\mathbb{R}^d)\hookrightarrow L^2(\mathbb{R}^d)\hookrightarrow \mathbf{S}_0'(\mathbb{R}^d).$$ Henceforth, we call such Gelfand triples \emph{Banach-Gelfand triples}.

In \cite{Ri88}, Rieffel constructed the so-called \emph{Heisenberg modules} as a Banach space completion of the Schwartz space to study the geometry of non-commutative tori (a C*-algebra). These specific modules over C*-algebras are nowadays called Hilbert C*-modules \cite{La95}. In particular, finitely generated full Hilbert C*-modules are interpreted as non-commutative complex vector bundles \cite{Se55,Sw62}. Full Hilbert C*-modules can always be realised as \emph{equivalence bimodules}, and are important tools in the structure theory of C*-algebras \cite{ Rie74,RaWi98}. In \cite{Lu07}, Luef showed that the Heisenberg modules can instead be constructed using Feichtinger algebra, making new connections between non-commutative geometry and time-frequency analysis \cite{Lu09, JaLu21}. More specifically, given a lattice $\Delta\subseteq \mathbb{R}^{2d}$, the Heisenberg module $\mathcal{E}_{\Delta}(\mathbb{R}^d)$ contains the Feichtinger algebra $\mathbf{S}_0(\mathbb{R}^d)$ and $\mathbf{S}_0(\mathbb{R}^d)$ is dense in $\mathcal{E}_{\Delta}(\mathbb{R}^d)$. Furthermore, from \cite{AuEn20}, it is known that there is the dense embedding $\mathcal{E}_{\Delta}(\mathbb{R}^d)\hookrightarrow L^2(\mathbb{R}^d)$ by localisation using a faithful finite trace (see Theorem \ref{thm: localization} or the original paper \cite{AuEn20}), which allows us to interpret $\mathcal{E}_{\Delta}(\mathbb{R}^d)$ as an exotic function space. 

It is now natural to ask whether the Heisenberg modules, similarly to the Feichtinger algebras and the Schwartz function spaces, have kernel theorems of their own. The answer is affirmative even in the abstract setting of LCA groups. Indeed, we shall show that there are abstract kernel theorems for finitely generated projective full Hilbert C*-modules that can be faithfully localised by a finite faithful trace. This paper is organised as follows. In Section \ref{sec: hilb-cstar}, we first discuss functional analytic notations pertaining to Banach spaces and equivalence bimodules. A Banach-Gelfand triple result of (pre)-equivalence bimodules is obtained in Lemma \ref{lem: localizable-is-gelfand-triple} and Theorem \ref{thm: gelfand-quintuple}.  We then introduce the notion of $\tau$-adjointability. It allows us to generalise adjointable maps between Hilbert C*-modules with different but $\tau$-isomorphic coefficient C*-algebras. In Section \ref{sec: tensors}, we review tensor products for general Banach spaces, C*-algebras, and Hilbert C*-modules, together with their interactions with one another. There is an extended discussion on the so-called \emph{approximation property} (Proposition \ref{prop: the-ap}), as this will allow us to generate Banach-Gelfand triples from projective tensor products. In Section \ref{sec: hilb-cstar-tensors}, we obtain our main abstract results on the tensors of equivalence bimodules, including their operator interpretations. Section \ref{sec: hilb-cstar-tensors} culminates in Theorem \ref{thm: abstract-kernels}, which summarises the abstract kernel theorems for (pre)-equivalence bimodules and can be laid out as a diagram of ever-finer Gelfand triples involving equivalence bimodules. In Section \ref{sec: tfa}, we first provide some essential information on time-frequency analysis, including the approximation property of the Feichtinger algebra $\mathbf{S}_0(G)$ as $G$ is a second countable LCA group (Theorem \ref{thm: feich-has-ap}). We then show that the external tensor product of Heisenberg modules is also a Heisenberg module \ref{thm: concrete-external-tensor} and, as a corollary, obtain the inner and outer kernel theorems for the Heisenberg modules (Theorems \ref{thm: heis-outer-kernel} and \ref{thm: heis-inner-kernel}).  

\section{Faithfully Localisable Hilbert C*-modules}\label{sec: hilb-cstar}
We shall assume that the reader is familiar with the basics of C*-algebras and Hilbert C*-modules \cite{Mu90, RaWi98, La95}. In this paper, we shall mean a \emph{left} Hilbert C*-module by a Hilbert C*-module over a C*-algebra $A$ or simply a Hilbert $A$-module. Otherwise, we shall explicitly state that it deals with a \emph{right} Hilbert C*-module. A typical Hilbert $A$-module $Z$ comes with an $A$-valued inner product linear in the first variable. Since we shall be working with different Hilbert C*-modules at the same time, we shall label our inner products
\begin{align}\label{form: left-inner} 
    \lin{A}{z_1}{z_2}, \qquad \forall z_1,z_2\in Z.
\end{align}
We shall use the notation $\|z\|_Z$ for the module norm on $Z$. We shall denote the space of all adjointable maps from a Hilbert $A$-module $Z$ onto a Hilbert $A$-module $W$ by $\mathcal{L}_A(Z,W)$. If $Z$ and $W$ are right Hilbert $B$-modules, $\mathcal{L}_B(Z,W)$ will stand for the space of all adjointable maps from $Z$ onto $W$. If $Z=W$, $\mathcal{L}_A(Z,W):=\mathcal{L}_A(Z)$, $\mathcal{L}_B(Z,W):=\mathcal{L}_B(Z)$ and the C*-algebra of compact adjointable operators are denoted by $\mathcal{K}_A(Z)\subseteq \mathcal{L}_A(Z)$ and $\mathcal{K}_B(Z)\subseteq \mathcal{L}_B(Z)$. We shall use the notation $S(A)$ for the states of a C*-algebra $A$. The space of bounded linear operators from a Banach space $X$ onto a Banach space $Y$ will be denoted by $\mathcal{B}(X,Y)$. The topological dual of a Banach space is denoted by $X'$. We shall use the bilinear dual brackets $\rin{X,X'}{\cdot}{\cdot}: X\times X'\to \mathbb{C}$ as follows
\begin{align*}
    \sigma(x) := \rin{X,X'}{x}{\sigma}, \qquad \forall x\in X, \sigma \in X'.
\end{align*}
The operator $\theta_{x',y} \in \mathcal{B}(X,Y)$ of \emph{rank one} is defined by $\theta_{x',y}(x) := \rin{X,X'}{x}{x'}y$ for $x'\in X$ and $y\in Y$. The sum of any number $N$ operators of rank one is a  operator of \emph{of rank} $N$, belongs to $\mathcal{B}(X,Y)$ and is also called an operator \emph{of finite rank}. The completion of the set of all finite-rank operators in operator norm is the space of \emph{approximable operators} $\mathcal{A}(X,Y)$ mapping the Banach space $X$ onto the Banach space $Y$. This space is in general different from the space \emph{compact operators} $\mathcal{K}(X,Y)$ that transform bounded subsets of $X$ into relatively compact subsets in $Y$. If $X$ and $Y$ are Hilbert spaces, $\mathcal{A}(X,Y)=\mathcal{K}(X,Y)$. The same is true if $X$ and $Y$ are Banach spaces and either $X$ or $Y$ has the so-called \emph{approximation property} (see Proposition \ref{prop: the-ap}). If $T\in \mathcal{B}(X,Y)$, the \emph{formal adjoint} $T': Y'\to X'$ is defined by
\begin{align*}
    T'(y') = y'\circ T, \qquad y'\in Y'.
\end{align*}
Since $T$ is a continuous linear map, $T'$ is also the unique continuous linear map that satisfies the duality relation
\begin{align*}
    \rin{Y,Y'}{Tx}{y'} = \rin{X,X'}{x}{T'y'}. 
\end{align*}
One of the main questions that will concern us is when a bounded linear map is injective. We shall generally refer to such maps as \emph{embeddings} and use hooked arrows $T:X\hookrightarrow Y$ for the embedding of a space $X$ in a space $Y$. Unless we need to emphasise it, we shall not mention that a map is an embedding and simply treating an embedded space as a subspace of the space in which it is embedded. There is a very useful duality between embeddings and maps with dense range \cite[Theorem 4.12]{Ru91-2}. 
\begin{theorem}\label{thm: embedding-dense-range}
    Let $T\in \mathcal{B}(X,Y)$, then
    \begin{enumerate}
        \item $T\in \mathcal{B}(X,Y)$ is an embedding if and only if $T'\in \mathcal{B}(Y',X')$ has a weak-$*$ dense range in $X'$. 
        \item $T\in \mathcal{B}(X,Y)$ has a norm dense range in $Y$ if and only if $T\in \mathcal{B}(Y',X')$ is an embedding. 
        \item If $T:X\to Y$ is a Banach space isomorphism, $T':Y'\to X'$ is also a Banach space isomorphism and a homeomorphism with respect to weak-$*$ topologies.
    \end{enumerate}
\end{theorem}

We shall consider only \emph{full} Hilbert C*-modules $Z$ over a C*-algebra $A$, i.e. the span of $A$-valued inner products such as in \eqref{form: left-inner} are dense in $A$. If so, $Z$ is an equivalence bimodule \cite[Proposition 3.8]{RaWi98} (A.K.A. imprimitivity bimodule) for the C*-algebras $A$ and $\mathcal{K}_A(Z)=:B$. Note that this setting is completely symmetric, that is, if $Z$ is a full right Hilbert $B$-module, then it is also an equivalence bimodule between $\mathcal{K}_B(Z)$ and $B$. The notion of full Hilbert C*-modules and equivalence bimodules are then interchangeable. If $Z$ is an $A-B$ equivalence bimodule, we shall use $\lin{A}{\cdot}{\cdot}$ to denote the $A$-valued inner product on $Z$, while we use $\rin{B}{\cdot}{\cdot}$ for the $B$-valued inner product on $Z$.

Hilbert C*-modules can be viewed as the non-commutative version of complex vector bundles as that there is the equivalence of the category of vector bundles over a fixed compact Hausdorff space and that of finitely generated projective modules over the commutative C*-algebra associated with the above-mentioned compact Hausdorff space \cite{Se55,Sw62}. This is why we are particularly interested in full Hilbert $A$-modules $Z$ that are finitely generated and projective as an $A$-module. From the equivalence bimodule picture, it turns out that this is equivalent \cite[Proposition 2.3]{AuJaMaLu20} to $Z$ having the right C*-algebra coefficient $B:=\mathcal{K}_A(Z)$ unital. The following results will be of particular use to us.
\begin{proposition}\label{prop: parseval-module-frame}
    Let $Z$ be an $A-B$ equivalence bimodule and $B$ unital. Then there are such sets $\{f_1,...,f_n\}\subseteq Z$ and $\{g_1,...,g_n\}\subseteq Z$ that $\displaystyle \sum_{i=1}^n \rin{B}{f_i}{g_i}=1_B$, where $1_B$ is the unital element in $B$. Furthermore, the elements can be chosen so that $w_i=f_i=g_i\in Z$ for all $n=1,...,n.$
\end{proposition}
\begin{proof}
    Let $I=\op{span}\{\rin{B}{z}{w}: z,w\in Z\}$ and so by fullness of $Z$ as a right Hilbert $B$-module, $\overline{I}^{\|\cdot\|_B}=B$. It follows that there is such a $b\in I$ that $\|1_B-b\|_{B}<1$ and so, by Neumann Series, there must be the inverse $(1_B-(1_B-b))^{-1}=b^{-1}$ in $B$. However this implies $b\cdot b^{-1}\in I$, and by the ideal property of $I$, $I=B$. It follows that there are such sets $\{f_1,...,f_n\}$ and $\{g_1,...,g_n\}$ in $Z$ that $\displaystyle \sum_{i=1}^n \rin{B}{f_i}{g_i}=1_{B}.$ Define the maps $F_f, F_g: Z\to Z$ by $\displaystyle F_f(z) = \sum_{i=1}^n \lin{A}{z}{f_i}\cdot f_i$ and $\displaystyle F_g(z) = \sum_{i=1}^n \lin{A}{z}{g_i}\cdot g_i$ for all $Z\in Z.$ These maps are the so-called \emph{module frame} maps. They are positive as there are only finitely many elements associated with both $F_f$ and $F_g$ and due to the isomorphism $\mathcal{K}_A(Z)\cong B$. They are necessarily invertible in $\mathcal{K}_A(Z)$ with $F_f^{-1}=F_g$ \cite[Proposition 3.9]{AuJaMaLu20}. Let us define $w_i = F_g^{1/2}(f_i)$ for each $i=1,...,n$. For any $z\in Z$,
    \begin{align*}
        z = F_g^{1/2}(F_f (F_g^{1/2}z))= \sum_{i=1}^n \lin{A}{z}{F_g^{1/2}(f_i)}\cdot F_{g}^{1/2}(f_i) = \sum_{i=1}^n \lin{A}{z}{w_i}\cdot w_i = z\cdot \sum_{i=1}^n \rin{B}{w_i}{w_i}.
    \end{align*}
    Therefore $\displaystyle z\cdot \left(1_B -z\cdot \sum_{i=1}^n \rin{B}{w_i}{w_i}\right)=0$ and, since the module action in a Hilbert C*-module is faithful, $\displaystyle \sum_{i=1}^n\rin{B}{w_i}{w_i}=1_B.$
\end{proof}
\begin{corollary}
    If $Z$ is an $A-B$ equivalence bimodule and $B$ is unital, then every adjointable map in $\mathcal{L}_A(Z)$ is finite-rank (in the module sense). This implies, in particular, that $\mathcal{L}_A(Z)=\mathcal{K}_A(Z).$
\end{corollary}
\begin{proof}
    From Proposition \ref{prop: parseval-module-frame}, there are such $\{f_1,...,f_n\}\subseteq Z$ and $\{g_1,...,g_n\}\subseteq Z$ that $\displaystyle \sum_{i=1}^n \rin{B}{f_i}{g_i}=1_B$. Let $T\in \mathcal{L}_A(Z)$. Then, for any $z\in Z$,
    \begin{align*}
        T(z) = T(z\cdot 1_B) = T\left(z\cdot \sum_{i=1}^n \rin{B}{f_i}{g_i} \right) = T\left( \sum_{i=1}^n\lin{A}{z}{f_i}\cdot g_i \right)=\sum_{i=1}^n \lin{A}{z}{f_i}\cdot T(g_i),
    \end{align*}
    where we used the fact that $T$ is necessarily $A$-linear. It follows that we can write $T$ as an $n$-sum of (module) rank-one operators $F_{f_i,w_i}: Z\to Z$ defined by
    \begin{align*}
        F_{f_i,w_i}(z) = \lin{A}{z}{f_i}\cdot w_i
    \end{align*}
    with $w_i=T(g_i)$ for all $i=1,...,n.$ Indeed, $T = \sum_{n=1}^n F_{f_i,w_i}$, and therefore $T$ must be finite-rank.
\end{proof}
\begin{proposition}\label{prop: pre-parseval-module-frame}
    Let $\mathcal{A}$ and $\mathcal{B}$ be Banach $*$-algebras with C*-envelopes $A$ and $B$ respectively. Let $\mathcal{B}$ and $B$ be unital with a common unit $1_B$, and $\mathcal{B}$ inverse-closed in $B$. If $\mathcal{Z}$ is an $\mathcal{A}-\mathcal{B}$ pre-equivalence bimodule, there are such sets $\{\mathscr{f}_1,...,\mathscr{f}_n\} \subseteq \mathcal{Z}$ and $\{\mathscr{g}_1,...,\mathscr{g}_n\}\subseteq \mathcal{Z}$ that $\displaystyle \sum_{i=1}^n \rin{\mathcal{B}}{\mathscr{f}_i}{\mathscr{g}_i}=1_{B}$.
\end{proposition}
\begin{proof}
    We consider the ideal $\mathcal{I} = \op{span}\{\rin{\mathcal{B}}{\mathscr{f}}{\mathscr{g}}:\mathscr{f},\mathscr{g}\in \mathcal{Z}  \}$ of $\mathcal{B}$. Since $\mathcal{Z}$ is an $\mathcal{A}-\mathcal{B}$ pre-equivalence bimodule, $\overline{\mathcal{I}}^{\|\cdot\|_B}=B$. Moreover, there must be an element in $\mathscr{b}\in \mathcal{I}$ that is invertible in $B$. However, since $\mathcal{B}$ is inverse-closed, $\mathscr{b}$ must be invertible in $\mathcal{B}$ too. Since $\mathcal{I}$ is an ideal in $\mathcal{B}$, $\mathscr{b}\cdot \mathscr{b}^{-1}\in \mathcal{I}$. Therefore $\mathcal{I}=\mathcal{B}.$ The desired result now follows from the fact that $1_B\in \mathcal{I}.$
\end{proof}

It will be important for us to be able to continuously and densely embed full Hilbert C*-modules in Hilbert spaces. Furthermore, whenever we talk about faithful finite traces, say $\tau$ on a C*-algebra $B$, we mean a faithful positive linear functional satisfying the trace condition $\tau(b_1b_2)=\tau(b_2b_1)$ for all $b_1,b_2\in B$. When one is dealing with nice (unital) C*-algebras, the term `finite' comes from the fact that there may be faithful positive maps on a C*-algebra satisfying the trace condition on a dense subset, albeit possibly unbounded with respect to the C*-norm. This is why we recall the \emph{localisation scheme} of \cite{AuEn20}. 
\begin{theorem}\label{thm: localization}
    Let $Z$ be an $A-B$ equivalence bimodule and $\op{tr}_B: B\to \mathbb{C}$ a faithful finite trace in $B$. Then there is a unique lower semi-continuous faithful trace defined on $\op{span}\{\lin{A}{z}{w}:z,w\in Z \}$ that satisfies
    \begin{align*}
        \op{tr}_A\left(\lin{A}{z}{w}\right) = \op{tr}_B\left(\rin{B}{w}{z} \right) =: \rin{\op{tr}_B}{z}{w}, \qquad z,w\in Z.
    \end{align*}
    \begin{enumerate}
        \item The form $\rin{\op{tr}_A}{z}{w}:=\op{tr}_A\left(\lin{A}{z}{w}\right)$ defines an inner product on $Z$ that satisfies $\|z\|_{\op{tr}_A} = \sqrt{\rin{\op{tr}_A}{z}{z}}=\sqrt{\rin{\op{tr}_B}{z}{z}}$ for $z,w\in Z$. The completion of $Z$ with respect to the inner product $\rin{\op{tr}_A}{\cdot}{\cdot}$ coincides with the localisation of $Z$ with respect to $\rin{\op{tr}_B}{\cdot}{\cdot}$. One then denotes the induced Hilbert space by $\op{Loc}(Z)$ and there is the dense embedding $Z\hookrightarrow \op{Loc}(Z)$ with the following estimate for $z\in Z$
        \begin{equation}\label{form: loc-estimate}
            \begin{split}
                \|z\|_{\op{tr}_A} \leq \|\op{tr}_B\| \|z\|_Z.
            \end{split}
        \end{equation}
        \item Let $W$ be another $A-B$ equivalence bimodule. The localisation $\op{Loc}$ can be seen as a functor from equivalence $A-B$ bimodules to Hilbert spaces. There are the embeddings
        \begin{equation}\label{form: loc-op-embedding}
            \begin{split}
            \mathcal{L}_A(Z,W)& \hookrightarrow \mathcal{B}(\op{Loc}(Z),\op{Loc}(W)), \\
            \mathcal{L}_B(Z,W)& \hookrightarrow \mathcal{B}(\op{Loc}(Z),\op{Loc}(W)).
            \end{split}
        \end{equation}
        If $Z=W$, the embedding \eqref{form: loc-op-embedding} is an isometry. 
    \end{enumerate}
\end{theorem}
\noindent In the above-mentioned setting, we generally denote the inner product on $\op{Loc}(Z)$ by $\rin{\op{Loc}(Z)}{\cdot}{\cdot}$. When we are working with the dense subspace $Z$ and want to emphasise that the inner product is induced either by traces on $A$ or $B$, we shall use $\rin{\op{tr}_A}{\cdot}{\cdot}$ or $\rin{\op{tr}_B}{\cdot}{\cdot}$ respectively. In practice, $A-B$ equivalence bimodules such as $Z$ are usually constructed densely. Let $A_0$ and $B_0$ be dense $*$-sub-algebras of $A$ and $B$ respectively and $\mathcal{Z}$ an $A_0-B_0$ bimodule satisfying $A_0-B_0$ \emph{pre-equivalence bimodule} properties (A.K.A. $A_0-B_0$ \emph{imprimitivity bimodule properties}) \cite[Definition 3.9]{RaWi98}. Then $\mathcal{Z}$ can be uniquely completed to the $A-B$ equivalence bimodule $Z$ \cite[Proposition 3.12]{RaWi98}. In this construction, $\mathcal{Z}$ lies in $Z$ as a dense $A_0-B_0$ bimodule.

The following definition encapsulates the setting of Theorem \ref{thm: localization}.
\begin{definition}\label{def: faithful-localization}
    If $Z$ is an $A-B$ equivalence bimodule with a faithful finite trace $\op{tr}_B: B\to \mathbb{C}$, $Z$ is said to be \emph{faithfully localisable}. Let $A_0$ and $B_0$ be dense $*$-sub-algebras of $A$ and $B$ respectively and $\mathcal{Z}$ an $A_0-B_0$ pre-equivalence bimodule that can be completed to $Z$. $Z$ is said to be \emph{faithfully localisable from} $\mathcal{Z}$ if there exists a faithful finite trace $\op{tr}_{B_0}: B_0\to \mathbb{C}$ that can be extended to a faithful finite trace $\op{tr}_{B}: B\to \mathbb{C}$.
\end{definition}
\begin{remark}\label{rem: extending-traces}
    In other words, $\op{tr}_{B_0}:B_0\to \mathbb{C}$ is said to be a faithful finite trace in $B_0$, if it satisfies the trace condition $\op{tr}_{B_0}(b_1b_2)= \op{tr}_{B_0}(b_2b_1)$ for all $b_1,b_2\in B_0$ and $\op{tr}(b_0)\geq 0$ whenever $b_0\geq 0$ in $B$ for all $b_0\in B_0$ (and one obtains exactly $\op{tr}_{B_0}(b_0)$ when $b_0=0$). Such traces cannot be extended to the C*-algebra $B$ and, in general, boundedness with respect to the C*-norm is a very strong condition. As an example, the trace-class operators $\mathcal{S}_1(H)$ are dense in C*-algebra $\mathcal{K}(H)$ of compact operators with respect to the operator norm. One can define a trace $\op{Tr}$ in $\mathcal{S}_1(H)$ that is faithful in the sense given above, but $\op{Tr}$ is unbounded with respect to the operator norm, and thus cannot be extended to $\mathcal{K}(H)$. 
\end{remark} 
We now briefly discuss the continuous dual of a Hilbert space $H$, which we shall denote by $H'$. In this work, we shall follow the standard mathematical convention of Hilbert spaces having a sesqui-linear form or inner product that is linear in the former variable, and anti-linear in the latter variable. By the Riesz-representation theorem, there is an anti-linear Banach space isomorphism $R: H\to H'$ given by $\rin{H,H'}{k}{R(h)} := \rin{H}{k}{h}$. One can make $H'$ into a Hilbert space by pushing forward the structure of $H$ via $R$. Since $R$ is invertible, we can always define the structure of $H'$ through $R$. Scalar multiplication in $H'$ is defined via $\lambda R(h) := R(\overline{\lambda} h)$, while the inner product is defined via conjugation
\begin{align*}
    \rin{H'}{R(h)}{R(k)} = \rin{H}{k}{h}. 
\end{align*}
This makes $H'$ into a Hilbert space and  $R$ into an anti-unitary isomorphism because it flips the inner product structure. For the rest of this article, we identify $H'$ with the image $\op{Im}(R)$, this gives use the advantage of seeing the duality brackets $\rin{H,H'}{\cdot}{\cdot}:H\times H'\to \mathbb{C}$ as a true bilinear map since $$\rin{H,H'}{k}{\lambda R(h)}:= \rin{H,H'}{k}{R(\overline{\lambda} h)}= \rin{H}{k}{\overline{\lambda} h}=\lambda \rin{H}{k}{h} = \lambda \rin{H,H'}{k}{R(h)}.$$

On the other hand, given an $A-B$ equivalence bimodule $Z$, let us discuss its continuous dual $Z'$. One can equip $Z'$ with a natural $B-A$ bimodule structure
\begin{align*}
    (\xi\cdot a)(z) := \xi(a\cdot z) \\
    (b\cdot \xi)(z) := \xi(z\cdot b)
\end{align*}
for $\xi\in Z'$. This definition makes it immediately clear that $Z'$ is a Banach $B-A$ module since $|(\xi \cdot a) (z)|= |\xi(a\cdot z)|\leq \|\xi\| \|a\cdot z\|\leq \|\xi\|\|a\|\|z\|.$ An analogous estimate holds for the left $B$-module action. Indeed, this is a basic result of Banach module theory \cite{Ka81}. The continuous dual $Z'$ generally loses much of the Hilbert C*-module properties of $Z$, as it is not obvious that one can define a natural C*-algebra valued inner products on it. In contrast, there is a natural notion of the dual module \cite{RaWi98} $\widetilde{Z}$ for an $A-B$ equivalence bimodule $Z$ that results in a $B-A$ equivalence bimodule. However, we shall be working with finitely generated projective Hilbert C*-modules, and this property imposes the Banach space self-duality $Z\cong \widetilde{Z}$. In  interpreting $Z$ as a function space, we avoid working with the module dual as we want $Z$ to be a test function space with ``nontrivial distributions''. This is the first step towards this idea: we shall show that in the setting of Definition \ref{def: faithful-localization}, Hilbert C*-modules generate Banach-Gelfand triples.
\begin{lemma}\label{lem: localizable-is-gelfand-triple}
    Let $Z$ be a faithfully localisable $A-B$ equivalence bimodule via trace $\op{tr}_B:B\to \mathbb{C}$. Then there are the norm-continuous embeddings
    \begin{align}
        Z \hookrightarrow \op{Loc}(Z) \hookrightarrow Z'
    \end{align}
    with the following estimate
    \begin{align}
        \|h\|_{Z'} \leq \sqrt{\|\op{tr}_B\|} \|h\|_{\op{Loc}(Z)}.
    \end{align}    
    Furthermore, under the embedding above, $\op{Loc}(Z)$ is weak-$*$ dense in $Z'$.
\end{lemma}
\begin{proof}    
    If $\iota: Z\to \op{Loc}(Z)$ is the norm-continuous embedding given by the localisation scheme in Theorem \ref{thm: localization}, with Banach space adjoint $\iota': \op{Loc}(Z)'\to Z'$, we define $\iota^*:\op{Loc}(Z)\to Z'$ via $\iota^*:=\iota'\circ R$, with $R:\op{Loc}(Z)\to \op{Loc}(Z)'$ being the Riesz-representation map. Following the definition, we have 
    \begin{align*}
        \rin{Z,Z'}{z}{\iota^*(h)}:=\rin{\op{Loc}(Z),\op{Loc}(Z)'}{\iota(z)}{R(h)}, \qquad \forall h\in \op{Loc}(Z), z\in Z.
    \end{align*}
    From the Cauchy-Schwarz inequality and that of \eqref{form: loc-estimate} it follows that $$|\rin{Z,Z'}{z}{\iota^*(h)}|\leq \|z\|_{\op{Loc(Z)}} \|h\|_{\op{Loc}(Z)}\leq \sqrt{\|\op{tr}_B\|} \cdot \|z\|_Z\cdot \|h\|_{\op{Loc}(Z)}.$$ By taking the supremum over all $\|z\|_Z=1$ results in the desired estimate $\|h\|_{Z'}:= \|\iota^*(h)\|_{Z'}\leq \sqrt{\|\op{tr}_B\|}\cdot \|h\|_{\op{Loc}(Z)}.$ 

    Since $\iota$ is injective, it follows from Theorem \ref{thm: embedding-dense-range} that the Banach space adjoint $\iota': \op{Loc}(Z)' \to Z'$ has a weak-$*$ dense range in $Z'$. Since $R$ implements the anti-linear self-duality of $\op{Loc}(Z)$, $\iota^*$ also has a weak-$*$ dense range in $Z'$. The last assertion follows from the basic fact that norm-density implies weak-* density. Therefore $Z$ is also weak-$*$ densely embedded in $Z'$. 
\end{proof}

One observes that the embedding of $Z$ in $\op{Loc}(Z)$ bridges a natural identification of $Z$ with its dual space $Z'$. Indeed, any Banach space $X$ that is norm-densely embedded in a Hilbert space $H$ generates a Banach-Gelfand triple $(X,H,X')$. However, Lemma \ref{lem: localizable-is-gelfand-triple} gives us an explicit norm-estimate involving the trace $\op{tr}_B$ that has a natural interpretation in time-frequency analysis (see for example the paragraph preceding Theorem \ref{thm: localization-heis}). If $Z$ can be constructed by a pre-equivalence bimodule $\mathcal{Z}$ that is also embedded in it as a Banach space, then we can strengthen Lemma \ref{lem: localizable-is-gelfand-triple} and show that $\mathcal{Z}$ automatically generates a Gelfand triple. Indeed, our setting generates \emph{Gelfand quintuples} in the following sense.
\begin{theorem}\label{thm: gelfand-quintuple}
Let $A_0$ and $B_0$ be dense $*$-subalgebras of $A$ and $B$ and $Z$ an $A-B$ bimodule that is faithfully localisable from an $A_0-B_0$ pre-equivalence bimodule $\mathcal{Z}$ with trace $\op{tr}_B: B\to \mathbb{C}$. Furthermore, let $\mathcal{Z}$ be a Banach space whose topology is finer than the one induced by its $A_0-B_0$ pre-equivalence bimodule structure in the following sense: there exists such a $C>0$ that 
\begin{align}\label{form: pre-equiv-est}
    \|\lin{A_0}{\mathscr{z}}{\mathscr{z}}\|_A^{1/2}=: \|\mathscr{z}\|_Z \leq C\|\mathscr{z}\|_{\mathcal{Z}}, \qquad \mathscr{z}\in \mathcal{Z}.
\end{align}
Then there are the embeddings
\begin{align}\label{bgquint-embeddings}
    \mathcal{Z}\hookrightarrow Z\hookrightarrow \op{Loc}(Z)\hookrightarrow Z' \hookrightarrow \mathcal{Z}'
\end{align}
The first two embeddings have norm dense ranges, while the last two have weak-$*$ dense ranges. Furthermore, we necessarily have the following estimate for the last embedding
\begin{align}\label{form: pre-equiv-dual-est}
    \|z'\|_{\mathcal{Z}'}\leq C \|z'\|_{Z'}, \qquad \forall z' \in Z'.
\end{align}
\end{theorem}
\begin{proof}
    The two embeddings in the middle come from Lemma \ref{lem: localizable-is-gelfand-triple}. The embedding $\mathscr{i}:\mathscr{Z}\hookrightarrow Z$ has a norm-dense range since $\mathscr{Z}$ is a pre-equivalence bimodule for $Z$. Therefore its adjoint $\mathscr{j}': Z'\hookrightarrow \mathscr{Z}'$ is an embedding with weak-$*$ dense range. For $z'\in Z'$, $$\|z'\|_{\mathscr{Z}'}=\sup_{\|\mathscr{z}\|_{\mathcal{Z}}\leq 1}|\rin{Z,Z'}{\mathscr{z}}{z'}|\leq \|z'\|_{Z'} \sup_{\|\mathscr{z}\|_{\mathcal{Z}}\leq1} \|\mathscr{z}\|_{Z}\leq C \|z\|_{Z'},$$ which gives us the required norm estimate.
\end{proof}

Before we close this section, we introduce a natural notion of the morphism between Hilbert C*-modules whose coefficient C*-algebras are not necessarily the same, but are related by C*-isomorphisms on C*-subalgebras.
\begin{definition}
    For C*-algebras $A$ and $C$, let $Z$ be a Hilbert $A$-module, $W$ a Hilbert $C$-module, and $\tau: A\to C$ a C*-isomorphism between C*-algebra $A$ and $C$. A map $T: Z\to W$ is said to be $\tau$-\emph{adjointable} if there exists a map $T^*: W\to Z$, which we call a $\tau$-\emph{adjoint} of $T$, satisfying
    \begin{align}\label{form: tau-adjointable}
        \lin{C}{Tz}{w} = \tau\left(\lin{A}{z}{T^*w} \right)
    \end{align}
    for $z\in Z$ and $w\in W$. We call the map $T^*$ the $\tau$\emph{-adjoint of} $T$. We denote the set of all $\tau$-adjointable maps from $Z$ to $W$ by $\mathcal{L}_{\tau}(Z,W).$
\end{definition}
For the rest of this section, unless otherwise stated, we assume that $\tau: A\to C$ is a C*-isomorphism whenever $\tau$ is invoked. The following result shows that $\tau$-adjointable maps have the same properties as the adjointable maps.
\begin{proposition}\label{prop: adjointable-auto-properties}
    Let $T: Z\to W$ be a $\tau$-adjointable map. Then
    \begin{enumerate}
        \item $T$ is linear;
        \item $T$ is $\tau$\emph{-linear} in the following sense: $T(a\cdot z) = \tau(a)\cdot T(z)$ for all $a\in A$ and $z\in Z$;
        \item $T$ is bounded;
        \item $T$ has a unique $\tau$-adjoint $T^*$, and $\|T\|= \|T^*\|$;
        \item $T^*\in \mathcal{L}_{\tau^{-1}}(W,Z)$.
    \end{enumerate}
\end{proposition}
\begin{proof}
We shall prove $\tau$-linearity first, the proof for ordinary linearity being similar. For $a\in A$, $z\in Z$ and $w\in W$,
\begin{align*}
    \lin{C}{T(a\cdot z)}{w} =\tau\left( a\cdot \lin{A}{z}{T^*w}\right) = \tau(a) \cdot  \lin{C}{Tz}{w} = \lin{C}{\tau(a)\cdot Tz}{w}.
\end{align*}
Therefore $\lin{C}{T(a\cdot z)- \tau(a)\cdot T(z)}{w}=0$ for arbitrary $w\in W$ and so $T(a\cdot z)= \tau(a) \cdot T(z).$ Showing that the $\tau$-adjoint is unique is only slightly different. Suppose $T_1^*$ and $T_2^*$ are both $\tau$-adjoints. Then
\begin{align*}
    \tau\left(\lin{A}{z}{T_1^*w -T_2^*w} \right) =0,
\end{align*}
for $z\in Z$ and $w\in W$. Since $\tau$ is injective and $z$ is arbitrary, $T_1^*w=T_2^*w$ for all $w\in W$. To show boundedness and the equality of the operator norms we let $B_1(Z)$ be the unit sphere in $Z$. We define the maps $\{T_z: W\to C: z\in B_1(Z)\}$ as follows
\begin{align*}
    T_z(w):=\lin{C}{w}{Tz}, \qquad w\in W.
\end{align*}
Since $\tau$ is an isometry of $C^*$-algebras, $\sup_{z\in B_1(Z)}\|T_z(w)\| = \sup_{z\in B_1(Z)}\|\tau\left(\lin{A}{T^*w}{z} \right)\| = \|T^*w\|<\infty$ for any $w\in W$. From the Banach-Steinhaus Theorem we infer that
\begin{align*}
    \infty>\sup_{z\in B_1(Z)}\|T_z\|_{W\to C}= \sup_{z\in B_1(Z)} \sup_{w\in B_1(W)}\|\lin{C}{Tz}{w} \|= \sup_{z\in B_1(Z)}\|Tz\| = \|T\|.
\end{align*}
Hence we have shown the boundedness of $T$ and the same computation shows that $\|T\|=\|T^*\|.$ The last assertion follows from the definition and the fact that $\tau$ is invertible.
\end{proof}
We now give a generalization of unitarity for $\tau$-adjointable maps.
\begin{definition}\label{def: tau-preserving}
    Suppose $\tau: A\to C$ is a $C^*$-isomorphism. A map $T: Z\to W$ is said to be $\tau$-\emph{preserving} if
    \begin{align*}
        \lin{C}{Tz}{Tz} = \tau\left(\lin{A}{z}{z}\right)
    \end{align*}
    for $z\in Z$. Furthermore, if $T$ is also surjective, then we say that $T$ is a $\tau$-\emph{unitary} map. We say that $Z$ and $W$ are $\tau$\emph{-unitarily equivalent} if we can find a $\tau$-unitary map $T:Z\to W$.
\end{definition}
\begin{corollary}\label{cor: unitary-is-adjointable}
    Let $T: Z\to W$ be a $\tau$-unitary map, then $T$ is an isometric $\tau$-adjointable map with $T^*=T^{-1}$. 
\end{corollary}
\begin{proof}
    It follows from the $\tau$-preserving property of $T$ that it is isometric. From the polarization identity in $Z$ and $W$, along with the $\tau$-preserving property on $T$, it follows that $\lin{C}{Tz_1}{Tz_2}= \tau\left(\lin{A}{z_1}{z_2} \right)$ for $z_1,z_2\in Z$. Since $T$ is also surjective, it is invertible. We let $T^{-1}w=z_2$ for $w\in W$ and find that $\lin{C}{Tz_1}{w} = \lin{C}{Tz_1}{Tz_2} = \tau\left(\lin{A}{z_1}{z_2}\right) = \tau\left(\lin{A}{z_1}{T^{-1}(w)}\right).$ It follows that the $\tau$-adjoint exists and, by its uniqueness, it must satisfy $T^{*}=T^{-1}. $
\end{proof}
\begin{remark}
     If $Z$ and $W$ are Hilbert C*-modules and $\tau$-unitarily equivalent, say, via $T:Z\to W$, this gives us the expected isomorphism between Hilbert C*-modules with isomorphic C*-algebra coefficients. To reiterate, $T$ is a bijection that satisfies
    \begin{align*}
        \lin{C}{T z_1}{Tz_2} &= \tau(\lin{A}{z_1}{z_2}), \\
        T(a\cdot z) &= \tau(a) \cdot T(z), \\
        \lin{A}{T^*w_1}{T^*w_2} &= \tau^{-1}(\lin{C}{w_1}{w_2}), \\
        T^*(c\cdot w) &= \tau^{-1}(c)\cdot T^*(w).
    \end{align*}
    for all $a\in A$, $c\in C$, $z,z_1,z_2\in Z$ and $w,w_1,w_2\in W$. 
\end{remark}
\begin{proposition}\label{prop: send-gen-to-gen}
    Let $Z$ and $W$ be full Hilbert C*-modules and $T: Z\to W$ a $\tau$-unitary map. If $Z$ is finitely generated and projective, then so is $W$ and $T$ transforms generators to generators.
\end{proposition}
\begin{proof}
If $Z$ is finitely generated and projective as a Hilbert $A$-module, then equivalently $\mathcal{K}_A(Z)$ is unital and, according to Proposition \ref{prop: parseval-module-frame}, there are such $z_1,...,z_n\in Z$ that $\displaystyle z = \sum_{i=1}^n \lin{A}{z}{z_i}\cdot z_i$ for all $z\in Z.$ Since $T$ is $\tau$-unitary, for any $w\in W$, $$w = T(T^*(w)) = T\left(\sum_{i=1}^n \lin{A}{T^*(w)}{z_i}\cdot z_i \right) = \sum_{i=1}^n \tau(\lin{C}{T^*(w)}{z_i})\cdot T(z_i)= \sum_{i=1}^n \lin{C}{w}{T(z_i)}\cdot T(z_i).$$
Therefore $\displaystyle \sum_{i=1}^n \rin{\mathcal{K}_C(W)}{T(z_i)}{T(z_i)}$ acts as a unit to the right of $W$ and, since $W$ is full, it must be finitely generated and projective as a Hilbert $C$-module. The last assertion follows from the fact that $T$ is $\tau$-linear and so it must transform generators to generators. 
\end{proof}
\section{Tensor Products of Banach spaces, C*-algebras and Hilbert C*-modules}\label{sec: tensors} 

The algebraic tensor product $V_1\odot V_2$ of two vector spaces $V_1$ and $V_2$ over the the field of complex numbers is the complex linear span of the elementary tensors $v_1\otimes v_2$ where $v_1\in V_1$ and $v_2\in V_2$. If $V_i$ and $W_i$ are vector spaces and $f_i:V_i\to C_i$ is linear map where $i=1,2$, $f_1\otimes f_2 : V_1\odot V_2 \to W_1\odot W_2$ is the linear map defined by $f_1\otimes f_2 (v_1\otimes v_2) = f_1(v_1)\otimes f_2(v_2)$ on the elementary tensor products. Where appropriate, we shall use the isomorphisms of the tensor products. For example, if $W_i=\mathbb{C}$ for $i=1,2$, $(f_1\otimes f_2)(v_1\otimes v_2)=f_1(v_1)f_2(v_2)$, since $\mathbb{C}\odot \mathbb{C} \cong \mathbb{C}$. The same idea can also be used when the underlying vector spaces are completed with respect to suitable norms. The tensor product  $H_1\otimes_{2} H_2$
of two Hilbert spaces $H_1$ and $H_2$ is the completion of the algebraic tensor product $H_1\odot H_2$ with respect to the norm induced by the inner product
\begin{align*}
    \rin{2}{h_1 \otimes h_2}{h_1' \otimes h_2'} := \rin{H_1}{h_1}{h_1'} \cdot \rin{H_2}{h_2}{h_2'}, \qquad h_1,h_1'\in H_1,\ h_2,h_2'\in H_2 .
\end{align*}
 $H_1\otimes_{2} H_2$ is a Hilbert space too, which can be identified with the Hilbert-Schmidt operators $\mathcal{HS}(H_1', H_2)$. The following lemma  will be used on many occasions. 
\begin{lemma}\label{lem: dense-tensor}
    For $i=1,2$, let $\mathcal{X}_i$ be a dense subspace of the Banach space $X_i$ and $\|\cdot\|_{t}$ a norm on the algebraic tensor product $X_1\odot X_2$ satisfying the sub-cross property 
\begin{align}\label{form: cross-norm}
        \|x_1\otimes x_2\|_t \leq \|x_1\|_{X_1}\cdot \|x_2\|_{X_2}
    \end{align} 
    for all $x_1\in X_1$ and $x_2\in X_2$. If $X_1\otimes_{t}X_2$ is the completion of $X_1\odot X_2$ with respect to norm $\|\cdot\|_{t}$,  then $\mathcal{X}_1\odot \mathcal{X}_2$ is a dense subspace of $X_1\otimes_{t}X_2.$
\end{lemma}
\begin{proof}
    It suffices to show that any finite sum of elementary tensors from $X_1\odot X_2$ can be approximated arbitrarily well by a finite sum of elementary tensors from $\mathcal{X}_1\odot \mathcal{X}_2$.  We consider $h_1^{(1)},...,h_1^{(n)}\in X_1$ and $X_2^{(1)},...,h_2^{(n)}\in X_2$ and, for an $\varepsilon>0$, choose, for each $i=1,...,n$, such a $k_1^{(i)}\in \mathcal{X}_1$ that
    \begin{align*}
        \left\|h_1^{(i)}-k_1^{(i)}\right\|_{X_1}<\frac{\varepsilon}{\left\|h_2^{(i)} \right\|_{X_2}} \cdot \frac{1}{n(n+1)}.
    \end{align*} and such a $k_2^{(i)}\in \mathcal{X}_2$ that 
    \begin{align*}
        \left\|h_2^{(i)}-k_2^{(i)}\right\|_{X_2} <\frac{\varepsilon}{\left\|k_1^{(i)} \right\|_{X_1}}\cdot \frac{1}{n(n+1)}.
    \end{align*}
   From the the sub-cross property we infer that \eqref{form: cross-norm} of $\|\cdot\|_t$
    \begin{align*}
        \left\|\sum_{i=1}^n h_1^{(i)}\otimes h_2^{(i)}   - \sum_{i=1}^n k_1^{(i)}\otimes k_2^{(2)} \right\|_{t} &= \left\|\sum_{i=1}^n \left( \left(h_1^{(i)}-k_1^{(i)}\right)\otimes h_2^{(i)} + k_1^{(i)}\otimes \left(h_2^{(i)}-k_2^{(i)} \right)  \right)  \right\|_{t} \\
        & \leq \sum_{i=1}^n \left(\left\|h_1^{(i)}-k_1^{(i)}\right\|_{X_1}\left\|h_2^{(i)}\right\|_{X_2} + \left\|h_2^{(i)}-k_2^{(i)}\right\|_{X_2}\left\|k_1^{(i)}\right\|_{X_1} \right) \\
        &< \sum_{i=1}^n \varepsilon \frac{2}{n(n+1)} = \varepsilon .
     \end{align*}
\end{proof}

For Banach spaces $X$ and $Y$ and any $\mathbf{u}$ in $X\odot Y$,
\begin{align}\label{form: projective-tensor-norm}
    \|\mathbf{u}\|_{\pi} :=
    \|\mathbf{u}\|_{\pi(X,Y)}
    :=
\op{inf}\left\{\sum_{i=1}^N \|x_i\|_X\cdot \|y_i\|_Y: \mathbf{u}=\sum_{i=1}^N x_i\otimes y_i  \right\}
\end{align} is a norm on $X\odot Y$, called the \emph{projective tensor norm}. The \emph{projective tensor product} $X\otimes_{\pi} Y$ of two Banach spaces $X$ and $Y$ is the completion of $X\odot Y$ with respect to the projective tensor norm $\|\cdot\|_{\pi}$ \cite{DeFl92, Ry02}. In other words, any $\mathbf{u}\in X\otimes_{\pi}Y$ is of the form 
\begin{align}\label{form: general-projective-tensor}
 \mathbf{u} = \sum_{i=1}^{\infty} x_i\otimes y_i   
\end{align}
where $\displaystyle \sum_{i=1}^{\infty}\|x_i\|_X\cdot \|y_i\|_Y <\infty$ and the sequences $(x_i)_i$ and $(y_i)_i$ are bounded in $X$ and $Y$ respectively.  Very useful will be the following Banach space isomorphism
\begin{align}\label{form: duality-for-projective-tensor}
    (X\otimes_{\pi}Y)' \cong \mathcal{B}(X,Y').
\end{align}
In general, the projective tensor product behaves quite badly with respect to embeddings, as it is hard to figure out when a representation \eqref{form: general-projective-tensor} equals zero in $X\otimes_{\pi} Y$. Indeed, as we shall see in the sequel that most technical difficulties in the interpretation of equivalence bimodules as Banach spaces with kernel theorems lie in embedding projective tensor products in the external tensor product for Hilbert C*-modules. However, in the following special case, the canonical isometry $X \hookrightarrow X^{''}$ induces a true embedding in the projective tensor product.
\begin{proposition}[{\cite[Corollary 2.12]{Ry02}}]\label{prop: canonical-isom-to-proj}
    Let $X$ and $Y$ be Banach spaces, the canonical isometry $X\hookrightarrow X^{''}$ induces an isometry $X\otimes_{\pi} Y \hookrightarrow X''\otimes_{\pi} Y.$
\end{proposition}
We shall also be very much concerned with the space of \emph{nuclear operators}.
\begin{definition}
    For Banach spaces $X$ and $Y$, a linear operator $T:X\to Y$ is said to be \emph{nuclear} if there are such sequences $\{x'_i\}_{i-1}^{\infty}\subseteq X'$ and $\{y_i\}_{i=1}^{\infty}\subseteq Y$ that $\displaystyle \sum_{i=1}^{\infty}\|x_i'\|_{X'}\cdot \|y_i\|_{Y}<\infty$ and 
\begin{align}\label{form: nuclear-representation}
        T(x) :=\sum_{i=1}^{\infty}\theta_{x_i',y_i}(x) =\sum_{i=1}^{\infty}\rin{X,X'}{x}{x'_i} y_i.
    \end{align}
We denote the space of all such nuclear operators by $\mathcal{N}(X,Y)$.
\end{definition}
\noindent For any nuclear operator $T$, the map $$\|T\|_{\mathcal{N}}:= \op{inf}\left\{\sum_{i=1}^{\infty} \|x_i'\|_{X'}\cdot \|y_i\|_{Y} :T=\sum_{i=1}^{\infty}\theta_{x_i',y_i} \right\}$$ is a norm, called  the \emph{nuclear operator norm}. $\mathcal{N}(X,Y)$ is a Banach space with respect to this norm. It is clear from the definition that $T\in \mathcal{B}(X,Y)$ and $\|T\|\leq \|T\|_{\mathcal{N}}$ and so $T\in \mathcal{A}(X,Y)$. There is the natural surjection $J:X'\otimes_{\pi} Y \to \mathcal{N}(X,Y)$ transforming $\displaystyle \mathbf{u} = \sum_{i=1}^{\infty}x_i' \otimes y_i$ into $\displaystyle J(\mathbf{u})=\sum_{i=1}^{\infty} \theta_{x_i',y_i}$. 
One might deduce that $J$ is an isometry and so an isometric isomorphism of Banach spaces. However, a closer look at the norms $\|\mathbf{u}\|_{\pi}$ and $\|J(\mathbf{u})\|_{\mathcal{N}}$ reveals that the infima run over different ranges. The isometry property of $J$ follows from injectivity of $J$ as this will make the set on which we take the respective infima for the norms equal. For this, one needs the so-called \emph{approximation property} either for $X'$ or $Y$. We state it as a proposition and it will allow us later on to inject the projective tensor product successfully to other tensor product constructions.
\begin{proposition}[{\cite[Proposition 4.6]{Ry02}}]\label{prop: the-ap}
    Let $X$ be a Banach space. Then following are equivalent.
    \begin{enumerate}
        \item \emph{\textbf{(Approximation Property).}} For any compact subset $K$ of $X$ and any $\varepsilon>0$, there is such a finite-rank operator $T:X\to X$ that $\|x-Sx\|_X<\varepsilon$ for any $x\in K$;
        \item for any Banach space $Y$, $u=0$ in $X\otimes_{\pi} Y$, if $\displaystyle u= \sum_{i=1}^{\infty} x_i\otimes y_i \in X\otimes_{\pi} Y$ is such that $\displaystyle \sum_{i=1}^{\infty}\rin{X,X'}{x_i}{x'}\cdot y_i=0$ in $Y$ for any $x'\in X_0'$ where $X_0'$ stands any separating subset of $X'$;
        \item for any Banach space $Y$, $u=0$ in $X\otimes_{\pi} Y$, if $\displaystyle u= \sum_{i=1}^{\infty} x_i\otimes y_i \in X\otimes_{\pi} Y$ is such that $\displaystyle \sum_{i=1}^{\infty}\rin{Y,Y'}{y_i}{y'}\cdot x_i=0$ in $X$ for any $y'\in Y_0'$ where $Y_0'$ stands any separating subset of $Y'$. 
    \end{enumerate}
\end{proposition}
\begin{corollary}\label{cor: projective-nuclear}
Let $X,Y$ be Banach spaces. If $X'$ or $Y$ has the approximation property, then the map $J: X'\otimes_{\pi}Y\to \mathcal{N}(X,Y)$ is an isometric isomorphism of Banach spaces.    
\end{corollary}
\begin{proof}
    Let $X'$ have the approximation property. From Proposition \ref{prop: the-ap} Item (2) and the fact that $X$ is isometrically embedded in $X''$ as a separating subset we infer that,     
    for any $\displaystyle \mathbf{u}= \sum_{i=1}^{\infty} x_i'\otimes y_i\in X'\otimes_{\pi}Y$, 
    \begin{align*}
        J(\mathbf{u})(x) = \sum_{i=1}^{\infty}\rin{X',X''}{x_i'}{x} \cdot y_i = \sum_{i=1}^{\infty}\rin{X,X'}{x}{x'} \cdot y_i=0 ,\qquad x\in X
    \end{align*} implies
    $\mathbf{u}=0.$ This shows that $J$ is injective and so  an isometric isomorphism. If, on the other hand, $Y$ has the approximation property, we use Item (3) of Proposition \ref{prop: the-ap}, along with the easily verifiable fact that $X\otimes_{\pi}Y \cong Y\otimes_{\pi}X$. 
\end{proof}
    In this paper, we shall be concerned with the space of nuclear operators $\mathcal{N}(X',Y)$. As we shall see later on, it can be viewed as the operator space reservoir for the projective tensor product, similar to the Hilbert space tensor product being identified with the Hilbert-Schmidt operators.
\begin{corollary}\label{cor: proj-goes-into-nuclear}
     If either $X''$ or $Y$ has the approximation property, then there is the isometric embedding $k: X\otimes_{\pi} Y\hookrightarrow \mathcal{N}(X',Y).$
\end{corollary}
\begin{proof}
   Due to Proposition \ref{prop: canonical-isom-to-proj}, there is the isometry $X\otimes_{\pi} Y \hookrightarrow X''\otimes_{\pi} Y$. If either $X''$ or $Y$ has the approximation property, then from Corollary \ref{cor: projective-nuclear} follows that $J: X''\otimes_{\pi} Y\to \mathcal{N}(X',Y)$ is an isometric isomorphism. We define $k$ to be the composition $X\otimes_{\pi} Y \hookrightarrow X''\otimes_{\pi} Y \overset{J}{\hookrightarrow} \mathcal{N}(X',Y)$ and it is an isometric embedding.
\end{proof}
Since $\mathcal{N}(X',Y)\subseteq \mathcal{B}(X',Y)$ and $\mathcal{N}(Y',X)\subseteq \mathcal{B}(Y',X)$, there are the natural maps $i_X: X\otimes_{\pi} Y \to \mathcal{B}(X',Y)$ and $i_Y: X\otimes_{\pi}Y\to \mathcal{B}(Y',X)$. The following corollary shows that the approximation property implies that these maps are injective and that the projective tensor product is embedded in many natural Banach spaces.
\begin{corollary}\label{cor: many-injections}
    Let $X$ and $Y$ be Banach spaces and either $X$ or $Y$ have the approximation property. Then both $i_X: X\otimes_{\pi} Y \to \mathcal{B}(X',Y)$ and $i_Y: X\otimes_{\pi}Y \to \mathcal{B}(Y',X)$ are injective. Furthermore, any bounded linear operator that factors through these maps are injective too. In other words, we have the following commutative diagram
    \[
    \begin{tikzcd}[row sep=large, column sep=large]
     & Q_X \arrow[d] & X\otimes_{\pi}Y \arrow[l, hook', "T_X"'] \arrow[ld, hook', "i_X"] \arrow[r, hook, "T_Y"] \arrow[rd, hook, "i_Y"] & Q_Y \arrow[d] \\
        &\mathcal{B}(X',Y) & &\mathcal{B}(Y',X).
    \end{tikzcd}
    \]
    for any Banach spaces $Q_X$, $Q_Y$ and bounded linear operators $T_X: X\otimes_{\pi} Y \to Q_X$ and $T_Y: X\otimes_{\pi} Y \to \mathcal{B}(Y',X)$.
\end{corollary}
\begin{proof}
    If $\displaystyle \mathbf{u} = \sum_{i=1}^{\infty} x_i\otimes y_i\in X\otimes_{\pi}Y$, then 
    \begin{align*}
        i_X(\mathbf{u})(x') &=\sum_{i=1}^{\infty}\rin{X,X'}{x_i}{x'} y_i, \qquad \forall x'\in X' \\
        i_Y(\mathbf{u})(y') &= \sum_{i=1}^{\infty}\rin{Y,Y'}{y_i}{y'} x_i, \qquad \forall y'\in Y'.
    \end{align*}
    Let $X$ have the approximation property. Then Proposition \ref{prop: the-ap} Item (2) implies that $i_X$ is injective, while Item (3) implies that $i_Y$ is injective. The same can be said if $Y$ has the approximation property on account of $X\otimes_{\pi} Y\cong Y\otimes_{\pi} X.$ Now let $T_X: X\otimes_{\pi} Y \to Q_X$ be such a bounded linear operator that there exists another bounded linear operator $q_X: Q_X\to \mathcal{B}(X',Y)$ where $i_X= q_X \circ T_X$. Since $i_X$ is injective, $T_x$ must be injective too. The same applies to a bounded linear operator $T_Y: X\otimes_{\pi} Y \to Q_Y$ factoring through $i_Y$.
\end{proof}
In the light of this corollary and for the sake of simplicity, we shall often be inclined to assume, when necessary, that the first factor has the approximation property. 
\begin{proposition}\label{prop: proj-into-proj}

For $i=1,2$, let $\mathcal{X}_i$ be Banach spaces and the embeddings  $\iota_i: \mathcal{X}_i \hookrightarrow X_i$ have dense ranges satisfying the conditions
    \[
    \|\iota_i(\mathscr{x}_i)\|_{X_i}\leq C_i \|\mathscr{x}_i\|_{\mathcal{X}_i}, \qquad \forall \mathscr{x}_i\in \mathcal{X}_i.
    \]
    for some $C_i>0$. If $\mathcal{X}_1$ has the approximation property, the tensor product of the embeddings can be extended to the embedding $\iota=\iota_1\otimes \iota_2: \mathcal{X}_1\otimes_{\pi}\mathcal{X}_2\hookrightarrow X_1\otimes_{\pi} X_2$ with a dense range satisfying the condition $\|\iota(\mathbf{u})\|_{\pi}\leq C_1 C_2 \|\mathbf{u}\|_{\pi}$ for any $\mathbf{u}\in \mathcal{X}_1\otimes_{\pi}\mathcal{X}_2$.
\end{proposition}
\begin{proof}
    Let $\iota:=\iota_1\otimes \iota_2: \mathcal{X}_1\odot\mathcal{X}_2 \to X_1\odot X_2$ be the tensor product of the embeddings. For the sake of clarity, let us for the time being denote the norm on $\mathcal{X}_1\otimes_{\pi}\mathcal{X}_2$ by $\|\cdot\|_{\pi(\mathcal{X}_1,\mathcal{X}_2)}$ and similarly $\|\cdot\|_{\pi(X_1,X_2)}$ for $X_1\otimes_{\pi} X_2.$ If $\displaystyle \mathbf{u} = \sum_{i=1}^n \mathscr{x}_1^{(i)}\otimes \mathscr{x}_2^{(i)}\in \mathcal{X}_1\otimes_{\pi} \mathcal{X}_2$, 
    \begin{align*}
        \|\iota(\mathbf{u})\|_{\pi(X_1,X_2)} \leq \sum_{i=1}^n \|\iota_1(\mathscr{x}_1^{(i)})\|_{X_1}\cdot \|\iota_2(\mathscr{x}_2^{(i)})\|_{X_2} \leq C_1C_2 \sum_{i=1}^n \|\mathscr{x}_1^{(i)}\|_{\mathcal{X}_1}\cdot \|\mathscr{x}_2^{(i)}\|_{\mathcal{X}_2}.
    \end{align*}
    From definition of the projective tensor product norm $\pi(X_1,X_2)$, we infer that $\|\iota(\mathbf{u})\|_{\pi(X_1,X_2)}\leq C_1 C_2\|\mathbf{u}\|_{\pi(\mathcal{X}_1,\mathcal{X}_2)}$. Therefore the natural inclusion map $\iota$ can be extended to the bounded linear functional $\iota: \mathcal{X}_1\otimes_{\pi} \mathcal{X}_2\to X_1\otimes_{\pi} X_2$. From Lemma \ref{lem: dense-tensor} it follows that $\iota$ has a dense range. It remains to show that it is injective. Since $\iota_i: \mathcal{X}_i\hookrightarrow X_i$ is an embedding with norm dense range, $\iota_i': X_i'\hookrightarrow \mathcal{X}_i'$ is an embedding with weak-$*$ dense range (see Theorem \ref{thm: embedding-dense-range}). Now, let $\displaystyle \mathbf{u}= \sum_{i=1}^{\infty}\mathscr{x}_1^{(i)}\otimes\mathscr{x}_2^{(i)}\in \mathcal{X}_1\otimes_{\pi}\mathcal{X}_2$ be such that $\iota(\mathbf{u})=0 \in X_1\otimes_{\pi} X_2$ and consider the map $i_{X_1}: X_1\otimes_{\pi} X_2\to \mathcal{B}(X_1',X_2)$. $i_{X_1}(\iota(\mathbf{u}))=0$. We note that, for any $x_1'\in X_1'$,
    \begin{align*}
        \iota_2( i_{\mathcal{X}_1}(\mathbf{u})(\iota_1'(x_1'))) = i_{X_1}(\iota(\mathbf{u}))(x_1')=0.
    \end{align*}
    As $\iota_2$ is an embedding,  $i_{\mathcal{X}_1}(\mathbf{u})(\iota_1'(x_1'))=0$ for any $x_1'\in X_1'$. Let $\mathscr{x}_1'\in \mathcal{X}_1'$ and $\{x_{1,\lambda}' \}_{\lambda\in \Lambda}\subseteq X_1'$ be such a net that $\iota_1'(x_{1,\lambda}')\to \mathscr{x}_1'$ in weak-$*$ topology. For an $\mathscr{x}_2'\in \mathcal{X}_2'$ and we note that $\displaystyle \sum_{i=1}^{\infty}\left\|\rin{\mathcal{X}_2,\mathcal{X}_2'}{\mathscr{x}_2^{(i)}}{\mathscr{x}_2'}\cdot \mathscr{x}_1^{(i)} \right\|_{\mathcal{X}_1}\leq \|\mathscr{x}_2'\|_{\mathcal{X}_2}\cdot \sum_{i=1}^{\infty}\|\mathscr{x}_1^{(i)}\|_{\mathcal{X}_1}\cdot \|\mathscr{x}_2^{(i)}\|_{\mathcal{X}_2}<\infty$. This series is absolutely convergent and $\mathcal{X}_1$ is a Banach space. Therefore we can define the element $\displaystyle \sum_{i=1}^{\infty}\rin{\mathcal{X}_2,\mathcal{X}_2'}{\mathscr{x}_2^{(i)}}{\mathscr{x}_2'}\cdot \mathscr{x}_1^{(i)}:=\mathscr{x}_1\in \mathcal{X}_1$ and write
    \begin{align*}    \rin{\mathcal{X}_2,\mathcal{X}_2'}{i_{\mathcal{X}_1}(\mathbf{u})(\mathscr{x}_1')}{\mathscr{x}_2'} &= \sum_{i=1}^{\infty} \rin{\mathcal{X}_1,\mathcal{X}_1'}{\mathscr{x}_1^{(i)}}{\mathscr{x}_1'}\cdot \rin{\mathcal{X}_2,\mathcal{X}_2'}{\mathscr{x}_2^{(i)}}{\mathscr{x}_2'} = \rin{\mathcal{X}_1,\mathcal{X}_1'}{\mathcal{x}_1}{\mathcal{x}_1'} \\ &=\lim_{\lambda}\rin{\mathcal{X}_1,\mathcal{X}_1'}{\mathcal{x}_1}{\iota_1(x_{1,\lambda}')} =\lim_{\lambda} \sum_{i=1}^{\infty} \rin{\mathcal{X}_1,\mathcal{X}_1'}{\mathscr{x}_1^{(i)}}{\iota_1'(x_{1,\lambda}')}\cdot \rin{\mathcal{X}_2,\mathcal{X}_2'}{\mathscr{x}_2^{(i)}}{\mathscr{x}_2'}\\
    &= \lim_{\lambda} \rin{\mathcal{X}_2,\mathcal{X}_2'}{i_{\mathcal{X}_1}(\mathbf{u})(\iota_1'(x_{1,\lambda}')) }{\mathscr{x}_2'}=0.
\end{align*}
Since this is true for any $\mathscr{x}_1'\in \mathcal{X}_1'$ and any $\mathscr{x}_2'\in \mathcal{X}_2'$, $i_{\mathcal{X}_1}(\mathbf{u})=0$. However, due to Corollary \ref{cor: many-injections}, $i_{\mathcal{X}_1}$ is injective and so $\mathbf{u}=0$ in $\mathcal{X}_1\otimes_{\pi}\mathcal{X}_2.$
\end{proof}
Let us consider the special case of $X$ and $Y$ being Banach $*$-algebras. It is not hard to see that the projective tensor product $X\otimes_{\pi} Y$ too is a Banach $*$-algebra. Indeed, if $\displaystyle \mathbf{u} = \sum_{i=1}^{\infty}x_i \otimes y_i$ and $\displaystyle \mathbf{v} = \sum_{i=1}^{\infty} w_i \otimes z_i'$ are tensors in $X\otimes_{\pi} Y$, one can define a product and an involution factorwise on $\mathbf{u}$ and $\mathbf{v}$. We also note that
\begin{align*}
    \|\mathbf{u}\cdot \mathbf{v}\|_{\pi} \leq \sum_{i=1}^{\infty}\|x_i\|_X\cdot  \|y_i\|_Y \cdot \sum_{j=1}^{\infty} \|w_j\|_X\cdot \|z_j\|_Y.
\end{align*}
Therefore $\|\mathbf{u}\cdot \mathbf{v}\|_{\pi}\leq \|\mathbf{u}\|_{\pi}\cdot \| \mathbf{v}\|$ and $X\otimes_{\pi}Y$ is a Banach algebra. Finally,
\begin{align*}
    \|\mathbf{v}^*\|_{\pi} &= \op{inf}\left\{\sum_{i=1}^{\infty}\|x_i^*\|_X\cdot \|y_i^*\|_Y: \mathbf{u} = \sum_{i=1}^{\infty}x_i\otimes y_i \right\} \\
    &= \op{inf}\left\{\sum_{i=1}^{\infty}\|x_i\|_X\cdot \|y_i\|_Y: \mathbf{u} = \sum_{i=1}^{\infty}x_i\otimes y_i \right\} = \|\mathbf{u}\|_{\pi}.
\end{align*}
We have shown that $X\otimes_{\pi}Y$ is a Banach $*$-algebra. In general $\|\cdot\|_{\pi}$ is not a C*-norm, though. Therefore $X\otimes_{\pi}Y$ is not a C*-algebra even if both $X$ and $Y$ are. However, one can construct the enveloping C*-algebra $C^*(X\otimes_{\pi}Y)$ \cite[Definition 10.4]{1fell88}, or the universal C*-completion, of the Banach $*$-algebra $X\otimes_{\pi}Y$ (see Proposition \ref{prop: envelope-of-proj} and Remark \ref{rem: envelope-of-proj}).

Clearly the tensor product $A_1\odot A_2$ of two C*-algebras $A_1$ and $A_2$ is a vector space, which, together with the factorwise multiplication and involution, becomes a $*$-algebra. One can define several different C*-norms on $A_1\odot A_2$ and so obtain several different C*-algebras of the tensor products of $A_1$ and $A_2$. However, we shall only consider the \emph{spatial tensor product}. A convenient way for defining the spatial C*-norm is to invoke the fact \cite[Theorem B.9]{RaWi98} that, for positive linear functionals, or states, $\rho_1$ of $A_1$ and $\rho_2$ of $A_2$, $\rho_1\otimes \rho_2 (s^*s)\geq0$ for $s\in A_1\odot A_2$. 
\begin{definition}
    The spatial tensor product, denoted by $A_1\otimes_{\op{sp}}A_2$, is the completion of $A_{1}\odot A_2$ with respect to the norm
    \begin{align}\label{form: spatial-tensor-norm}
        \|t\|_{\op{sp}} = \sup\left\{\frac{(\rho_1\otimes \rho_2)((ts)^*(ts))}{(\rho_1\otimes \rho_2)(s^*s)}:\rho\in S(A_1), \rho_2\in S(A_2),s\in A_1\odot A_2, (\rho_1\otimes \rho_2)(s^*s)\neq 0 \right\}
    \end{align}
    for $t\in A_1\odot A_2$.
\end{definition}
\begin{remark}[A summary of the Properties of the Spatial Tensor Product \cite{Mu90, 2KaRi97, 1KaRi97}]\label{rem: spatial-tensor-props}
    The spatial tensor product norm is a \emph{cross-norm}, i.e. $\|a_1\otimes a_2\|_{\op{sp}}= \|a_1\|_{A_1} \|a_2\|_{A_2}$ for any $a_1\in A_1$ and any $a_2\in A_2$. Furthermore, the spatial tensor product behaves well with respect to different morphisms. If $\pi_1:A_1\to \mathcal{B}(H_1)$, $\pi_2: \mathcal{B}(H_2)$ are faithful $*$-representations, their tensor product can be extended to $\pi_1\otimes \pi_2: A_1\otimes_{\op{sp}} A_2\to \mathcal{B}(H_1\otimes_2 H_2)$, which is also a faithful $*$-representation. Furthermore, if $\rho_1:A_1\to \mathbb{C}$ and $\rho_2:A_2\to \mathbb{C}$ are positive linear functionals or states, their tensor product can be extended to $\rho_1\otimes\rho_2: A_1\otimes_{\op{sp}}A_2\to \mathbb{C}$, a positive linear functional or state. More generally, if $\rho_1$ and $\rho_2$ are bounded linear functionals on $A_1$ and $A_2$ respectively, $\rho_1\otimes \rho_2 \in (A_1\otimes_{\op{sp}}A_2)'$ with $\|\rho_1 \otimes \rho_2\| = \|\rho_1\|\|\rho_2\|$.
\end{remark} 
\begin{remark}\label{rem: cstar-nuclearity}
    As already mentioned, one can define different C*-norms on $A_1 \odot A_2$. However, if $A_1$ or $A_2$ is a \emph{nuclear} C*-algebra, all C*-norms on $A_1\odot A_2$ coincide. Therefore, we shall not miss much by considering only the spatial tensor norm on $A_1\odot A_2$ if the C*-algebras we consider are nuclear. We shall deal only with twisted group C*-algebras over locally compact abelian groups $K$ with the $2$-co-cycle $\eta: K\times K \to \mathbb{T}$. Lance indicated that the \emph{amenable} locally compact groups have nuclear group C*-algebras \cite{La73}. $K$ is abelian and so amenable. Therefore the group C*-algebra $C^*(K)$ is nuclear. Raeburn and Packer showed that $2$-co-cyles do not have any effect on nuclearity \cite[Corollary 3.9]{PaRa89} (this a result that holds for twisted crossed-products in general). Therefore $C^*(H,\eta)$ is a nuclear C*-algebra. We also note that nuclearity is preserved under Morita equivalence \cite{Ze82}. That is, if either $B$ or $A$ is nuclear and there is an $A-B$ equivalence bimodule $Z$, $A$ (or $B$) too is nuclear. Lastly and perhaps more importantly for this work, the nuclear C*-algebras have the approximation property and, indeed, are characterised by the so-called \emph{completely positive approximation property} \cite{ChEf78}.
\end{remark}

Given C*-algebras $A$ and $B$, it is natural to ask what the enveloping C*-algebra $C^*(A\otimes_{\pi}B)$ of the Banach $*$-algebra $A\otimes_{\pi}B$ is. To answer this question, we shall use the following lemma.
\begin{lemma}\label{lem: slice-map-lem}
    Let $A_1$ and $A_2$ be C*-algebras and $a'_1\in A'_1$. Then, for any $\displaystyle \sum_{i=1}^n a^{(i)}_1 \otimes a_2^{(i)}\in A_1\otimes_{\op{sp}}A_2$, the map 
    \begin{align}\label{form: slice-map}
       (a'_1\otimes \op{Id}_{A_2}) \left(\sum_{i=1}^n a^{(i)}_1 \otimes a_2^{(i)}\right):=\sum_{i=1}^n \rin{A_1,A_1'}{a_1^{(i)}}{a'_1}\cdot a_2^{(i)}
    \end{align}
    satisfies:
    \begin{align}\label{form: slice-map-est}
        \left\|\left(a'_1\otimes \op{Id}_{A_2}\right)\left(\sum_{i=1}^na_1^{(i)}\otimes a_2^{(i)}\right) \right\|_{A_2}\leq \|a_1'\|_{A'}\cdot \left\|\sum_{i=1}^n a_1^{(i)}\otimes a_2^{(i)} \right\|_{\op{sp}}
    \end{align}
    Consequently, it can be extended to a continuous linear map  $a_1'\otimes\op{Id_{A_2}}: A_1\otimes_{\op{sp}}A_2\to A_2.$
\end{lemma}
\begin{proof}
    Let $\mathbf{a}=\displaystyle \sum_{i=1}^n a_1^{(i)}\otimes a_2^{(i)}$ and $a_2'\in A_2'$ be a state. Then 
    \begin{align*}
        \rin{A_2,A_2'}{a_2'}{(a'_1\otimes \op{Id}_{A_2}) \left(\sum_{i=1}^n a^{(i)}_1 \otimes a_2^{(i)}\right)} &= \sum_{i=1}^n \rin{A_1,A_1'}{a_1^{(i)}}{a_1'} \rin{A_2,A_2'}{a_2^{(i)}}{a_2'} \\
        &= (a_1'\otimes a_2')(\mathbf{a})
    \end{align*}
    follows from Remark \ref{rem: spatial-tensor-props} that $a_1'\otimes a_2' \in (A_1\otimes_{\op{sp}}A_2)'$ with $\|a_1'\otimes a_2'\|_{(A_1\otimes_{\op{sp}}A_2)'} = \|a_1'\|_{A_1'}$. Therefore
    \begin{align*}
        \left\|\left(a'_1\otimes \op{Id}_{A_2}\right)\left(\sum_{i=1}^na_1^{(i)}\otimes a_2^{(i)}\right) \right\|_{A_2} &=
        \sup_{a_2'\in S(A_2)}\left| \rin{A_2,A_2'}{a_2'}{(a'_1\otimes \op{Id}_{A_2}) \left(\sum_{i=1}^n a^{(i)}_1 \otimes a_2^{(i)}\right)}\right| \\ &= \sup_{a_2'\in S(A_2)} |(a_1'\otimes a_2')(\mathbf{a})| \leq \|a_1'\|_{A_1'} \|\mathbf{a}\|_{\op{sp}} .
    \end{align*}
\end{proof}
\begin{proposition}\label{prop: envelope-of-proj}
    If $A_1$ and $A_2$ C*-algebras, $A_1$ is nuclear and $C^*(A_1\otimes_{\pi} A_2)\cong A_1\otimes_{\op{sp}} A_{2}$, the projective tensor product $A_1\otimes_{\pi} A_2$ is densely embedded in $A_1\otimes_{\op{sp}} A_2$.
\end{proposition}
\begin{proof}
    $C^*(A_1\otimes_{\pi} A_2)$ is, by definition, the completion of the $*$-algebra $A_1\otimes_{\pi} A_2$ with respect to the universal C*-norm $\|\cdot\|_{u}$. On the other hand, $A_1\odot A_2$ is a dense $*$-sub-algebra of $A_1\otimes_{\pi} A_2$ and so $C^*(A_1\otimes_{\pi} A_2)$ is also the completion of $A_1\odot A_2$ with respect to $\|\cdot\|_{u}$. Indeed, $A_1\odot A_2$ is a legitimately admissible $*$-algebra -- i.e. the universal C*-norm is well-defined -- that can be completed to a $C^*$-algebra and so $A\odot B$ has a Banach $*$-algebra norm. Since $\|\cdot\|_{u}$ is a C*-norm on $A_1\odot A_2$ and either $A_1$ or $A_2$ is nuclear, the universal C*-norm $\|\cdot\|_{u}$ coincides with the spatial tensor norm $\|\cdot\|_{\op{sp}}$ on $A_1\odot A_2.$ Therefore $C^*(A_1\otimes_{\pi} A_2) \cong A_1\otimes_{\op{sp}} A_2$.

    The identification map $\iota: A_1\otimes_{\pi} A_2 \hookrightarrow A_1\otimes_{\op{sp}}A_2$ is well-defined due to the cross-norm property of $A_1\otimes_{\op{sp}}A_2$. Indeed, if $\displaystyle \sum_{i=1}^n a_1^{(i)}\otimes a_2^{(i)} \in A_1\otimes_{\op{sp}} A_2$,
    \[
    \left\|\sum_{i=1}^n a_1^{(i)}\otimes a_2^{(i)} \right\|_{\op{sp}} \leq \sum_{i=1}^n \|a_1^{(i)}\|_{A_1}\cdot \|a_2^{(i)}\|_{A_2}.
    \]
    Since $\|\cdot \|_{\op{sp}}$ is a norm on $A_1\otimes_{\op{sp}}A_2$, $\displaystyle \left\|\sum_{i=1}^n a_1^{(i)}\otimes a_2^{(i)} \right\|_{\op{sp}} \leq \left\|\sum_{i=1}^n a_1^{(i)}\otimes a_2^{(i)} \right\|_{\pi}$, which shows that $\iota: A_1\otimes_{\pi}A_2 \hookrightarrow A_1\otimes_{\op{sp}}A_2$ can be extended to a continuous linear map and it remains to show that $\iota$ is also injective. Since either $A_1$ or $A_2$, we use Corollary \ref{cor: many-injections} and show that there is such a bounded linear map $q_{A_1}: A_1\otimes_{\op{sp}}A_2 \to \mathcal{B}(A'_1,A_2)$ that $i_{A_1}= q_{A_1} \circ \iota$. Using Lemma \ref{form: slice-map}, we define for any $t\in A_1\otimes_{\op{sp}} A_2$ and any $a_1'\in A_1'$, $$q_{A_1}(t)(a_1') = (a_1' \otimes \op{Id}_{A_2})(t)\in A_2.$$ From Estimate \eqref{form: slice-map-est} we infer that $\op{sup}_{\|a_1'\|\leq 1}\|q_{A_1}(t)(a_1')\|\leq \|t\|_{\op{sp}}$. Therefore $q_{A_1}:A_1\otimes_{\op{sp}}A_2 \to \mathcal{B}(A_1',A_2)$ is, indeed, a bounded linear map. Since $\iota$ is the identification map, $i_{A_1} = q_{A_1}\circ \iota$. Since at least one of $A_i$ is nuclear, one of them has the approximation property (see Remark \ref{rem: cstar-nuclearity}). It now follows from Corollary \ref{cor: many-injections} that $\iota$ is injective. That $\iota$ has a dense range follows from the fact that it is restricted to the identity map in $A_1\odot A_2$ and Lemma \ref{lem: dense-tensor}.
\end{proof}
\begin{remark}\label{rem: envelope-of-proj}
    Completing $A\odot B$ with respect to the universal C*-norm is, by definition, the maximal tensor norm denoted by $A \otimes_{\op{max}}B$. Therefore, without nuclearity, the best one can get is that $C^*(A\otimes_{\pi}B) \cong  A\otimes_{\op{max}}B$.
\end{remark}
\begin{corollary}\label{cor: dense-tensor-envelope}
    Let $\mathcal{A}_1$ and $\mathcal{A}_2$ be dense $*$-subalgebras of C*-algebras $A_1$ and $A_2$. If $A_1$ is nuclear, $C^*(\mathcal{A}_1\otimes_{\pi}\mathcal{A}_2)\cong A_1\otimes_{\op{sp}}A_2.$ In particular, if $X_1$ and $X_2$ are Banach $*$-algebras and $C^*(X_1)$ is nuclear, $C^*(X_1\otimes_{\pi}X_2)\cong C^*(X_1)\otimes_{\op{sp}}C^*(X_2).$
\end{corollary}
\begin{proof}
    This follows directly from Proposition \ref{prop: proj-into-proj} and Proposition  \ref{prop: envelope-of-proj}.
\end{proof}
Using the spatial tensor product of C*-algebras one can define the external tensor product of Hilbert C*-modules.
\begin{definition}
    The \textbf{external tensor product} $Z_1\otimes_{\op{ex}}Z_2$ 
    of a full Hilbert $A_1$-modules $Z_1$ and a full Hilbert $A_2$-modules $Z_2$
    is a Hilbert $A_1\otimes_{\op{sp}}A_2$-module induced by the following $A_1\otimes_{\op{sp}}A_2$-valued inner product on the elementary tensors
    \begin{align}\label{form: external-inner}
    \lin{A_1\otimes_{\op{sp}}A_2}{\sum_{i=1}^n z_1^{(i)}\otimes z_2^{(i)}}{\sum_{j=1}^m w_1^{(j)}\otimes w_2^{(j)}}:= \sum_{i=1}^n\sum_{j=1}^m \lin{A_1}{z_1^{(i)}}{w_1^{(j)}} \otimes \lin{A_2}{z_2^{(i)}}{w_2^{(j)}}. 
    \end{align}
    for  $z_1^{(i)},w_1^{(j)}\in Z_1$ and $z_2^{(i)},w_2^{(j)}\in Z_2$ where $i=1,...,n$, $j=1,...m$. The induced module norm is
 \begin{equation}\label{eq: tensor-module-norm}
  \left\|\sum_{i=1}^n z_1^{(i)}\otimes z_2^{(i)}  \right\|_{\op{ex}} := \left \| \sum_{i=1}^n \sum_{j=1}^n \lin{A_1}{z_1^{(i)}}{z_1^{(j)}} \otimes \lin{A_2}{z_2^{(i)}}{z_2^{(j)}} \right \|^{1/2}_{\op{sp}}.
 \end{equation}
The $A_{1}\otimes_{\op{sp}}A_2$ module action on $Z_1\otimes_{\op{ex}}Z_2$ can be defined densely by factorwise module action.
\end{definition}
\begin{remark}
If $Z_1$ and $Z_2$ are full, $Z_1 \otimes_{\op{ex}}Z_2$ was shown to be full \cite[Proposition 3.36]{RaWi98}. Consequently, the external tensor product of equivalence bimodules is an equivalence bimodule. Therefore there is the following isomorphism
    \begin{align*}
        \mathcal{K}_{A_1\otimes_{\op{sp}}A_2}(Z_1\otimes_{_{\op{ex}}}Z_2)\cong \mathcal{K}_{A_1}(Z_1)\otimes_{\op{sp}}\mathcal{K}_{A_2}(Z_2) = B_1 \otimes_{\op{sp}} B_2
    \end{align*}
    where $B_i := \mathcal{K}_{A_i}(Z_i)$ for $i=1,2.$ Furthermore, from
    Remark \ref{rem: spatial-tensor-props} and Equation
    \eqref{form: external-inner} it follows that the external tensor product norm is a cross-norm, i.e.
    \begin{align}
        \|z_1\otimes z_2\|_{\op{ex}} = \|z_1\|_{Z_1}\cdot \|z_2\|_{Z_2}, \qquad z_1\in Z_1,\ z_2\in Z_2.
    \end{align}
\end{remark}
\begin{example}\label{exa: ex-are-hilb}
    Aside from being quite a natural tensor product of Hilbert C*-modules, the external tensor product can be seen as a generalization of the Hilbert space tensor product. Indeed, for $i=1,2$, if one views Hilbert spaces $H_i$ as being Hilbert $\mathbb{C}$-modules, the external tensor product norm on $H_1\otimes_{\op{ex}}H_2$ is nothing but the Hilbert space tensor product norm on $H_1\otimes_2 H_2.$
\end{example}
\section{Tensors of Faithfully Localisable Hilbert C*-modules}\label{sec: hilb-cstar-tensors}
In this section, we show that the external tensor product of faithfully localisable Hilbert C*-modules interacts well with the localisation scheme and state and prove inner and outer kernel theorems for the equivalence bimodules. 

We first verify that faithful finite traces can be induced on the spatial tensor product of two C*-algebras if both factors have faithful finite traces.
\begin{lemma}\label{lem: extended-tensor-trace}
    For $i=1, 2$, let $\op{tr}_{B_i}: B_i\to \mathbb{C}$ be faithful finite traces. The map $\op{tr}_{B_1\otimes_{\op{sp}}B_2}: B_1\odot B_2\to \mathbb{C}$ defined by
    \begin{align}\label{form: trace-tensor}
        \op{tr}_{B_1\otimes_{\op{sp}}B_2}:= \op{tr}_{B_1}\otimes \op{tr}_{B_2}
    \end{align}
    is a faithful finite trace on the spatial tensor product $B_1\otimes_{\op{sp}}B_2.$ Furthermore, the operator norm satisfies
    \begin{align}\label{form: trace-tensor-norm}
        \|\op{tr}_{B_1\otimes_{\op{sp}}B_2}\| = \|\op{tr}_{B_1}\| \|\op{tr}_{B_2}\|.
    \end{align}
\end{lemma}
\begin{proof}
Since $\op{tr}_{B_1}$ and $\op{tr}_{B_2}$ are faithful finite traces, it follows from Remark \ref{rem: spatial-tensor-props} that their tensor product $\op{tr}_{B_1\otimes_{\op{sp}}B_2}$ can be linearly and densely extended to a well-defined finite positive linear functional on $B_1\otimes_{\op{sp}}B_2$ that also satisfies the norm condition \eqref{form: trace-tensor-norm}. The trace condition follows from the fact that, for any $\displaystyle \mathbf{b}=\sum_{i=1}^nb_1^{(i)}\otimes b_2^{(i)}, \mathbf{b}'=\sum_{j=1}^m b_1'^{(j)}\otimes b_2'^{(j)}\in B_1\otimes_{\op{sp}}B_2$, 
\begin{align*}
\op{tr}_{B_1\otimes B_2}(\mathbf{b}\cdot  \mathbf{b}')&= \sum_{i=1}^n \sum_{j=1}^m \op{tr}_{B_1}\left(b_1^{(i)}\cdot b_1'^{(j)}\right)\op{tr}_{B_2}\left(b_2^{(i)}\cdot b_2'^{(j)}\right) \\ 
&= \sum_{j=1}^m \sum_{i=1}^n \op{tr}_{B_1}\left(b_1'^{(j)}\cdot b_1^{(i)}\right) \op{tr}_{B_2}\left(b_2'^{(j)}\cdot b_2^{(i)}\right) = \op{tr}_{B_1\otimes_{\op{sp}}B_2}(\mathbf{b}'\cdot \mathbf{b})
\end{align*} 
It remains to show faithfulness. Let $(\pi_1, h_1, H_1)$ and $(\pi_2, h_2,H_2)$ be the associated GNS (cyclic) representations of $\op{tr}_{B_1}$ and $\op{tr}_{B_2}$ respectively. Since $\op{tr}_{B_i}(b_i) = \rin{H_i}{\pi_i(b_i)h_i}{h_i}$ for any $b_i\in B_i$ where $i=1,2$, the faithfulness of both $\op{tr}_{B_i}$ implies that of the $*$-representation $\pi_i$. Furthermore $h_1\otimes h_2$ is, by construction, a cyclic vector for $\pi_1\otimes_{\op{sp}} \pi_2: B_1\otimes B_2\to \mathcal{B}(H_1\otimes_{2} H_2)$ satisfying $\op{tr}_{B_1}\otimes \op{tr}_{B_2}(\mathbf{b}) = \rin{H_1\otimes_2 H_2}{\pi_1\otimes \pi_2 (\mathbf{b})(h_1\otimes h_2)}{h_1\otimes h_2}$ for $\mathbf{b}\in B_1\otimes_{\op{sp}}B_2$. Therefore $\pi_1\otimes \pi_2$ is unitarily equivalent to the GNS representation of $\op{tr}_{B_1\otimes_{\op{sp}}B_2}=\op{tr}_{B_1}\otimes \op{tr}_{B_2}$. Let $t\in B_1\otimes_{\op{sp}}B_2$ be such that $\op{tr}_{B_1}\otimes \op{tr}_{B_2}(t^*t)=0$. Then
\begin{align*}
    \op{tr}_{B_1}\otimes \op{tr}_{B_2}(t^*t) &= \rin{H_1\otimes_2 H_2}{\pi_1\otimes \pi_2(t^*t)(h_1\otimes h_2)}{h_1\otimes h_2} \\
    &= \|\pi_1\otimes \pi_2 (t)(h_1\otimes h_2)\|^2_{H_1\otimes_2 H_2}=0.
\end{align*}
Hence $\pi_1\otimes \pi_2(t)(h_1\otimes h_2)=0$. Now, let $x\in B_1\otimes_{\op{sp}}B_2$. Then,
\begin{align*}
    &\rin{H_1\otimes_2 H_2}{\pi_1\otimes \pi_2(tt^*)\left(\pi_1\otimes \pi_2(x)(h_1\otimes h_2)\right)}{\pi_1\otimes \pi_2 (x)(h_1\otimes h_2)}\\
    &=\rin{H_1\otimes H_2}{\pi_1\otimes\pi_2((t^*x)^*(t^*x)) \left(h_1\otimes h_2\right)}{h_1\otimes h_2}\\
    &\stackrel{\text{trace property}}{=} \rin{H_1\otimes_2 H_2}{\pi_1\otimes \pi_2(t^*xx^*) \left(\pi_1\otimes \pi_2 (t)(h_1\otimes h_2)\right)}{h_1\otimes h_2}\\
    &=0
\end{align*}
We have shown that $\pi_1\otimes \pi_2(t^*t)$ is zero over the dense subspace $\{\pi_1\otimes \pi_2 (x) (h_1\otimes h_2): x\in B_1\otimes_{\op{sp}}B_2 \}\subseteq H_1\otimes_{2} H_2$. Therefore $\pi_1\otimes \pi_2(t^*t)=0$ and so $\pi_1\otimes \pi_2$ is faithful and $t=0$.
\end{proof}
We now state and prove one of the main results of this section, namely that the localisation of the tensor products of full Hilbert C*-modules is the tensor product of the localisation of full Hilbert C*-modules.
\begin{theorem}\label{thm: localization-ext}
    For $i=1,2$, let $Z_i$ be full right Hilbert $B_i$-modules such that $\op{tr}_{B_i}:B_i\to \mathbb{C}$ are finite faithful traces and let $\op{Loc}(Z_1\otimes_{\op{ex}}Z_2)$ be the localisation of $Z_1\otimes_{\op{ex}}Z_2$ with respect to the finite faithful trace $\op{tr}_{B_1\otimes_{\op{sp}} B_2}:=\op{tr}_{B_1}\otimes \op{tr}_{B_2}$. Then $Z_1\otimes_{\op{ex}}Z_2$ is densely embedded in both $\op{Loc}(Z_1\otimes_{\op{ex}} Z_2)$ and $\op{Loc}(Z_1)\otimes_{2}\op{Loc}(Z_2)$. Furthermore, there is the Hilbert space isomorphism
    \begin{align*}
        \op{Loc}(Z_1\otimes_{\op{ex}}Z_2)\cong \op{Loc}(Z_1)\otimes_{2}\op{Loc}(Z_2).
    \end{align*}
    This result can be summarised in the commutative diagram
    \[
\begin{tikzcd}[row sep=large, column sep=large]
& & \op{Loc}(Z_1)\otimes_{2}\op{Loc}(Z_2) \arrow[d, "\cong"]\\
&Z_1\otimes_{\op{ex}}Z_2 \arrow[r, hook] \arrow[ru, hook] &\op{Loc}(Z_1\otimes_{\op{ex}}Z_2).
\end{tikzcd}
\]
Lastly, the embedding satisfies the following estimate
\begin{equation} \label{form: loc-tensor-norm-est}
    \begin{split}
        \|\mathbf{z}\|_{\op{tr}_{B_1\otimes_{\op{sp}} B_2}} \leq \|\op{tr}_{B_1}\| \|\op{tr}_{B_2}\| \|\mathbf{z}\|_{\op{ex}}. 
    \end{split}
\end{equation}
for $\mathbf{z}\in Z_1\otimes_{\op{ext}}Z_2$.
\end{theorem}
\begin{proof}
     That $Z_1\otimes_{\op{ex}}Z_2$ is densely embedded in $\op{Loc}(Z_1\otimes_{\op{ex}}Z_2)$ follows from the construction of localisation of an equivalence bimodule. Since $Z_1 \odot Z_2$ is dense in $Z_1\otimes_{\op{ex}}Z_2$, it is also densely embedded in $\op{Loc}(Z_1\otimes_{\op{ext}} Z_2)$ with respect to the inner-product induced by $\op{tr}_{B_1\otimes_{\op{sp}} B_2}$. On the other hand, there are the dense embeddings $Z_1\hookrightarrow \op{Loc}(Z_1)$ and $Z_2\hookrightarrow \op{Loc}(Z_2)$ and so one can identify any elementary tensors $z_1\otimes z_2$ in $Z_1\otimes_{\op{ex}}Z_2$ as an elementary tensor $\op{Loc}(Z_1)\otimes_{2}\op{Loc}(Z_2)$. From Lemma \ref{lem: dense-tensor} it follows that $Z_1\odot Z_2$ is also dense in $\op{Loc}(Z_1)\otimes_{2}\op{Loc}(Z_2)$ with respect to the Hilbert space tensor norm. Therefore $Z_1\otimes_{\op{ex}}Z_2$ is densely embedded in $\op{Loc}(Z_1)\otimes_{2} \op{Loc}(Z_2)$.

    We now know that the span of elementary tensors of the form $z_1\otimes z_2$ where $z_1\in Z_1$ and $z_2\in Z_2$ is a common dense subspace of both $\op{Loc}(Z_1\otimes_{\op{ext}}Z_2)$ and $\op{Loc}(Z_1)\otimes_{2} \op{Loc}(Z_2)$. We note that, for any elementary tensors $z_1\otimes z_2$ and $z_1'\otimes z_2\in Z_1\odot  Z_2$,
    \begin{equation}\label{form: tr-to-hilb-equation}
        \begin{split}
            \rin{\op{tr}_{B_1\otimes_{\op{sp}} B_2} }{z_1\otimes z_2}{z_1'\otimes z_2'} &=\op{tr}_{B_1\otimes_{\op{sp}}B_2}\left(\rin{B_1\otimes_{\op{sp}} B_2}{z_1\otimes z_2}{z_1'\otimes z_2'}\right)\\
        &= \op{tr}_{B_1\otimes_{\op{sp}}B_2}\left(\rin{B_1}{z_1}{z_1'}\otimes \rin{B_2}{z_2}{z_2'}\right) \\
        &= \op{tr}_{B_1}\left( \rin{B_1}{z_1'}{z_1} \right)\cdot \op{tr}_{B_2} \left(\rin{B_2}{z_2'}{z_2}  \right) \\
        &= \rin{H_1\otimes_{2}H_2 }{z_1\otimes z_2}{z_1'\otimes z_2'}.
        \end{split}
    \end{equation}
    The computation above can be extended linearly. Therefore, on the dense subspace $Z_1 \odot Z_2$, the inner products on $\op{Loc}(Z_1\otimes_{\op{ex}} Z_2)$ and $\op{Loc}(Z_1)\otimes_{2}\op{Loc}(Z_2)$ agree with one another. Therefore $\op{Loc}(Z_1\otimes_{\op{ex}}Z_2)\cong \op{Loc}(Z_1)\otimes_{2}\op{Loc}(Z_2)$. Lastly, we infer Estimate \eqref{form: loc-tensor-norm-est} from \eqref{form: loc-estimate} and \eqref{form: trace-tensor-norm}.
\end{proof}
\begin{corollary}\label{cor: localization-of-external-tensor}
    For $i=1, 2$, let $Z_i$ be faithfully localisable $A_i$-$B_i$ equivalence bimodules with faithful finite traces $\op{tr}_{B_i}: B_i \to \mathbb{C}$. There exists a unique lower semi-continuous faithful trace $\op{tr}_{A_1 \otimes_{\op{sp}} A_2}$ defined on $$\op{span}\left\{\lin{A_1\otimes_{\op{sp}} A_2}{z_1\otimes z_2}{w_1\otimes w_2}: z_1,w_1\in Z_1, \ z_2,w_2\in Z_2 \right\}\subseteq A_1\otimes_{\op{sp}}A_2,$$ fully determined by
    \begin{align*}
        \op{tr}_{A_1\otimes_{\op{sp}} A_2}\left( \lin{A_1\otimes_{\op{sp}}A_2}{z_1\otimes z_2}{w_1\otimes w_2}\right) = \op{tr}_{B_1\otimes_{\op{sp}}B_2} \left( \rin{B_1\otimes_{\op{sp}} B_2}{w_1\otimes w_2}{z_1\otimes z_2} \right)
    \end{align*}
    for any $z_1$ and $w_1\in Z_1$ and any $z_2$ and $ w_2\in Z_2$. Furthermore,
    \begin{enumerate}
        \item the completion of $Z_1\otimes_{\op{ex}}Z_2$ with respect to the induced inner product of $\op{tr}_{A_1\otimes_{\op{sp}} A_2}$ is isomorphic to the Hilbert space tensor $\op{Loc}(Z_1)\otimes_{2}\op{Loc}(Z_2)$;
        \item if $W_i$ are faithfully localisable $A_i$-$B_i$ equivalence bimodule for $i=1, 2$, then there are two continuous embeddings
        \begin{align*}
            \mathcal{L}_{A_1\otimes_{\op{sp}} A_2}(Z_1\otimes_{\op{ex}} Z_2, W_1\otimes_{\op{ex}} W_2) &\hookrightarrow \mathcal{B}(\op{Loc}(Z_1) \otimes_{2} \op{Loc}(Z_2),\op{Loc}(W_1)\otimes_{2}\op{Loc}(W_2)), \\
             \mathcal{L}_{B_1\otimes_{\op{sp}} B_2}(Z_1\otimes_{\op{ex}} Z_2, W_1\otimes_{\op{ex}} W_2) &\hookrightarrow \mathcal{B}(\op{Loc}(Z_1) \otimes_{2} \op{Loc}(Z_2),\op{Loc}(W_1)\otimes_{2}\op{Loc}(W_2));
        \end{align*}
        If $Z_i=W_i$ for $i=1, 2$, these embeddings are isometric. 
    \end{enumerate}
\end{corollary}
\begin{proof}
    The corollary follows directly from Theorem \ref{thm: localization} and Proposition \ref{thm: localization-ext}. 
\end{proof}

Tensor products are intimately related with operators. This is true for the Banach space projective tensor product, as $X\otimes_{\pi}Y\cong k(\mathcal{N}(X',Y))$ due to Corollary \ref{cor: proj-goes-into-nuclear}, and the Hilbert space tensor, due to $H_1\otimes_2 H_2 \cong \mathcal{HS}(H_1',H_2)$. One constructs an appropriate rank-one operator from elementary tensors and then completes the space of resulting finite-rank operators using a topology inherited from the tensor product. Using this idea, we give an operator description of the external tensor product $Z_1\otimes_{\op{ex}}Z_2$. For each $(z_1,z_2)\in Z_1\times Z_2$, we define the map
\begin{align}\label{form: external-tensor-as-operator}
    \mathfrak{R}_{z_1,z_2}: Z_1'\to Z_2 \ \text{by } \mathfrak{R}_{z_1,z_2}(z_1')=\rin{Z_1,Z_1'}{z_1}{z_1'}\cdot z_2
\end{align}
for $z_1'\in Z_1'.$ We then define module actions and inner products for these operators by
\begin{align*}
(a_1\otimes a_2)\cdot\mathfrak{R}_{z_1,z_2} &:= \mathfrak{R}_{a_1\cdot z_1, a_2\cdot z_2}: Z'_1\to Z_2, \\
\mathfrak{R}_{z_1,z_2}\cdot (b_1\otimes b_2) &:= \mathfrak{R}_{z_1\cdot b_1, z_2\cdot b_2}: Z_1'\to Z_2
\end{align*} for any $(a_1,a_2)\in A_1\times A_2$ and any $(b_1,b_2)\in B_1\times B_2$.
Furthermore, if $(z_1',z_2')\in Z_1\times Z_2$, we define
\begin{align*}
    \lin{A_1\otimes_{\op{sp}}A_2}{\mathfrak{R}_{z_1,z_2}}{\mathfrak{R}_{z_1',z_2'}}&:= \lin{A_1}{z_1}{z_1'} \otimes \lin{A_2}{z_2}{z_2'}, \\
    \rin{B_1\otimes_{\op{sp}}B_2}{\mathfrak{R}_{z_1,z_2}}{\mathfrak{R}_{z_1',z_2'}}&:= \rin{B_1}{z_1}{z_1'} \otimes \lin{B_2}{z_2}{z_2'}.
\end{align*}
The linear extension of this structure allows us to view $\op{span}\{\mathfrak{R}_{z_1,z_2}:(z_1,z_2)\in Z_1\times Z_2 \}$ as an $A_1\otimes_{\op{sp}}A_2-B_1\otimes_{\op{sp}}B_2$ pre-inner product bimodule. The structure is essentially inherited from the external tensor product $Z_1\otimes_{\op{ext}} Z_2$. This allows us to state and prove the following proposition.
\begin{proposition}\label{prop: ex-tensor-as-operator}
   The map $\mathfrak{R}: Z_1\odot Z_2 \to \op{span}\{\mathfrak{R}_{z_1,z_2}: (z_1,z_2)\in Z_1\times Z_2\}$ defined by
   \begin{align}
       \mathfrak{R}\left(\sum_{i=1}^n z_1^{(i)}\otimes z_2^{(i)} \right) := \sum_{i=1}^n \mathfrak{R}_{z_1^{(i)},z_2^{(i)}},
   \end{align}
is a bijection that can be extended to the unitary $A_1\otimes_{\op{sp}}A_2$-$B_1\otimes_{\op{sp}}B_2$ bimodule map. The completion $\mathcal{M}(Z_1',Z_2)$ of $\op{span}\{\mathfrak{R}_{z_1,z_2}:(z_1,z_2)\in Z_1\times Z_2 \}$ with respect to the bimodule structure is an $A_1\otimes_{\op{sp}} A_2-B_1\otimes_{\op{sp}}B_2$ equivalence bimodule. 
\end{proposition}
\noindent It is far from being obvious whether the space $\mathcal{M}(Z_1',Z_2)$, which is an abstract completion of finite-rank operators, defines bounded linear maps from $Z_1'$ onto $Z_2$. The following lemma provides a sufficient condition for this.
\begin{lemma}\label{lem: for-external-embed}
    For $i=1, 2$, let $Z_i$ be $A_i-B_i$ equivalence bimodules. If $B_1$ is unital, $\mathcal{M}(Z_1',Z_2)\subseteq \mathcal{B}(Z_1',Z_2)$ and
    \begin{align*}
        \|\mathfrak{R}(\mathbf{z})\|_{\mathcal{B}(Z_1',Z_2)}\leq \|\mathbf{z}\|_{\op{ex}}=\|\mathfrak{R}(\mathbf{z})\|_{\mathcal{M}(Z_1',Z_2)}, \qquad \mathbf{z}\in Z_1\otimes_{\op{ex}}Z_2.
    \end{align*}
\end{lemma}
\begin{proof}
    Since $B_1$ is unital, there are such $w_1^{(1)},...,w_1^{(n)}\in Z_1$ that $\displaystyle \sum_{k=1}^n \rin{B_1}{w_1^{(k)}}{w_1^{(k)}}=1_{B_1}$ due to Proposition \ref{prop: parseval-module-frame}. Let $\displaystyle \mathbf{z} = \sum_{i=1}^m z_1^{(i)}\otimes z_2^{(i)} \in Z_1\otimes_{\op{ex}}Z_2$. For $z_1'\in Z_1'$, let us define $f_k: A_1\to \mathbb{C}$ by $\rin{A_1,A_1'}{a_1}{f_k}=\rin{Z_1,Z_1'}{a_1\cdot w_k}{z_1'}$ for any $a_1\in A_1$ and $k=1,...,m.$ It follows that $f_k\in A_1'$, and $\displaystyle \rin{Z_1,Z_1'}{z_1}{z_1'} = \sum_{k=1}^n \rin{A_1,A_1'}{\lin{A_1}{z_1}{w_k}}{f_k}$ For $z_2\in Z_2$,
    \begin{align*}
        \lin{A_2}{\mathfrak{R}(\mathbf{z})(z_1')}{z_2} &= \sum_{i=1}^{m} \rin{Z_1,Z_1'}{z_1^{(i)}}{z_1'} \lin{A_2}{z_2^{(i)}}{z_2}\\
        &= \sum_{i=1}^m \rin{Z_1,Z_1'}{\sum_{k=1}^n \lin{A_2}{z_2^{(i)}}{w_1^{(k)}}\cdot w_1^{(k)} }{z_1'}\cdot  \lin{A_2}{z_1^{(i)}}{z_2} \\
        &= \sum_{i=1}^m \sum_{k=1}^n \rin{A_1,A_1'}{ \lin{A_1}{z_1^{(i)}}{w_1^{(k)}} }{f_k}\cdot  \lin{A_2}{z_2^{(i)}}{z_2} \\
        &=\sum_{k=1}^n \left((f_k \otimes \op{Id}_{A_2}) \left( \sum_{i=1}^m \lin{A_1}{z_1^{(i)}}{w_1^{(k)}} \otimes \lin{A_2}{z_2^{(i)}}{z_2}  \right)   \right) \\
        &=\sum_{k=1}^n \left((f_k \otimes \op{Id}_{A_2}) \left( \lin{A_1\otimes_{\op{sp}}A_2}{\sum_{i=1}^m z_1^{(i)}\otimes z_2^{(i)} }{ w_1^{(k)}\otimes z_2}  \right) \right) .
    \end{align*}
    It follows from Lemma \ref{lem: slice-map-lem} that
    \begin{align*}
    \|\lin{A_2}{\mathfrak{R}(\mathbf{z})(z_1')}{z_2} \|_{A_2} \leq \sum_{k=1}^n \left\|f_k \right\|_{A_1'} \left\|\lin{A_1\otimes_{\op{sp}}A_2}{\sum_{i=1}^m z_1^{(i)}\otimes z_2^{(i)} }{ w_1^{(k)}\otimes z_2} \right\|_{\op{sp}}.
    \end{align*}
    Applying the Cauchy-Schwarz inequality to the external tensor product $Z_1\otimes_{\op{ex}}Z_2$ results in
    \begin{align}\label{form: intermediate-frakR}
        \|\lin{A_2}{\mathfrak{R}(\mathbf{z})(z_1')}{z_2} \|_{A_2} \leq \sum_{k=1}^n\|f_k\|_{A_1'}\cdot \|w_1^{(k)}\|_{Z_1} \cdot \|z_2\|_{Z_2}\cdot \|\mathbf{z}\|_{\op{ex}}.
    \end{align}
    $0\leq \rin{B_1}{w_1^{(k)}}{w_1^{(k)}}\leq 1_{B_1}$ for any $k=1,...,n$.  We first note that $\|w_1^{(k)}\|_{Z_1}\leq 1$. We now estimate $\displaystyle \sum_{i=1}^{n}\|f_k\|_{A_1'}$. By definition, for any $\varepsilon>0$, there is such a $v_1\in Z_1$ that $\|v_1\|_{Z_1}\leq 1$ and $\displaystyle \|z_1'\|_{Z_1'}-\varepsilon < \left|\rin{Z_1,Z_1'}{v_1}{z_1'}\right|\leq \sum_{i=1}^n \left|\rin{A_1,A_1'}{\lin{A_1}{v_1}{w_k}}{f_k}\right|\leq \sum_{i=1}^n\|f_k\|_{A_1'}$. Since $\varepsilon$ is arbitrary, we infer that $\displaystyle \|z_1'\|_{Z_1'}\leq \sum_{i=1}^n \|f_k\|_{A_1'}$. On the other hand, $\|z_1'\|_{Z_1'}=\displaystyle \sup_{\|z_1\|_{Z_1}\leq 1}\left|\rin{Z_1,Z_1'}{z_1}{z_1'}\right | = \sup_{\|z_1\|_{Z_1}\leq 1} \left|\sum_{i=1}^n \rin{A_1,A_1'}{\lin{A_1}{z_1}{w_k}}{f_k} \right|\leq \sum_{i=1}^n \|f_k\|_{A_1'}.$ Therefore $\displaystyle \|z_1'\|_{Z_1'}= \sum_{i=1}^n \|f_k\|_{A_1'}$. Finally, from Inequality \eqref{form: intermediate-frakR} we infer that
    \begin{align*}
        \|\mathfrak{R}(\mathbf{z})(z_1')\|_{Z_2} = \sup_{\|z_2\|_{Z_2}\leq 1}\|\lin{A_2}{\mathfrak{R}(\mathbf{z})(z_1')}{z_2}\|_{A_2}\leq \|z_1'\|_{Z_1'} \|\mathbf{z}\|.
    \end{align*}
    Taking the supremum over all $\|z_1'\|_{Z_1'}\leq 1$ in this inequality results in $\|\mathfrak{R}(\mathbf{z})\|_{\mathcal{B}(Z_1',Z_2)}\leq \|\mathbf{z}\|_{\op{ex}}$.
\end{proof}
The assumption that $B_1$ is unital is essential to this proof. We made use of the fact that $Z_1$ must be a finitely generated and projective Hilbert $A_1$-module. This, in its turn, allowed us to pull back linear functionals in $Z_1'$ to linear functionals in $A_1'$. We can prove even more under this assumption.
\begin{lemma}\label{lem: z-has-ap}
    Let $Z$ be an $A-B$ equivalence bimodule with $B$ being unital and nuclear. Then $Z$ has the approximation property. 
\end{lemma}
\begin{proof}
Since $B$ is nuclear and $A$ and $B$ are Morita equivalent through $Z$, $A$ has the approximation property due to Remark \ref{rem: cstar-nuclearity}. As $B$ is unital, there is such a set $\{w_1,...,w_n\}\subseteq Z$ that $\displaystyle \sum_{i=1}^n \rin{B}{w_i}{w_i}=1_B$, when $\displaystyle z = \sum_{i=1}^n \lin{A}{z}{w_i}\cdot w_i$ for all $z\in Z$. Now, let $\varepsilon>0$ and $K$ be a compact subset of $Z.$ We define the operator $F\in \mathcal{L}_A(Z,A)\subseteq \mathcal{B}(Z,A)$ by
\begin{align*}
    F(a) = \sum_{i=1}^n \lin{A}{z}{w_i}.
\end{align*}
Since $A$ has the approximation property and $F(K)$ is a compact subset of $A$, there exists, according to Proposition \ref{prop: the-ap} Item (1), such a finite-rank operator $F_A: A\to A$ that $\|a- F_A(a)\|<\varepsilon$ for any $a\in F(K)$. Equivalently, 
\begin{align*}
    \|F(z)- F_A(F(z))\|_A<\varepsilon, \qquad z\in K.
\end{align*}
We define $T\in \mathcal{B}(Z,Z)$ via $\displaystyle T(z) = \sum_{i=1}^n F_A(\lin{A}{z}{w_i})\cdot w_i$ for any $z\in Z$. We note that $T$ must be of finite-rank as $F_A$ is. More explicitly, $\displaystyle F_A(a) = \sum_{k=1}^{m} \rin{A,A'}{a}{\varphi_k}a_k$ for some $\varphi_k\in A'$, and $a_k\in A$, $k=1,...,m$. Therefore
\begin{align*}
    T(z) = \sum_{i=1}^n \sum_{k=1}^m \rin{Z,Z'}{z}{z'_{i,j}} z_{i,j}
\end{align*}
where $z_{i,j}'\in Z'$ is given by $\rin{Z,Z'}{z}{z_{i,j}'} = \rin{A,A'}{\lin{A}{z}{w_i} }{ \varphi_k }$ and $z_{i,j}\in Z$ is given by $z_{i,j}= a_k\cdot w_j.$ Since $\|w_i\|_{Z}\leq 1$ for any $i=1,...,n$, for any $z\in K$,
\begin{align*}
    \|z - T(z)\|_Z &= \left\|\sum_{i=1}^n \lin{A}{z}{w_i}\cdot w_i - \sum_{i=1}^n F_A(\lin{A}{z}{w_i})\cdot w_i \right\|_Z \\
    &\leq \left\|\sum_{i=1}^n \lin{A}{z}{w_i} -F_A\left(\sum_{i=1}^n \lin{A}{z}{w_i}\right)\right\|_A \\
    & = \|F(z) - F_A(F(z))\|_A<\varepsilon
\end{align*} and so $Z$ has the approximation property.
\end{proof}

\begin{corollary}\label{cor: proj-into-external}
    Let $Z_i$ be an $A_i-B_i$-equivalence bimodule. If $B_1$ is unital and nuclear, there is the embedding
    \[
    j: Z_1\otimes_{\pi}Z_2 \hookrightarrow Z_1\otimes_{\op{ex}}Z_2
    \]
    with dense range satisfying $\|j(\mathbf{z})\|_{\op{ex}}\leq \|\mathbf{z}\|_{\pi(Z_1,Z_2)}$ for any $\mathbf{z}\in Z_1\otimes_{\pi}Z_2.$
\end{corollary}
\begin{proof}
We consider $j: \op{id}_{Z_1}\otimes \op{id}_{Z_2}: Z_1\odot Z_2\to Z_1\odot Z_2$. For $\mathbf{z} = \sum_{i=1}^n z_1^{(i)}\otimes z_2^{(2)}\in Z_1\odot Z_2$, it follows from the Cauchy-Schwarz inequality for Hilbert C*-modules that $\|j(\mathbf{z})\|_{\op{ex}}\leq \|\mathbf{z}\|_{\pi(Z_1,Z_2)}$ and so it can be extended to the continuous linear map $j: Z_1\otimes_{\pi}Z_2\to Z_1\otimes_{\op{ex}}Z_2$. It must have dense range according to Lemma \ref{lem: dense-tensor}.

Now, lemma \ref{lem: for-external-embed} implies that $(\mathfrak{R} \circ j)(\mathbf{z})\in \mathcal{B}(Z_1',Z_2)$ for any $\mathbf{z}\in Z_1\otimes_{\pi}Z_2$. According to Lemma \ref{lem: z-has-ap}, $Z_1$ has the approximation property and so $j$ must be injective due to Corollary \ref{cor: many-injections}. 
\end{proof}
The following result allows us to embed the the projective tensor product of pre-equivalence bimodules in the external tensor product of their completions.
\begin{theorem}\label{thm: proj-tensor-are-preequiv}
    For $i=1, 2$, let $Z_i$ be $A_i-B_i$ equivalence bimodules faithfully localisable from $\mathcal{A}_i-\mathcal{B}_i$ pre-equivalence bimodules $\mathcal{Z}_i$ and $\mathcal{Z}_i$ Banach spaces whose topologies are finer than those induced by their $\mathcal{A}_i-\mathcal{B}_i$ pre-equivalence bimodule structures where there are such constants $C_1, C_2>0$ that
    \begin{align*}
        \|\mathscr{z}_i\|_{Z_i} \leq C_i \|\mathscr{z}_i\|_{\mathcal{Z}_i}, \qquad \forall \mathscr{z}_i\in \mathcal{Z}_i,\ \text{for i=1,2}.
    \end{align*}
    Furthermore, let $B_1$ be a nuclear unital $C^*$-algebra and $\mathcal{A}_1$, $\mathcal{B}_1$ and $\mathcal{Z}_1$ all possess the approximation property as Banach spaces. Then $\mathcal{Z}_1\otimes_{\pi}\mathcal{Z}_2$ is an $\mathcal{A}_1\otimes_{\pi} \mathcal{A}_2 - \mathcal{B}_1\otimes_{\pi} \mathcal{B}_2$ pre-equivalence bimodule and there is the Banach space embedding
    \begin{align*}
        \mathcal{j}: \mathcal{Z}_1\otimes_{\pi}\mathcal{Z}_2 \hookrightarrow Z_1\otimes_{\op{ex}}Z_2
    \end{align*} with a dense range that satisfies
        \begin{align}\label{form: proj-to-ex-continuity}
        \|\mathcal{j}(\mathscr{z})\|_{\op{ex}}\leq C_1C_2 \|\mathscr{z}\|_{\pi(\mathscr{Z}_1,\mathscr{Z}_2)}, \qquad \forall \mathbf{z}\in \mathcal{Z}_1\otimes_{\pi}\mathcal{Z}_2.
    \end{align}
    The same injection implements $Z_1\otimes_{\op{ex}}Z_2$ as a faithfully localisable $A_1\otimes_{\op{sp}}A_2-B_1\otimes_{\op{sp}}B_2$ equivalence bimodule from $\mathcal{Z}_1 \otimes_{\pi} \mathcal{Z}_2$. 
\end{theorem}
\begin{proof}
    As $B_1$ is nuclear, so is $A_1$. Therefore, either of them possesses the approximation property according to Remark \ref{rem: cstar-nuclearity}. From the hypothesis that $\mathcal{B}_1$ and $\mathcal{A}_1$ both have approximation property and Propositions \ref{prop: proj-into-proj} and \ref{prop: envelope-of-proj} we infer that both $\mathcal{A}_1\otimes_{\pi}\mathcal{A}_2$ and $\mathcal{B}_1\otimes_{\pi}\mathcal{B}_2$ can be seen as dense $*$-subalgebras of their C*-envelopes $A_1\otimes_{\op{sp}}A_2$ and $B_1\otimes_{\op{sp}}B_2$ respectively. We now define the following structures on $\mathcal{Z}_1 \otimes_{\pi} \mathcal{Z}_2$, which will make it into a $\mathcal{A}_1\otimes_{\pi} \mathcal{A}_2 - \mathcal{B}_1 \otimes_{\pi}\mathcal{B}_2$ pre-equivalence bimodule. For $\displaystyle \mathscr{a} = \sum_{i=1}^{\infty}\mathscr{a}_1^{(i)}\otimes \mathscr{a}_2^{(i)}\in \mathcal{A}_1 \otimes_{\pi} \mathcal{A}_2$, $\mathscr{z}\in \mathcal{Z}_1\otimes_{\pi} \mathcal{Z}_2$ and $\mathscr{w}= \mathcal{Z}_2\otimes_{\pi} \mathcal{Z}_2$, we define
    \begin{align*}
        \mathscr{a} \cdot \mathscr{z}&:= \sum_{i=1}^{\infty}\sum_{j=1}^{\infty}\mathscr{a}_1^{(i)}\cdot \mathscr{z}_1^{(j)} \otimes \mathscr{a}_2^{(i)} \cdot \mathscr{z}_2^{(j)},\\
        \lin{\mathcal{A}_1\otimes_{\pi}\mathcal{A}_2}{\mathscr{z}}{\mathscr{w}}&:= \lin{A_1\otimes_{\op{sp}}A_2 }{\mathscr{z}}{\mathscr{w}} = \sum_{i=1}^{\infty}\sum_{j=1}^{\infty}\lin{A_1}{\mathscr{z}_1^{(i)}}{\mathscr{w}_1^{(j)}}\otimes \lin{A_2}{\mathscr{z}_2^{(i)}}{\mathscr{w}_2^{(j)}}.
    \end{align*}
    Let $\displaystyle \mathscr{z} = \sum_{i=1}^{\infty} \mathscr{z}_1^{(i)}\otimes \mathscr{z}_2^{(i)}$ and $\displaystyle \mathscr{w}=\sum_{i=1}^{\infty}\mathscr{w}_1^{(i)}\otimes \mathscr{w}_2^{(i)}.$ Applying the Cauchy-Schwarz inequality to the Hilbert C*-modules results in
    \begin{align*}
        \|\lin{\mathcal{A}_1\otimes_{\pi}\mathcal{A}_2 }{\mathscr{z}}{\mathscr{w}}\|_{\op{sp}} &\leq \sum_{j=1}^{\infty}\sum_{i=1}^{\infty}\left\|\lin{A_1}{\mathscr{z}_1^{(i)}}{\mathscr{w}_1^{(j)}} \right\| \left\|\lin{A_2}{\mathscr{z}_2^{(i)}}{\mathscr{w}_2^{(j)}} \right\| \\
        & \leq \sum_{i=1}^{\infty}\|\mathscr{z}_1^{(i)}\|_{Z_1}\cdot \|\mathscr{z}_2^{(i)}\|_{Z_2} \cdot \sum_{j=1}^{\infty}\|\mathscr{w}_1^{(j)}\|_{Z_1}\cdot \|\mathscr{w}_2^{(j)}\|_{Z_2} \\
        &\leq C_1^2C_2^2 \cdot \sum_{i=1}^{\infty}\|\mathscr{z}_1^{(i)}\|_{\mathcal{Z}_1}\cdot \|\mathscr{z}_2^{(i)}\|_{\mathcal{Z}_1} \cdot \sum_{j=1}^{\infty}\|\mathscr{w}_1^{(j)}\|_{\mathcal{Z}_1}\cdot \|\mathscr{w}_2^{(j)}\|_{\mathcal{Z}_2}.
    \end{align*}
    Therefore $\|\lin{\mathcal{A}_1\otimes_{\pi} \mathcal{A}_2}{\mathscr{z}}{\mathscr{w}}\|_{\op{sp}} \leq C_1^2C_2^2 \|\mathscr{z}\|_{\pi} \cdot \|\mathscr{w}\|_{\pi}$ and so $\lin{\mathcal{A}_1\otimes_{\pi}\mathcal{A}_2 }{\mathscr{z}}{\mathscr{w}}\in \mathcal{A}_1 \otimes_{\pi} \mathcal{A}_2$. A similar estimate shows that $\mathscr{a}\cdot \mathscr{z} \in \mathcal{Z}_1\otimes_{\pi}\mathcal{Z}_2.$ 
Setting $\mathscr{w}=\mathscr{z}$ in the previous expression results in 
    \begin{align*}
        \|\mathscr{z}\|_{\op{ex}}^2 = \|\lin{A_1\otimes_{\op{sp}}A_2}{\mathscr{z}}{\mathscr{z}}\|\leq C_1^2 C_2^2 \|\mathscr{z}\|_{\pi(\mathscr{Z}_1,\mathscr{Z}_2)}^2
    \end{align*} as the external tensor product norm is a cross-norm.
    It is immediately clear that
    \begin{align}        \|\mathscr{z}\|_{\op{ex}}\leq C_1C_2 \|\mathscr{z}\|_{\pi(\mathscr{Z}_1,\mathscr{Z}_2)}.
    \end{align}
    From Proposition \ref{prop: proj-into-proj} and Corollary \ref{cor: proj-into-external} it follows that there are the embeddings with norm dense ranges $\iota: \mathcal{Z}_1\otimes_{\pi}\mathcal{Z}_2 \hookrightarrow Z_1\otimes_{\pi}Z_2$, and $j:Z_1\otimes_{\pi}Z_2 \hookrightarrow Z_1\otimes_{\op{ex}}Z_2$. We now define $\mathcal{j} = j \circ \iota : \mathcal{Z}_1\otimes_{\pi}\mathcal{Z}_2 \hookrightarrow Z_1\otimes_{\op{ex}}Z_2$. For $\mathscr{a}\in \mathcal{A}_1\otimes_{\pi}\mathcal{A}_2$, $\mathscr{b}\in \mathcal{B}_1\otimes_{\pi}\mathcal{B}_2$, and $\mathscr{z},\mathscr{w},\mathscr{v}\in \mathcal{Z}_1\otimes_{\pi}\mathcal{Z}_2$, it follows from the definition of the above-defined maps and the structure imposed on $\mathcal{Z}_1\otimes_{\pi}\mathcal{Z}_2$ that
    \begin{align*}
        \lin{A_1\otimes_{\op{sp}}A_2}{\mathcal{j}(\mathscr{z})}{\mathscr{j}(\mathscr{w})} &= \lin{\mathcal{A}_1\otimes_{\pi}\mathcal{A}_2}{\mathscr{z}}{\mathscr{w}}, \\
        \rin{B_1\otimes_{\op{sp}}B_2}{\mathcal{j}(\mathscr{z})}{\mathcal{j}(\mathscr{w})} &= \rin{\mathcal{B}_1\otimes_{\pi}\mathcal{B}_2}{\mathscr{z}}{\mathscr{w}}.
    \end{align*}
    Furthermore, $\mathcal{j}(\mathscr{a}\cdot\mathscr{z})= \mathscr{a}\cdot \mathcal{j}(\mathscr{z})$ and $j(\mathscr{z}\cdot \mathscr{b})= \mathscr{j}(\mathscr{z})\cdot \mathscr{b}$. Since $Z_1\otimes_{\op{ex}}Z_2$ is an equivalence bimodule, the algebraic properties of $\mathcal{j}$ imply that
    \begin{align*}
        \rin{\mathcal{B}_1\otimes_{\pi}\mathcal{B}_2}{\mathscr{a}\cdot \mathscr{z}}{\mathscr{a}\cdot \mathscr{z}} \leq \|\mathscr{a}\|_{\op{sp}}^2\rin{\mathcal{B}_1\otimes_{\pi}\mathcal{B}_2}{\mathscr{z}}{\mathscr{z}} \  &\text{and } \lin{\mathcal{A}_1\otimes \mathcal{A}_2}{\mathscr{z}\cdot \mathscr{b}}{\mathscr{z}\cdot \mathscr{b}} \leq \|\mathscr{b}\|_{\op{sp}}^2 \lin{\mathcal{A}_1\otimes_{\pi}\mathcal{A}_2}{\mathscr{z}}{\mathscr{z}}
    \end{align*}
    \begin{align*}
    \lin{\mathscr{A}_1\otimes_{\pi}\mathscr{A}_2}{\mathscr{z}}{\mathscr{w}}\cdot \mathscr{v}= \mathscr{z}\cdot \rin{\mathcal{B}_1\otimes_{\pi}\mathcal{B}_2}{\mathscr{w}}{\mathscr{v}}.
    \end{align*}
    Since, for $i=1, 2$, $\mathscr{Z}_i$ are pre-equivalence bimodules, we can proceed in a way similar to the proof of Lemma \ref{lem: dense-tensor} and show that
    \begin{align*}
        \mathcal{A}_1\odot \mathcal{A}_2 &\subseteq \overline{\op{span}}^{\pi(\mathscr{Z}_1,\mathscr{Z}_2)}\left\{\lin{\mathcal{A}_1 \otimes_{\pi}\mathcal{A}_2}{\mathscr{z}}{\mathscr{w}}: \mathscr{z},\mathscr{w}\in \mathcal{Z}_1\otimes_{\pi}\mathcal{Z}_2 \right\}, \\
       \mathcal{B}_1\odot \mathcal{B}_2 &\subseteq \overline{\op{span}}^{\pi(\mathscr{Z}_1,\mathscr{Z}_2)}\left\{\rin{\mathcal{B}_1 \otimes_{\pi}\mathcal{B}_2}{\mathscr{z}}{\mathscr{w}}: \mathscr{z},\mathscr{w}\in \mathcal{Z}_1\otimes_{\pi}\mathcal{Z}_2 \right\}.
    \end{align*}
    By taking the closure with respect to the respective projective tensor product norms, we note that
        \begin{align*}
        \mathcal{A}_1\otimes_{\pi} \mathcal{A}_2 &= \overline{\op{span}}^{\pi(\mathscr{Z}_1,\mathscr{Z}_2)}\left\{\lin{\mathcal{A}_1 \otimes_{\pi}\mathcal{A}_2}{\mathscr{z}}{\mathscr{w}}: \mathscr{z},\mathscr{w}\in \mathcal{Z}_1\otimes_{\pi}\mathcal{Z}_2 \right\}, \\
       \mathcal{B}_1\otimes_{\pi} \mathcal{B}_2 &= \overline{\op{span}}^{\pi(\mathscr{Z}_1,\mathscr{Z}_2)}\left\{\rin{\mathcal{B}_1 \otimes_{\pi}\mathcal{B}_2}{\mathscr{z}}{\mathscr{w}}: \mathscr{z},\mathscr{w}\in \mathcal{Z}_1\otimes_{\pi}\mathcal{Z}_2 \right\}.
    \end{align*}
    Taking the closures with respect to the universal C*-norm or, equivalently, with respect to the spatial tensor products results in the sets
    \begin{align*}
    &\left\{\lin{\mathcal{A}_1 \otimes_{\pi}\mathcal{A}_2}{\mathscr{z}}{\mathscr{w}}: \mathscr{z},\mathscr{w}\in \mathcal{Z}_1\otimes_{\pi}\mathcal{Z}_2 \right\}, \\
    & \left\{\rin{\mathcal{B}_1 \otimes_{\pi}\mathcal{B}_2}{\mathscr{z}}{\mathscr{w}}: \mathscr{z},\mathscr{w}\in \mathcal{Z}_1\otimes_{\pi}\mathcal{Z}_2 \right\}.
       \end{align*}
Therefore $\lin{\mathcal{A}_1\otimes_{\pi}\mathcal{A}_2}{ \mathcal{Z}_1\otimes_{\pi}\mathcal{Z}_2 }{\mathcal{Z}_1\otimes_{\pi}\mathcal{Z}_2 }$ and $\rin{\mathcal{B}_1\otimes_{\pi}\mathcal{B}_2}{\mathcal{Z}_1\otimes_{\pi}\mathcal{Z}_2}{\mathcal{Z}_1\otimes_{\pi}\mathcal{Z}_2}$ span dense subspaces in $A_1\otimes_{\op{sp}}A_2$ and $B_1\otimes_{\op{sp}}B_2$ respectively. Finally, for $i=1,2$, let $\op{tr}_{\mathcal{B}_i}:\mathcal{B}_i\to \mathbb{C}$ be the traces in $\mathcal{B}_i$ that can be extended to faithful finite traces $\op{tr}_{B_i}:B_i\to \mathbb{C}$. It follows from Lemma \ref{lem: extended-tensor-trace} that $\op{tr}_{B_1\otimes_{\op{sp}}B_2}:=\op{tr}_{B_1}\otimes \op{tr}_{B_2}: B_1\otimes_{\op{sp}}B_2\to\mathbb{C}$ is a finite faithful trace in $B_1\otimes_{\op{sp}}B_2$. Therefore the restriction $\op{tr}_{\mathcal{B}_1\otimes_{\pi}\mathcal{B}_2}:= (\op{tr}_{B_1\otimes_{\op{sp}}B_2})_{\mathcal{B}_1\otimes_{\mathcal{B}_2}}$ is also a faithful finite trace on $\mathcal{B}_1\otimes_{\op{sp}}\mathcal{B}_2$ in the sense of Remark \ref{rem: extending-traces}. We now have shown that $\mathcal{Z}_1\otimes_{\pi}\mathcal{Z}_2$ is a pre-equivalence $\mathcal{A}_1\otimes_{\pi}\mathcal{A}_1 - \mathcal{B}_1\otimes_{\pi}\mathcal{B}_2$ bimodule and that $Z_1\otimes_{\op{ex}}Z_2$ is faithfully localisable from this pre-equivalence bimodule. 
\end{proof}

We would like to use Theorem \ref{thm: proj-tensor-are-preequiv} to construct formulate and prove kernel theorems for equivalence bimodules. This is why we need a more concrete description of the dual space $\mathcal{M}(Z_1',Z_2)'$. As $Z_1\otimes_{\pi}Z_2\hookrightarrow Z_1\otimes_{\op{ex}}Z_2$ is a norm dense embedding, the adjoint embedding $(Z_1\otimes_{\op{ex}}Z_2)'\hookrightarrow (Z_1\otimes_{\pi} Z_2)'$ must be weak-$*$ dense. We can therefore think of $(Z_1\otimes_{\op{ex}}Z_2) '\cong  \mathcal{M}(Z_1',Z_2)'$ as bounded linear operators inside $\mathcal{B}(Z_1,Z_2')\cong (Z_1\otimes_{\pi}Z_2)'$. Let us review the isomorphism $\mathcal{B}(Z_1,Z_2')\cong (Z_1\otimes_{\pi}Z_2)'$. For $T\in \mathcal{B}(Z_1,Z_2')$, it defines a functional $\phi_{T}: Z_1 \odot Z_2\to \mathbb{C}$ by
\begin{align}\label{form: the-map-b}
    \phi_{T}\left(\sum_{i=1}^{n}z_1^{(i)}\otimes z_2^{(i)} \right):= \sum_{i=1}^{n}\rin{Z_2,Z_2'}{z_2^{(i)}}{T(z_1^{(i)})}   
\end{align} that satisfies the condition
\begin{align}\label{form: proj-continuity}
    |\phi_T(\mathbf{z})|\leq \|T\|_{\mathcal{B}(Z_1,Z_2')}\|\mathbf{z}\|_{\op{\pi(Z_1,Z_2)}}, \qquad \forall \mathbf{z}=\sum_{i=1}^n z_1^{(i)}\otimes z_2^{(i)}\in Z_1\odot Z_2.
\end{align}
Therefore $\phi_T \in (Z_1\otimes_{\pi}Z_2)'$ and there is the isometry $\|\phi_T\|_{(Z_1\otimes_{\pi}Z_2)'}=\|T\|_{\mathcal{B}(Z_1,Z_2')}$. We can strengthen Condition \eqref{form: proj-continuity} and make in into a definition.
\begin{definition}\label{def: module-continuous}
    Let $Z_i$ be an $A_i-B_i$ equivalence bimodule for $i=1,2$. For $T\in \mathcal{B}(Z_1,Z_2')$, we say that $T$ is \emph{module continuous} if
    \begin{align}\label{form: module-continuous}
    |\phi_T(\mathbf{z})|\leq \|T\|_{\mathcal{B}(Z_1,Z_2')} \|\mathbf{z}\|_{\op{ex}}, \qquad \forall \mathbf{z}\in Z_1\odot Z_2.
    \end{align}
     This definition makes sense even if the bimodules $Z_i$ do not satisfy the hypotheses of Theorem \ref{thm: proj-tensor-are-preequiv}.
    We denote the space of all module continuous operators by $\mathcal{B}_{\mathcal{M}}(Z_1,Z_2')$.
\end{definition}
\begin{lemma}\label{lem: module-continuous-maps}
    For $i=1,2.$, let $Z_i$ be $A_i-B_i$ equivalence bimodules. There is the isomorphism
    \begin{align*}
        \mathcal{B}_{\mathcal{M}}(Z_1,Z_2') \cong (Z_1\otimes_{\op{ex}}Z_2)'.
    \end{align*}
\end{lemma}
\begin{proof}
For any $T\in \mathcal{B}_{\mathcal{M}}(Z_1,Z_2')$, we denote the functional defined by \eqref{form: the-map-b} by $b(T) = \phi_T: Z_1\odot Z_2\to \mathbb{C}$. Since $T$ is module continuous, $b(T)$ can be extended to a continuous linear functional on $(Z_1\otimes_{\op{ex}} Z_2)'$. Furthermore, from \eqref{form: module-continuous} we infer that $\|b(T)\|_{(Z_1\otimes_{\op{ex}}Z_2)'}\leq \|T\|_{\mathcal{B}(Z_1,Z_2')}$. On the other hand,
\begin{align*}
    \|T\|_{\mathcal{B}(Z_1,Z_2')} &= \sup_{\|z_1\|_{Z_1}\leq 1,\  \|z_2\|_{Z_2}\leq 1} |\rin{Z_2,Z_2'}{z_2}{T(z_1)} | \\
    &= \sup_{\|z_1\|_{Z_1}\leq 1,\  \|z_2\|_{Z_2}\leq 1} | \rin{(Z_1\otimes_{\op{ex}}Z_2), (Z_1\otimes_{\op{ex}}Z_2)'}{z_1\otimes z_2}{b(T)} |\leq \|b(T)\|_{(Z_1\otimes_{\op{ex}} Z_2)'}.
\end{align*}
Therefore there is the isometry $\|b(T)\|_{(Z_1\otimes_{\op{ex}} Z_2)'}=\|T\|_{\mathcal{B}(Z_1,Z_2')}.$ For any $\Phi\in (Z_1\otimes_{\op{ex}} Z_2)'$, we define $b^{-1}(\Phi): Z_1\to Z_2'$ by
\begin{align*}
    b^{-1}(\Phi)(z_1)(z_2) = \Phi(z_1\otimes z_2), \qquad \forall (z_1,z_2)\in Z_1\times Z_2.
\end{align*}
$b^{-1}(\Phi)$ is linear in $Z_1$ and $\sup_{\|z_1\|_{Z_1}\leq 1,\ \|z_2\|_{Z_1}\leq 1}|b^{-1}(\Phi)|\leq \|\Phi\|_{(Z_1\otimes_{\op{ex}}Z_2)'}$. Therefore $b^{-1}(\Phi)\in \mathcal{B}(Z_1,Z_2')$. Furthermore, $\phi_{b^{-1}(\Phi)}= \Phi$ and so $b^{-1}(\Phi)$ defines a module continuous map in $\mathcal{B}(Z_1,Z_2')$ and $b^{-1}$ is the inverse for $b.$
\end{proof}
As we saw, there is the embedding $\mathcal{B}_{\mathcal{M}}(Z_1,Z_2') \hookrightarrow \mathcal{B}(Z_1,Z_2')$. However, without the hypotheses of Theorem \ref{thm: proj-tensor-are-preequiv}, the identification map $Z_1\otimes_{\pi}Z_2\to Z_1\otimes_{\op{ex}}Z_2$ is not necessarily injective and so $\mathcal{B}_{\mathcal{M}}(Z_1,Z_2') \hookrightarrow \mathcal{B}(Z_1,Z_2')$ does not necessarily have weak-* dense range.

The next theorem sums up our findings. 
\begin{theorem}\label{thm: abstract-kernels}
Under the hypotheses of Theorem \ref{thm: proj-tensor-are-preequiv} there is the following commutative diagram
\begin{equation}\label{diag: inner-kernel}
\begin{tikzcd}[row sep=large, column sep=large] 
 & k(\mathcal{N}(\mathcal{Z}_1',\mathcal{Z}_2)) \arrow[r, hook]  & k(\mathcal{N}(Z_1',Z_2)) \arrow[r, hook] & \mathcal{M}(Z_1',Z_2) \arrow[r, hook] & \mathcal{HS}(\op{Loc}(Z_1)', \op{Loc}(Z_2))  \\
&  \mathcal{Z}_1\otimes_{\pi}\mathcal{Z}_2 \arrow[u, "\cong"] \arrow[r, hook] \arrow[d, hook] & Z_1\otimes_{\pi} Z_2 \arrow[u, "\cong"] \arrow[r, hook] \arrow[d, hook] & Z_1\otimes_{\op{ex}}Z_2 \arrow[r, hook] \arrow[u, "\cong"] & \op{Loc}(Z_1)\otimes_2 \op{Loc}(Z_2) \arrow[u, "\cong"] \\
&\mathcal{B}(\mathcal{Z}_1',\mathcal{Z}_2)  & \mathcal{B}(Z_1',Z_2) & 
\end{tikzcd}
\end{equation}
where all the arrows pointing to the right are embeddings with norm dense ranges. There is also the following commutative diagram for the duals
\begin{equation}\label{diag: outer-kernel}
\begin{tikzcd}[row sep=large, column sep=large]
&\mathcal{HS}(\op{Loc}(Z_1)',\op{Loc}(Z_2)) \arrow[r, hook] \arrow[d, "\cong"] & \mathcal{B}_{\mathcal{M}}(Z_1,Z_2') \arrow[r, hook] \arrow[d, "\cong"] & \mathcal{B}(Z_1,Z_2') \arrow[r, hook] \arrow[d, "\cong"] & \mathcal{B}(\mathcal{Z}_1,\mathcal{Z}_2') \arrow[d, "\cong"] \\
& \op{Loc}(Z_1)\otimes_2 \op{Loc}(Z_2)\arrow[r, hook] & (Z_1\otimes_{\op{ex}} Z_2)'\arrow[r, hook] & (Z_1\otimes_{\pi} Z_2)' \arrow[r, hook] & (\mathcal{Z}_1\otimes_{\pi} \mathcal{Z}_2)'
\end{tikzcd}
\end{equation}
where all the arrows pointing to the right are embeddings with weak-$*$ dense ranges.
\end{theorem}
\begin{proof}
   According to Theorem \ref{thm: proj-tensor-are-preequiv}, $Z_1\otimes_{\op{ex}}Z_2$ is faithfully localisable from $\mathscr{Z}_1\otimes_{\pi}\mathcal{Z}_2$. Theorem \ref{thm: gelfand-quintuple}, along with Corollary \ref{cor: proj-into-external}, give us the middle row of Diagram \eqref{diag: inner-kernel}. All embeddings have norm dense ranges. The bottom row of Diagram \eqref{diag: outer-kernel} are the formal adjoints of the embeddings of the aforementioned middle row and so are embeddings themselves and have weak-$*$ dense ranges. Going from the left to the right, the top row of Diagram \eqref{diag: inner-kernel} follows from the isomorphisms of Corollary \ref{cor: proj-goes-into-nuclear}, Proposition \ref{prop: ex-tensor-as-operator} and Theorem \ref{thm: localization-ext}, respectively. Similarly, the top row of Diagram \eqref{diag: outer-kernel} follows from the isomorphisms of Lemma \ref{lem: module-continuous-maps} and Equation \ref{form: duality-for-projective-tensor}. Finally, the embeddings $\mathcal{Z}_1\otimes_{\pi} \mathcal{Z}_2 \hookrightarrow \mathcal{B}(\mathcal{Z}_1',\mathcal{Z}_2)$ and $Z_1\otimes_{\pi} Z_2 \hookrightarrow \mathcal{B}(Z_1',Z_2)$ are result from the fact that both $\mathcal{Z}_1$ and $Z_1$ have the approximation property.    
\end{proof}
We refer to Commutative Diagrams \eqref{diag: inner-kernel} and \eqref{diag: outer-kernel} as the \emph{inner and outer kernel theorems} for the equivalence bimodules $Z_1$ and $Z_2$ that satisfy the hypotheses of Theorem \ref{thm: proj-tensor-are-preequiv}, respectively.

\section{Applications to Time-Frequency Analysis}\label{sec: tfa}
\subsection{Setting the Stage}\label{subsec: set-up}
In this section we apply our results to Banach spaces and Hilbert C*-modules that are of particular importance to time-frequency and Gabor analysis on locally compact abelian (LCA) groups. We shall assume that the reader is familiar with frame theory \cite{Gr01,Ch03, FrLa02, FrLa99}.

Let $G$ be a second countable locally compact abelian (LCA) group with a Haar measure $\mu_G$ and $\widehat{G}$ its dual with the dual measure $\mu_{\widehat{G}}$. The translation and modulation operators, defined by $T_x(f)(t)=f(t-x)$ and $(M_{\omega}f)(t)=\omega(t)f(t)$ for any $f\in L^2(G)$ respectively, are unitary. The time-frequency shift operator is defined by $\pi(x,\omega):=M_{\omega}T_x\in\mathcal{B}(L^2(G))$ for any $(x,\omega)\in G\times\widehat{G}$ and any $f\in L^2(G)$. The map $\pi: G\times\widehat{G} \to \mathcal{B}(L^2(G))$ is strongly continuous. The map $\pi^*: G\times\widehat{G} \to \mathcal{B}(L^2(G))$ is the strongly continuous unitary operator induced by the adjoint $\pi^*(x,\omega):= (\pi(x,\omega))^*$ of $\pi(x,\omega)$ for any $(x,\omega)\in G\times\widehat{G}$. Feichtinger introduced the algebra $\mathbf{S}_0(G)$ \cite{Fe81}, which proved extremely useful for time-frequency analysis \cite{Ja18, Be21-2, Sp07, SaSpSt10, De08-5}.
\begin{definition} For any LCA group $G$, the Feichtinger algebra is defined by
\begin{align*}
    \mathbf{S}_0(G) = \left\{f\in L^2(G): \int_{G\times \widehat{G}} |\rin{2}{f}{\pi(\chi)f}|d\mu_{G\times \widehat{G}}(\chi)<\infty \right\}. 
\end{align*}
\end{definition}
One can use any non-zero $\phi\in \mathbf{S}_0(G)$ to define the norm in $\mathbf{S}_0(G)$ by
\begin{align*}
    \|f\|_{\mathbf{S}_0(G), \, \phi}:= \int_{G\times\widehat{G}}|\rin{2}{f}{\pi(\lambda)g}|d\mu_{G\times\widehat{G}}(\chi), 
\end{align*}
different choices of $\phi \in \mathbf{S}_0(G)$ resulting in equivalent norms in $\mathbf{S}_0(G)$. This allows one to remove the reference to $\phi$ in $\|f\|_{\mathbf{S}_0(G), \, \phi}$ and write simply $\|f\|_{\mathbf{S}_0}$. We shall specify $\phi$ in $\|f\|_{\mathbf{S}_0(G),\phi}$ only when it is relevant and write $x\lesssim \|f\|_{\mathbf{S}_0}$ to indicate that for a certain window $\phi\in \mathbf{S}_0(G)$, there is such a $C>0$ that $x\leq C\|f\|_{\mathbf{S}_0(G), \, \phi}$. 

For any LCA group $H$ with a Haar measure $\mu_H$, the normalised continuous $2$-co-cycle is the continuous map $\psi: H\times H\to \mathbb{C}$ that satisfies the following conditions
\begin{enumerate}
    \item $\psi(h_1,h_2)\psi(h_2,h_3) = \psi(h_2,h_3)\psi(h_1,h_2+h_3)$, for any $h_1,h_2,h_3\in H$, and
    \item $\psi(0,h)=\psi(h,0)=1$ for any $h\in H$.
\end{enumerate}
Such a $2$-co-cycles can be used to define the so-called twisted $L^1$-space $L^1(H, \, \psi)$. We apply the same procedure to the Feichtinger algebra and obtain the following result.
\begin{proposition}\label{prop: twisted-feich}
    The Feichtinger algebra $\mathbf{S}_0(H)$ forms a Banach $*$-algebra  $\mathbf{S}_0(H,\varphi)$ when equipped with the $\psi$-twisted convolution and involution defined by
    \begin{equation}\label{form: twisted-conv-inv}
    \begin{split}
        (a_1*_{\psi}a_2)(x)&:= \int_{\Delta}a_1(y)a_2(x-y)\psi(y,x-y) d\mu_{H}(y), \\
    a^{*_{\psi}}(x) &= \overline{a(-x)} \psi(x,x).
    \end{split}
\end{equation}
for any $x\in H$ and any $a,a_1,a_2,a_2\in \mathbf{S}_0(H)$.
\end{proposition}
\begin{proof}
The proof is almost identical to that of \cite[Lemma 3.1]{JaLu21}. 
\end{proof}
We assumed $\psi$ to be continuous, rather than merely Borel as usually done in literature on C*-algebras \cite{PaRa92} for example, as the functions in $\mathbf{S}_0(H)$ are continuous anyway. The space $\mathbf{S}_0(H,\varphi)$ is embedded as a dense $*$-subalgebra in $L^1(H,\varphi)$ and so it is also a dense $*$-sub-algebra of the enveloping C*-algebra $C^*(H,\psi)=: C^*(L^1(H,\psi))$, called the $\psi$-\emph{twisted C*-algebra of} $H$. Here is how the Feichtinger algebra interacts with the projective tensor product. For LCA groups $G_1$, $G_2$ and functions $f_1:G_1\to \mathbb{C}$, $f_2: G_2\to \mathbb{C}$, one can define the map $E$ that makes any tensor $f_1\otimes f_2$ into the function $E(f_1\otimes f_2): G_1\times G_2\to \mathbb{T}$ defined by
\begin{align}\label{form: function-tensor}
    E(f_1\otimes f_2)(x,y)= f_1(x)f_2(y), \qquad (x,y)\in G_1\times G_2.
\end{align}
The map $E$ can be extended to a Banach space isomorphism of the projective tensor product of Feichtinger algebras \cite[Theorem 7.4]{Ja18}. We shall make use of the following fact.
\begin{proposition}\label{prop: s0-proj-tensor}
    For any $G_1,G_2$ LCA groups, $E: \mathbf{S}_0(G_1)\otimes_{\pi}\mathbf{S}_0(G_2) \xrightarrow{\sim} \mathbf{S}_0(G_1\times G_2)$ is a Banach space isomorphism.
\end{proposition}
\begin{corollary}\label{cor: spatial-tensor-of-twisted-gcstar}
    For any LCA groups $G_1$ and $G_2$ and any normalised continuous $2$-co-cycles $\psi_1: G_1\times G_1 \to \mathbb{T}$ and $\psi_2: G_2\times G_2\to \mathbb{T}$  on $G_1$ and $G_2$ respectively, the pointwise multiplication $\psi_1\cdot\psi_2: (G_1\times G_2)\times(G_1\times G_2) \to \mathbb{T}$ is a normalised $2$-co-cycles on $G_1\times G_2$ and the isomorphism $E$ in Proposition \ref{prop: s0-proj-tensor} is an twisted group C*-algebra isomorphism
    \begin{align*}
        E:C^*(G_1,\psi_1) \otimes_{\op{sp}} C^*(G_2,\psi_2) \xrightarrow[]{\sim} C^*(G_1\times G_2, \psi_1\cdot\psi_2).
    \end{align*}
\end{corollary}
\begin{proof}
As laid out in Remark \ref{rem: cstar-nuclearity}, the twisted abelian group C*-algebras are nuclear. The statement of the corollary follows from the fact that the isomorphism $E$ can be made into the Banach $*$-algebra isomorphism $E: \mathbf{S}_0(G_1,\psi_1)\otimes_{\pi}\mathbf{S}_0(G_2,\psi_2)\xrightarrow[]{\sim} \mathbf{S}_0(G_1\times G_2,\psi_1\cdot \psi_2)$ using Proposition \ref{prop: s0-proj-tensor}, Corollary \ref{cor: spatial-tensor-of-twisted-gcstar} and Proposition \ref{prop: twisted-feich}.
\end{proof}

The main objective of \emph{Gabor analysis} is to provide a joint time-frequency representation of signals in $L^2(G)$ in terms of time-frequency shifts. We consider a closed subgroup $\Delta\leq G\times \widehat{G}$, along with its quotient subgroup $(G\times\widehat{G})/\Delta$. Both of them are LCA groups and can be equipped with Haar measures $\mu_{\Delta}$ and $\mu_{(G\times\widehat{G})/\Delta}$ respectively. Moreover, these measures can be chosen so that the measures $\mu_{G\times\widehat{G}}$, $\mu_
{\Delta}$ and $\mu_{(G\times\widehat{G})/\Delta}$ are \emph{canonically related}, that is, the \emph{Weil formula} is satisfied
\begin{align*}
    \int_{G\times\widehat{G}}F(z)d\mu_{G\times\widehat{G}}(z) = \int_{(G\times\widehat{G})/\Delta}\int _{\Delta}F(z+\lambda)d\mu_{\Delta}(\lambda) d\mu_{(G\times\widehat{G})/\Delta}(\dot{z}), \qquad \dot{z}=z+\Delta
\end{align*}  for any $F\in L^1(G\times\widehat{G})$.
We shall also assume that $\Delta$ is \emph{co-compact}. This means that $(G\times\widehat{G})/\Delta$ is a compact group or, equivalently, that $\mu_{(G\times\widehat{G})/\Delta}$ is a finite measure and consequently. We define the size $|\Delta|$ of $\Delta$ by
\begin{align*}
    |\Delta|:= \int_{(G\times\widehat{G})/\Delta}1 d\mu_{(G\times\widehat{G})/\Delta} <\infty.
\end{align*} 
For the rest of the article, whenever we mention $\Delta$ as a closed co-compact subgroup of $G\times\widehat{G}$, we mean for $\Delta$ and all the related subgroups to have measures chosen so that everything written above holds.

We can now define the adjoint lattice of a closed subgroup.
\begin{definition}
For any closed subgroup $\Delta\subseteq G\times\widehat{G}$, its adjoint subgroup $\Delta^{\circ}$ is defined by
\begin{align}
    \Delta^{\circ} = \{\chi\in G\times\widehat{G}: \pi(\chi)\pi(\lambda) = \pi(\lambda) \pi(\chi), \ \text{for any $\lambda \in \Delta$} \}.
\end{align}
\end{definition}
\noindent Since $\Delta^{\circ}$ is identifiable with the annihilator $\Delta^{\perp}$, there is  the following characterization: $\Delta$ is co-compact if and only if $\Delta^{\circ}$ is discrete. From the fact from abstract harmonic analysis that $\Delta^{\perp}\cong ((G\times\widehat{G})/\Delta)^{\widehat{}}$ follows that the measure on $\Delta^{\circ}$ is the counting measure scaled by $\displaystyle\frac{1}{|\Delta|}$. 
\begin{definition}
    For any $g\in L^2(G)$ and any closed subgroup $\Delta\subseteq G\times\widehat{G}$, the \emph{Gabor system} $\mathcal{G}(g,\Delta)$ is defined by
\begin{align}
    \mathcal{G}(g,\Delta):= \{\pi(\lambda)g: \lambda\in \Delta\}.
\end{align}
 The frame operator $\displaystyle S_{g,\Delta}:f\mapsto \int_{\Delta}\rin{2}{f}{\pi(\lambda)g}\pi(\lambda)g \, d\mu_{\Delta}(\lambda)$ associated with $\mathcal{G}(g,\Delta)$ and called the \emph{Gabor frame operator} is continuous in $L^2(G)$. The Gabor system $\mathcal{G}(g,\Delta)$ is called a \emph{Gabor frame} if it forms a frame in $L^2(G)$. The Gabor systems with several windows $\{ g_i\}_{i=1}^n \subset L^2(G)$, where $n \in \mathbb{N}$, or the \emph{multi-window Gabor system} is defined by
 \begin{align}\label{form: multi-window-system}
     \bigcup_{i=1}^n\mathcal{G}(g_i,\Delta) .
 \end{align}
 If the multi-window Gabor system forms a frame in $L^2(G)$, it is called a \emph{multi-window Gabor frame}.
\end{definition}
\noindent If $\mathcal{G}(g,\Delta)$ forms a Bessel system in $L^2(G)$, i.e there is such a $C>0$, called the \emph{upper frame constant}, that
\begin{align}
    \int_{\Delta}|\rin{L^2(G)}{f}{\pi(\lambda)g}|^2 d\mu_{\Delta}(\lambda) \leq C \|f\|_{L^2(G)}^2, \qquad \forall f\in L^2(G), 
\end{align}
the Gabor frame operator 
\begin{align}\label{form: gabor-frame-op}
S_{g,\Delta}: f\mapsto \int_{\Delta}\rin{L^2(G)}{f}{\pi(\lambda)g}\pi(\lambda)g \, d\mu_{\Delta}(\lambda)    
\end{align}
is continuous in $L^2(G)$. If $\mathcal{G}(g,\Delta)$ is a frame in $L^2(G)$, $S_{g,\Delta}$ is invertible in $L^2(G)$. The relevance of the adjoint lattice reveals itself in providing one with a convenient framework for various duality results \cite{DaLaLa94, Ja94, AuJaMaLu20} in Gabor analysis. The most striking of them is the fact that $\mathcal{G}(g,\Delta)$ is a Gabor frame in $L^2(G)$ if and only if $\mathcal{G}(g,\Delta^{\circ})$ is a Riesz basis in $L^2(G)$.

For any $\chi_1= (x_1,\omega_1)$ and $\chi_2=(x_2,\omega_2)$ in $G\times\widehat{G}$, the \emph{Heisenberg $2$-co-cycle} $\varphi: (G\times\widehat{G})\times (G\times\widehat{G})\to \mathbb{T}$ is defined by
\begin{align}\label{form: heisenberg-2-cocycle}
    \varphi(\chi_1,\chi_2) = \overline{\omega_2(x_1)} .
\end{align}
As its name itself suggests, it is a normalised continuous $2$-co-cycle on the time-frequency plane $G\times\widehat{G}$ and so is its conjugate $\overline{\varphi}$. The restrictions $\varphi_{|\Delta\times\Delta}: \Delta\times \Delta \to \mathbb{T}$ and $\overline{\varphi}_{|\Delta^{\circ}\times \Delta^{\circ}}: \Delta^{\circ}\times \Delta^{\circ}\to \mathbb{T}$ are also normalised continuous $2$-co-cycles on $\Delta$ and $\Delta^{\circ}$ respectively and will be referred to as $\varphi$ and $\overline{\varphi}$.

Proposition \ref{prop: twisted-feich} gives rise to two non-commutative Banach $*$-algebras, namely $\mathbf{S}_0(\Delta,\varphi)$ and $\mathbf{S}_0(\Delta^{\circ},\overline{\varphi})$, and the Feichtinger algebra $\mathbf{S}_0(G)$ itself has a module structure over these two Banach $*$-algebras \cite{JaLu21}.
\begin{proposition}
    $\mathbf{S}_0(G)$ is a $\mathbf{S}_0(\Delta,\varphi)$-$\mathbf{S}_0(\Delta^{\circ},\overline{\varphi})$ pre-equivalence bimodule whose structure is defined by
\begin{equation}\label{form: heisen-structure}
\begin{alignedat}{2}
a\cdot f &:= \int_{\Delta}a(\lambda)\pi(\lambda)f d\mu_{\Delta}(\lambda), \qquad &\lin{\Delta}{f_1}{f_2}(\lambda) &:= \rin{2}{f_1}{\pi(\lambda)f_2} \\
f\cdot b &:= \int_{\Delta^{\circ}}b(\lambda^{\circ})\pi^*(\lambda^{\circ})d\mu_{\Delta^{\circ}}(\lambda^{\circ}), \qquad &\rin{\Delta^{\circ}}{f_1}{f_2}(\lambda^{\circ}) &:= \rin{2}{f_2}{\pi^*(\lambda^{\circ})f_1}
\end{alignedat}
\end{equation} for any $a\in \mathbf{S}_0(\Delta,\varphi)$, any $b\in\mathbf{S}_0(\Delta^{\circ},\overline{\varphi})$, any $f,f_1,f_2\in \mathbf{S}_0(G)$, any $\lambda\in \Delta$ and any $\lambda^{\circ}\in \Delta^{\circ}$.
\end{proposition}
\begin{definition}
    The completion of the Feichtinger algebra $\mathbf{S}_0(G)$ as a $\mathbf{S}_0(\Delta,\varphi)$-$\mathbf{S}_0(\Delta^{\circ},\overline{\varphi})$ pre-equivalence bimodule to an equivalence $C^*(\Delta,\varphi)$-$C^*(\Delta^{\circ},\overline{\varphi})$ bimodule is called the \emph{Heisenberg module} and denoted by $\mathcal{E}_{\Delta}(G).$ The module norm on $\mathcal{E}_{\Delta}(G)$ wil be denoted by $\|\cdot\|_{\mathcal{E}_{\Delta}(G)}$.
\end{definition}

We chose $\Delta$ to be a co-compact closed subgroup, which is equivalent to $\Delta^{\circ}$ being discrete. This implies that the $C^*$-algebra $C^*(\Delta^{\circ},\overline{\varphi})$ is unital \cite{JaLe16} or, equivalently, that $\mathcal{E}_{\Delta}(G)$ is finitely generated and projective as a $C^*(\Delta,\varphi)$ module. The unitality of the right C*-algebra $C^*(\Delta^{\circ},\overline{\varphi})$ allows us to densely define a finite faithful trace on it by
\begin{align}\label{form: right-trace}
    \op{tr}_{\Delta^{\circ}}(b) = b(0)
\end{align}
for any $b\in \mathbf{S}_0(\Delta^{\circ},\overline{\varphi})$  \cite{BeOm18}. We note that $\Delta^{\circ}$ is equipped with the counting measure scaled with $1/|\Delta|$ and therefore the unit element in $C^*(\Delta^{\circ},\varphi)$ is given by $|\Delta|\delta_0$ where $\delta_0: \Delta^{\circ}\to \mathbb{C}$ stands for the Dirac measure
\begin{align*}
    \delta_0(\lambda^{\circ}) = \begin{cases}
        1, \qquad\ &\text{if $\lambda^{\circ}=0$},\\
        0, \qquad &\text{otherwise}.
    \end{cases}
\end{align*}
Applying the trace to the unit $|\Delta|\delta_0$ results in $\op{tr}_{\Delta^{\circ}}(|\Delta|\delta_0)=\|\op{tr}_{\Delta^{\circ}}\|=|\Delta|$. We can use Theorem \ref{thm: localization} and define a lower semi-continuous faithful trace on $\op{span}\{\lin{\Delta}{f_1}{f_2}: f_1,f_2\in \mathbf{S}_0(G)\}$ by
\begin{align*}
    \op{tr}_{\Delta}(\lin{\Delta}{f_1}{f_2}) = \op{tr}_{\Delta^{\circ}}(\rin{\Delta^{\circ}}{f_2}{f_1}) = \rin{2}{f_1}{\pi^*(0)f_2} = \rin{2}{f_1}{f_2}.
\end{align*}
Hence the norm induced by the inner product, in its turn, induced by the semi-continuous faithful trace is nothing but the $L^2(G)$-norm on $\mathbf{S}_0(G)$, which can be extended to the whole of $\mathcal{E}_{\Delta}(G)$. The next theorem provides a clear understanding of $\mathcal{E}_{\Delta}(G)$ as a Banach space \cite{AuEn20}. 
\begin{theorem}\label{thm: localization-heis}
    The Heisenberg module $\mathcal{E}_{\Delta}(G)$ is faithfully localisable from $\mathbf{S}_0(G)$ and densely embedded in $L^2(G)$. More precisely, the localisation of $\mathcal{E}_{\Delta}(G)$  with respect to the trace $\op{tr}_{\Delta^{\circ}}$ results in $\op{Loc}(\mathcal{E}_{\Delta}(G))\cong L^2(G)$ and it can be obtained as a Banach space by completing $\mathbf{S}_0(G)$ as a subspace of $L^2(G)$ with respect to the norm
    \begin{align}
        \|f_0\|_{B_{\Delta}(G)} := \inf\{D^{1/2}: D \text{ is a Bessel bound for $\mathcal{G}(f_0,\Delta)$}\}, \qquad f_0\in \mathbf{S}_0(G).
    \end{align}
    Furthermore
\begin{enumerate}
    \item If $f\in \mathcal{E}_{\Delta}(G)$, $\mathcal{G}(f,\Delta)$ is also a Bessel family with the optimal Bessel bound $\|f\|_{B_{\Delta}(G)}^{2}$ and $\|f\|_{\mathcal{E}_{\Delta}(G)}=\|f\|_{B_{\Delta}(G)}$.
    \item The structure of $\mathbf{S}_0(G)$ defined by \eqref{form: heisen-structure} can be extended to $\mathcal{E}_{\Delta}(G)$ and is compatible with the embeddings $\mathbf{S}_0(G)\hookrightarrow \mathcal{E}_{\Delta}(G)\hookrightarrow L^2(G)$. In particular, 
    \begin{equation}\label{ineq: heis-ineqs}
        \begin{split}
            \|f\|_{L^2(G)}&\leq \sqrt{|\Delta|}\|f\|_{\mathcal{E}_{\Delta}(G)}, \qquad f\in \mathcal{E}_{\Delta}(G) \\
        \|f\|_{\mathcal{E}_{\Delta}(G)} &\leq \frac{1}{\|\phi\|_{L^2(G)}} \|f\|_{\mathbf{S}_0(G),\phi}, \qquad f\in \mathbf{S}_0(G),
        \end{split}
    \end{equation}
    for a window function $\phi\in \mathbf{S}_0(G)$.
    \item The elements $\xi_1,...,\xi_n\in \mathcal{E}_{\Delta}(G)$ are generators of $\mathcal{E}_{\Delta}(G)$ as a $C^*(\Delta,\varphi)$-module if and only if $\bigcup_{i=1}^n\mathcal{G}(\xi_i, \Delta)$ is a multi-window Gabor frame for $L^2(G).$
\end{enumerate}
\end{theorem}

The second estimate in \eqref{ineq: heis-ineqs} follows from \cite[Lemma 3.11]{GeLaLu26} adapted for LCA groups.

\begin{corollary} Let $\mathcal{E}_{\Delta}'(G)$ be the topological dual space of $\mathcal{E}_{\Delta}(G)$ as a Banach space. There are the following embeddings
    \begin{align}\label{form: heis-dual-embeds}
        \mathbf{S}_0(G)\hookrightarrow \mathcal{E}_{\Delta
        }(G)\hookrightarrow L^2(G)\hookrightarrow \mathcal{E}_{\Delta}'(G)\hookrightarrow \mathbf{S}_0'(G)
    \end{align} that
    satisfy
    \begin{equation}\label{ineq: heis-dual-ineqs}
        \begin{split}
            \|\sigma\|_{\mathbf{S}_0',\phi} &\leq  \frac{\|\sigma\|_{\mathcal{E}_{\Delta}'}(G)}{\|\phi\|_{L^2(G)}} \\
        \|h\|_{\mathcal{E}_{\Delta}'(G)} &\leq \sqrt{|\Delta|}\|h\|_{L^2(G)}
        \end{split}
    \end{equation}
    for any $h\in L^2(G)$ and any $\sigma \in \mathcal{E}_{\Delta}'(G).$ 
    Furthermore, $L^2(G)$ is weak-$*$ dense in $\mathcal{E}_{\Delta}'(G)$ and $\mathcal{E}_{\Delta}'(G)$ is weak-$*$ dense in $\mathbf{S}_0'(G)$.
\end{corollary}
\begin{proof}
    This is a direct Corollary of Lemma \ref{lem: localizable-is-gelfand-triple} and Theorem \ref{thm: gelfand-quintuple} applied to the Heisenberg modules.
\end{proof}

We end this subsection by proving that the Feichtinger algebra has the bounded approximation property. Considering $\mathbf{S}_0(G)$ as a co-orbit space and using the atomic decomposition theorem \cite{FeGr89,FeGr89-1, Be22} could possibly have given us a shorter proof of Theorem \ref{thm: feich-has-ap}. We are also aware that Feichtinger has recently shown that Banach spaces of tempered distributions on $\mathbb{R}^d$ with a double module structure, including $\mathbf{S}_0(\mathbb{R}^d)$, have the bounded approximation property and that he pointed out that these techniques, mutatis mutandis, should be applicable to LCA groups \cite{Fe22-1}. However, the approach we shall adopt will supply a clear interpretation of the vector-valued integrals that arise in this work and fits more naturally with the Hilbert C*-module-theoretic notions that we have introduced.

It usually suffices to make use of the weak interpretation of Hilbert-space-valued integrals that exploits the Riesz representation theorem. This so-called weak formulation bears the name of \emph{Pettis integral} \cite[Definition 2]{DiUh77}. It contrasts with the so-called strong formulation, also called the \emph{Bochner integral}, for vector-valued functions that is based on the use of  vector-valued simple functions for their approximation, rather than duality \cite[Definition 1]{DiUh77}.

The following theorem shows that, in our setting, there is no ambiguity about integrals of continuous functions.
\begin{lemma}\label{lem: int-no-ambiguity}
    Let $Q$ be a LCA group with a Haar measure $\mu_{Q}$, $Y$ a Banach space and $\mathcal{H}$ a separable Hilbert space such that there is the embedding $Y\hookrightarrow \mathcal{H}$ with a dense range. If $F:Q\to Y$ is a continuous function satisfying $\displaystyle \int_Q \|F(x)\|_Y d\mu_Q(x)<\infty$, then $F$ is a Bochner integrable with respect to $Y$ and $\mathcal{H}$ and also a Pettis integrable with respect to $Y$ and $\mathcal{H}$. These four integrals coincide with $\displaystyle \int_Q F(x)d\mu_Q(x)\in Y\hookrightarrow \mathcal{H}$.
\end{lemma}
\begin{proof}
 According to Lemma \ref{lem: localizable-is-gelfand-triple} and subsequent remarks, there is the Banach-Gelfand triple $X\hookrightarrow \mathcal{H} \hookrightarrow X'$. As $F$ is contituous, the map $x\mapsto \rin{X,X'}{F(x)}{x'}$ is continuous for any $x'\in X'$. This implies that $F$ is Pettis (weak) $\mu_Q$-measurable as an $X$-valued function and so as an $\mathcal{H}$-valued function too. Since $\mathcal{H}$ is separable, it follows from the Pettis measurability theorem \cite[Theorem 2]{DiUh77} that $F$ is also Bochner (strong) $\mu_Q$-measurable both as an $X$- and a $\mathcal{H}$-valued function. Continuity of $F$ also implies that $x\mapsto \|F(x)\|_Y$ is $\mu_Q$-measurable in the usual sense. The hypothesis $\displaystyle \int_Q \|F(x)\|_Y d\mu_Q(x)<\infty$ implies that $F$ is Bochner-integrable. We denote the Bochner integral simply by $\displaystyle \int_Q F(x)d\mu_Q(x)\in X$. According to \cite[Theorem 6]{DiUh77}, for any $T\in \mathcal{B}(Y,W)$ where $W$ stands for a Banach space, $T\circ F: Q\to W$ is Bochner-integrable and 
\begin{align}\label{form: linear-transform-bochner}
    T\left(\int_Q F(x)d\mu_Q(x)\right)=\int_Q(T\circ F)(x)d\mu_Q(x) \in W.
\end{align}
If we consider any linear functional $x'\in \mathcal{B}(Y, \mathbb{C})=Y'$, we infer from Equation \eqref{form: linear-transform-bochner} that $F$ is Pettis-integrable and, by uniqueness of the Pettis integral, that the Pettis and Bochner integrals of $F$ on $Q$ as an $X$-valued function must coincide. The embeddings $X\hookrightarrow \mathcal{H} \hookrightarrow X'$, in its turn, imply that the corresponding Pettis and Bochner integrals of $F$ as an $\mathcal{H}$-valued function must coincide too. 
\end{proof}
\begin{remark}\label{rem: without-h}
    This lemma remains true without the assumption that $\mathcal{H}$ is separable as long as $Y$ is separable. The lemma ensures that the weak interpretation of the integrals encountered in time-frequency analysis -- such as those used to define the module the structure of the Heisenberg modules in Equations \eqref{form: heisen-structure} -- behave well with respect to the embedding $\mathbf{S}_0(G)\hookrightarrow L^2(G)$.
\end{remark}

Knowing that Bochner and Pettis integrals coincide for continuous vector-valued integrable functions allows one to use a specific  approximation of the integrals by simple functions without having to ask oneself what integral one is dealing with. In particular, one can use the so-called generalised Riemann integral. The following lemma defines the generalised Riemann integral \cite{Va22}.
\begin{lemma}\label{lem: riemann-integral}
    Let $K$ be a compact metric space $K$ with a finite Borel measure $\mu_K$, $Y$ a Banach space and the map $F:K\to Y$ continuous. Then there is an element $\displaystyle (R)\int_K F(x)d\mu_K(x)\in Y$ with the following property. For any $\varepsilon>0$, there is such a $\delta>0$ that, for any partition $\mathcal{K}=\{K_1,...,K_N\}$ of $K$ with a mesh smaller than $\delta$ and any such sample points $\{t_1,...,t_N\}$ that $t_i\in K_i$ for any $i\in \{1,...,n\}$,
    \begin{align}\label{form: riemann-approx}
        \left\|(R)\int_K F(x)d\mu_K(x)-\sum_{i=1}^N F(t_i)\mu_K(K_i) \right\|_Y<\varepsilon.
    \end{align}
\end{lemma}
\begin{corollary}\label{cor: lca-riemann-approx}
    Let $Q$ be a second countable LCA group with a Haar measure $\mu_Q$, $Y$ a Banach space and $F: Q\to Y$ is such a continuous function that $\displaystyle \int_Q\|F(x)\|_Yd\mu_Q(x)<\infty$. Then there exists the Bochner integral $\displaystyle \int_QF(x)d\mu_Q(x)\in Y$. Furthermore, for any $\varepsilon>0$, there is a compact subset $K\subseteq Q$, a partition $\mathcal{K}=\{K_1,...,K_n\}$ of $K$ and sample points $\{t_1,...,t_N\}\subseteq K$ such that
    \begin{align}\label{form: bochner-is-riemann-approximable}
        \left\|\int_QF(x)d\mu_Q(x)-\sum_{i=1}^n F(t_i)\mu_{Q}(K_i) \right\|_Y<\varepsilon
    \end{align}
\end{corollary}
\begin{proof}
    The hypotheses on $F$ implies that $F$ is Bochner integrable and so the Bochner integral $\displaystyle \int_Q F(x)d\mu_Q(x)\in Y$ exists. Since $Q$ is second countable, there is such an increasing sequence $\{K_n\}_{n\in \mathbb{N}}$ of compact subsets $K_n \subseteq Q$ that $\bigcup_{n\in \mathbb{N}}K_n = Q$. From the Lebesgue dominated convergence theorem for Bochner integrals \cite[Theorem 3]{DiUh77} we infer that
    \begin{align}\label{form: half-estimate}
        \int_{K_n} F(x)d\mu_{K_n}(x)\to \int_{Q}F(x)d\mu_Q(x)
    \end{align} in $Y$ as $n\to \infty$. Here $\mu_{K_n}$ is the restriction of $\mu_Q$ to the compact Borel subset $K_n$.
    
    On the other hand, $Q$ is a second countable LCA group and so $T_0$ with a countable open basis at the identity. This implies that $Q$ is metrisable \cite[Theorem 8.3]{HeRo79}, say, with the metric $d$. We note that, for any compact subset $K\subseteq Q$, $K$ is a compact metric space with a finite Borel measure $\mu_K$. The restriction $F_{|K}: K\to Y$ is a continuous function and so Lemma \ref{lem: riemann-integral} applies to $F_{|K}.$ Since $K$ is a compact subset, for any $\varepsilon>0$, there is such a $\delta>0$ that, for any partition $\mathcal{K}$ of $K$ with a mesh, measured by $d$, smaller than $\delta>0$ and any such sample points $\{t_1,...,t_N\}$ that $t_i\in K_i$ for any $i\in \{1,...,n\}$, \eqref{form: riemann-approx} applies to $F_{|K}.$ According to Lemma \ref{lem: riemann-integral}, the element $(R) \displaystyle \int_K F_{|K}(x)d\mu_K(x) \in Y$ generated in this way is the Bochner integral $\displaystyle \int_{K}F(x)d\mu_K(x)$. Combining \eqref{form: half-estimate} and \eqref{form: riemann-approx} applied to $ \displaystyle \int_{K_n}F(x)d\mu_{K_n}(x)$ completes the proof. 
\end{proof}
\begin{theorem}\label{thm: feich-has-ap}
    For any second countable LCA group $G$, the Feichtinger algebra $\mathbf{S}_0(G)$ has the approximation property.
\end{theorem}
\begin{proof}
    We choose a closed co-compact subgroup $\Delta \leq G\times \widehat{G}$. Since $\ell^1(\Delta^{\circ},\overline{\varphi}_1)\cong \mathbf{S}_0(\Delta^{\circ}, \overline{\varphi})$ is inverse-closed in $C^*(\Delta^{\circ},\overline{\varphi})$ \cite{GrLe04}, it follows from Proposition \ref{prop: pre-parseval-module-frame} that there are such $\{\mathscr{f}_1,...,\mathscr{f}_n\}$ and $\{\mathscr{g}_1,...,\mathscr{g}_n\}$ in $\mathbf{S}_0(G)$ that $\displaystyle \sum_{i=1}^n\rin{\Delta^{\circ}}{\mathscr{f}_i}{\mathscr{g}_i}= \frac{1}{|\Delta|}\delta_0$ or equivalently that $$\mathscr{f} = \sum_{i=1}^n\int_{\Delta} \rin{L^2(G)}{\mathscr{f}}{\pi(\lambda)\mathscr{f}_i}\pi(\lambda)\mathscr{g}_id\mu_{\Delta}(\lambda)$$ for any $\mathscr{f}\in \mathbf{S}_0(G)$.
    For the sake of simplicity, one can assume without loss of generality that $n=1$ and define the identity map $I: \mathbf{S}_0(G)\to \mathbf{S}_0(G)$ by
    \begin{align*}
        I(\mathscr{f}) = \int_{\Delta}\rin{L^2(G)}{\mathscr{f}}{\pi(\lambda)\mathscr{f}_1}\pi(\lambda)\mathscr{g}_1.
    \end{align*}
    Since $G$ is second countable, $L^2(G)$ is separable and so is $\mathbf{S}_0(G)$. According to Lemma \ref{lem: int-no-ambiguity}, $I(f)$ can be understood as a Bochner integral. Due to the properties of the Feichtinger algebra \cite[Lemma 4.15, Corollary 4.2 (iii) and Theorem 5.3(ii)]{Ja18}, the map $$\Delta\ni \lambda \mapsto \rin{L^2(G)}{\mathscr{f}}{\pi(\lambda)\mathscr{f}_1}\pi(\lambda)\mathscr{g}_2\in \mathbf{S}_0(G)$$ is continuous for any $\mathscr{f},\mathscr{f}_1,\mathscr{g}_1\in \mathbf{S}_0(G)$. Therefore, this map is Bochner integrable and
    \begin{align}\label{form: for-uniform-bound}
        \|I(\mathscr{f})\|_{\mathbf{S}_0}\lesssim \|\mathscr{g}_1\|_{\mathbf{S}_0}  \int_\Delta |\rin{L^2(G)}{\mathscr{f}}{\pi(\lambda)\mathscr{f}_1}|d\mu_{\Delta}(\lambda)<\infty.
    \end{align}
  As $\Delta$ is a second countable LCA group, it follows from Corollary \ref{cor: lca-riemann-approx} that, for any $N\in \mathbb{N}_{>0}$, there is a compact subset $\Delta_N\subseteq \Delta$, a partition $\{\Delta_N^{(1)},...,\Delta_N^{(n)}\}$ of $\Delta_N$ and sample points $\{t_N^{(1)},...,t_N^{(n)}\}\subseteq \Delta_N$ such that
    \begin{align*}
        \left\|I(\mathscr{f})-\sum_{i=1}^n \rin{L^2(G)}{\mathscr{f}}{\mu_{\Delta}(\Delta_N^{(i)}) \pi(t_N^{(i)})\mathscr{f}_1}\pi(t_N^{(i)})\mathscr{g}_1 \right\|_{\mathbf{S}_0}<\frac{1}{N}.
    \end{align*}
    We note that all $I_N$ are finite-rank operators in $\mathbf{S}_0(G)$ with a uniform bound. From \eqref{form: for-uniform-bound} we infer that
    \begin{align*}
        \sup_{N\in \mathbb{N}_{>0}}\|I_N\|\leq \sup_{N\in \mathbb{N}_{>0}}\sup_{\|\mathscr{f}\|_{\mathbf{S}_0}\leq 1} \|I(\mathscr{f})\|_{\mathbf{S}_0} \lesssim \|\mathscr{g}_1\|_{\mathbf{S}_0}\cdot \|\mathscr{f}\|_{\mathbf{S}_0} .
    \end{align*}
    Thus we have found a sequence of such uniformly bounded finite-rank operators in $\mathbf{S}_0(G)$ that, for any $\mathscr{f}\in \mathbf{S}_0(G)$, $\|I(\mathscr{f})-I_N(\mathscr{f})\|_{\mathbf{S}_0}=\|\mathscr{f}-I_N(\mathscr{f})\|_{\mathbf{S}_0}\to 0$ as $N\to \infty$. Invoking \cite[Proposition 4.3]{Ry02} completes the proof.
\end{proof}
\subsection{Kernel Theorems for the Heisenberg Modules} 
In this section we investigate the tensor products of the Heisenberg modules. For $i=1,2$, let LCA groups $G_i$ and $\Delta_i\leq G_i\times\widehat{G}_i$, and $\Delta_i \leq G_i\times\widehat{G}_i$ have the properties laid out in Section \ref{subsec: set-up}. We shall denote the Heisenberg $2$-co-cycle on $G_i\times \widehat{G}_i$ by $\varphi_i: (G_i\times \widehat{G}_i)\times (G_i\times \widehat{G}_i)\to \mathbb{T}$ and use the abbreviation $G_{12}:=G_1\times G_2$. According to Remark \ref{rem: cstar-nuclearity}, for any $i=1, 2$, $C^*(\Delta_i^{\circ},\overline{\varphi}_i)$ is a nuclear unital $C^*$-algebra. Therefore $\mathcal{E}_{\Delta_i}(G_i)$ is a finitely generated projective $C^*(\Delta_i,\varphi_i)$-module that is faithfully localisable from $\mathbf{S}_0(G_i)$. Since, according to Theorem \ref{thm: feich-has-ap}, $\mathbf{S}_0(G_i)$ has the approximation property, the hypotheses of Theorem \ref{thm: proj-tensor-are-preequiv} are satisfied by both the equivalence bimodules $\mathcal{E}_{\Delta_1}(G_1)$ and $\mathcal{E}_{\Delta_2}(G_2)$. From this we derive the abstract kernel theorems of Corollary \ref{thm: abstract-kernels} for the external tensor product $\mathcal{E}_{\Delta_1}(G_1)\otimes_{\op{ex}}\mathcal{E}_{\Delta_2}(G_2)$. However, we shall give a more concrete description to our kernel theorems, similar to those of the Schwartz kernel theorems. First of all we obtain the following result.
\begin{corollary}
    The map $E$ of \eqref{form: function-tensor} defines the embedding of $\mathcal{E}_{\Delta_1}(G_1)\otimes_{\op{ex}} \mathcal{E}_{\Delta_2}(G_2)$ in $L^2(G_{12})$ with a dense range that satisfies
    \begin{align}\label{form: heis-tensor-embedding}
    \|E(\mathbf{f})\|_{L^2(G_{12})} \leq \sqrt{|\Delta_1|} \sqrt{|\Delta_2|} \|\mathbf{f}\|_{\op{ex}},
    \end{align}
    for any $\mathbf{f}\in \mathcal{E}_{\Delta_1}(G_1) \otimes_{\op{ex}}\mathcal{E}_{\Delta_2}(G_2)$.
\end{corollary}
\begin{proof}
    From Theorem \ref{thm: localization-ext} we know that $$\op{Loc}(\mathcal{E}_{\Delta_1}(G_1)\otimes_{\op{ex}}\mathcal{E}_{\Delta_2}(G_2))\cong \op{Loc}(\mathcal{E}_{\Delta_1}(G_1))\otimes_2 \op{Loc}(\mathcal{E}_{\Delta_2}(G_2))\cong L^2(G_1)\otimes_{2} L^2(G_2).$$ 
    We also know that the map $E$ of Equation \eqref{form: function-tensor} implements the unitary map $E: L^2(G_1)\otimes_2 L^2(G_2)\xrightarrow[]{\sim} L^2(G_{12})$. Therefore, an application of the estimates of Theorem \ref{thm: localization-ext} to the Heisenberg modules with the traces satisfying $\|\op{tr}_{\Delta_i^{\circ}}\|=|\Delta_i|$ implies that 
    \begin{align*}
        \|E(\mathbf{f})\|_{L^2(G_{12})}=\|\mathbf{f}\|_{L^2(G_1)\otimes_{2} L^2(G_2)}\leq \sqrt{|\Delta_1|} \sqrt{|\Delta_2|} \|\mathbf{f}\|_{\op{ex}}. 
    \end{align*}
    for any $\mathbf{f}\in \mathcal{E}_{\Delta_1}(G_1)\otimes_{\op{ex}}\mathcal{E}_{\Delta_2}(G_2)$. 
\end{proof}
While it is good to know that the tensors in $\mathcal{E}_{\Delta_1}(G_1)\otimes_{\op{ex}} \mathcal{E}_{\Delta_2}(G_2)$ can be identified as elements of $L^2(G_{12})$, it is desirable for our purposes to find a more concrete  description of the external tensor product of Heisenberg modules as a function space. To do so, we identify the time-frequency plane $G_{12}\times \widehat{G}_{12}$ with $G_1\times\widehat{G}_1 \times G_2 \times \widehat{G}_2$ using the coordinate transformation $A: G_1\times \widehat{G}_1 \times G_2\times \widehat{G}_2\to G_{12}\times \widehat{G}_{12}$
\begin{align} \label{form: top-transform-1}
    A (x,\omega, y, \nu) = (x,y,\omega, \nu)
\end{align}
for any $(x,\omega,y,\nu)\in G_{12}\times \widehat{G}_{12}$.
 Since the transformation only swaps the second and the third coordinates in $G_1\times \widehat{G}_1\times G_2 \times \widehat{G}_2$, applying the Fubini theorem results in the following proposition.
\begin{proposition}\label{prop: preserving-transform} The coordinate transformation 
    $A: G_1\times \widehat{G}_1\times G_2\times \widehat{G}_2\to G_{12}\times \widehat{G}_{12}$ is a measure preserving topological group isomorphism.
\end{proposition}
Using the transformaion $A$, we can embed the product subgroup $\Delta_1 \times \Delta_2$ as a closed subgroup in $G_{12} \times \widehat{G}_{12}$. Simple computations involving the transformation $A$ lead to the following result.
\begin{lemma}\label{lem: simple-tensor-tfs}
    For any $\lambda,\lambda^{(1)},\lambda^{(2)} \in G_1\times \widehat{G}_1$, any $\mu,\mu^{(1)},\mu^{(2)} \in G_2\times \widehat{G}_2$, any $f\in L^2(G_1)$ and any $g\in L^2(G_2)$,
    \begin{align}\label{form: tfs-a1}
        \pi(A(\lambda,\mu))(E(f_1\otimes f_2)) = E(\pi(\lambda)f_1 \otimes \pi(\mu)f_2).
    \end{align}
    Furthermore, if $\varphi_{12}:(G_{12} \times \widehat{G}_{12})\times (G_{12}\times \widehat{G}_{12})\to \mathbb{T}$ is the Heisenberg $2$-co-cycle on $G_{12}\times \widehat{G}_{12}$,
    \begin{align}\label{form: g12-to-g1-cocycle}
        \varphi_{12}(A(\lambda^{(1)},\mu^{(1)}),A(\lambda^{(2)},\mu^{(2)})) = \varphi_{1}(\lambda^{(1)},\lambda^{(2)})\cdot \varphi_{2}(\mu^{(1)},\mu^{(2)}).
    \end{align} 
\end{lemma}
\begin{proof}
   For any $(x,y)\in G_{12}$, any $\lambda = (x', \omega')\in G_1\times \widehat{G}_1$ and any $\mu=(y',\nu')\in G_2\times \widehat{G}_2$, 
\begin{align*}
    \pi(A(\lambda,\mu))(E(f\otimes g))(x,y) &=E(f\otimes g)(x-x',y-y')\omega'(x)\nu'(y)\\
    &=f(x-x')\omega'(x)\cdot g(y-y')\nu'(y)\\
    &= E(\pi(\lambda)f\otimes \pi(\mu)g)(x,y).
\end{align*}
This proves Equation \eqref{form: tfs-a1}. On the other hand, writing the Heisenberg $2$-co-cycle on $\varphi_{12}:(G_{12}\times \widehat{G}_{12})\times (G_{12}\times \widehat{G}_{12})\to \mathbb{C}$ explicitly  results in
\begin{align*}
    \varphi_{12}((x_1,y_1,\omega_1,\nu_1),(x_2,y_2,\omega_2,\nu_2)) &= \overline{(\omega_2,\nu_2)}(x_1,y_1)\\
    &=\overline{\omega_2(x_1)}\cdot \overline{\nu_2(y_1)}\\
    &= \varphi_{1}((x_1,\omega_1),(x_2,\omega_2))\cdot \varphi_{2}((y_1,\nu_1),(y_2,\nu_2)).
\end{align*}
Setting $\lambda^{(1)}=(x_1,\omega_1)$, $\lambda^{(2)}=(x_2,\omega_2)$, $\mu^{(1)}=(y_1,\nu_1)$ and $\mu^{(2)}= (y_2,\nu_2)$ results in Equation \eqref{form: g12-to-g1-cocycle}.
\end{proof}
In general, one can consider any co-compact closed subgroup $\Delta_{12}\subseteq G_1\times \widehat{G}_1 \times G_2 \times \widehat{G}_2$. Then $A(\Delta_{12})$ is also a closed co-compact subgroup in $G_{12}\times \widehat{G}_{12}$. Since $A$ is measure-preserving,
\begin{align}\label{form: size-comparison}
    |\Delta_{12}| = |A(\Delta_{12})|.
\end{align}
We now consider the Heisenberg module $\mathcal{E}_{A(\Delta_{12})}(G_{12})$. From the projective tensor product characterisation of the Feichtinger algebra we derive the following lemma.
\begin{lemma}\label{lem: s0-tensor-approx}
For any subgroup $\Delta_{12} \subseteq G_1\times \widehat{G}_1\times G_2\times \widehat{G}_2$, $$\mathcal{E}_{A(\Delta)}(G_{12}) = \overline{\op{span}}\{E(f_1\otimes f_2): f_1\in \mathbf{S}_0(G_1),f_2\in \mathbf{S}_0(G_2)\}$$ where the closure is taken with respect to the $\mathcal{E}_{A(\Delta)}(G_{12})$-norm.
\end{lemma}
\begin{proof}
    The result follows from the fact that $\mathbf{S}_0(G_{12})$ is continuously and densely embedded in $\mathcal{E}_{A(\Delta)}(G_{12})$ and that the Feichtinger algebra can be factorised using the projective tensor product $\mathbf{S}_0(G_{12}) \cong \mathbf{S}_0(G_1)\otimes_{\pi}\mathbf{S}_0(G_2)$ according to Proposition \ref{prop: s0-proj-tensor}.
\end{proof}
\begin{theorem}\label{thm: concrete-tensor}
    For any $f\in \mathcal{E}_{\Delta_1}(G_1)$ and any $g\in \mathcal{E}_{\Delta_2}(G_2)$, $E(f\otimes g) \in \mathcal{E}_{A(\Delta_1\times \Delta_2)}(G_{12})$. Furthermore,
    \begin{align}\label{form: span-heis-tensor}
        \mathcal{E}_{A(\Delta_1\times \Delta_2)}(G_{12})= \overline{\op{span}}\{E(f\otimes g):f\in \mathcal{E}_{\Delta_1}(G_1),g\in \mathcal{E}_{\Delta_2}(G_2)\},
    \end{align}
    where the closure is taken with respect to the $\mathcal{E}_{A(\Delta_1\times \Delta_2)}(G_{12})$-norm.
\end{theorem}
\begin{proof}
    As $f\in \mathcal{E}_{\Delta_1}(G_1)$ and $g\in \mathcal{E}_{\Delta_2}(G_2)$, the families $\mathcal{G}(f,\Delta_1)$ and $\mathcal{G}(g,\Delta_2)$ are continuous Bessel systems in $L^2(G_1)$ and $L^2(G_2)$ with the Bessel bounds $\|f\|_{\mathcal{E}_{\Delta_1}(G_1)}^2$ and $\|g\|_{\mathcal{E}_{\Delta_2}(G_2)}^2$ respectively, according to Item (1) of Theorem \ref{thm: localization-heis}. We know from \cite[Theorem 3.3]{BaTe22} that
    \begin{align}\label{eq: tensor-bessel}
        \{\pi (\lambda)f\otimes \pi(\mu)g :(\lambda,\mu)\in \Delta_1\times \Delta_2 \}    
    \end{align}
    is a continuous Bessel system in $L^2(G_1)\otimes_2 L^2(G_2)$ with the Bessel bound  $\|f\|_{\mathcal{E}_{\Delta_1}(G_1)}^2\|g\|_{\mathcal{E}_{\Delta_2}(G_2)}^2$. From the isomorphism $E$ and Lemma \ref{lem: simple-tensor-tfs} we infer that the system defined by \eqref{eq: tensor-bessel} is nothing but $\mathcal{G}(E(f\otimes g), A(\Delta_1\times \Delta_2))$ in $L^2(G_{12})$. Therefore,
    \begin{align}\label{form: tensor-estimate}
        \|E(f\otimes g)\|_{\mathcal{E}_{A(\Delta_1 \times \Delta_2)}(G_{12})} \leq \|f\|_{\mathcal{E}_{\Delta_1}(G_1)}\cdot \|g\|_{\mathcal{E}_{\Delta_2}(G_2)}.
    \end{align}
    On the other hand, due to the density properties of the Feichtinger algebra, for any $\varepsilon>0$, there are such $\displaystyle f_0\in \mathbf{S_0}(G_1)$ and $g_0\in \mathbf{S}_0(G_2)$ that $\displaystyle \|f-f_0\|_{\mathcal{E}_{\Delta_1(G_1)}}<\frac{\varepsilon}{2\|g\|_{\mathcal{E}_{\Delta_2}(G_2)}}$ and $\|g-g_0\|_{\mathcal{E}_{\Delta_2}(G_2)}<\frac{\varepsilon}{2\|f_0\|_{\mathcal{E}_{\Delta_1}(G_1)}}.$ From  \eqref{form: tensor-estimate} follows that
    \begin{align*}
        \|E(f\otimes g -f_0\otimes g_0)\|_{\mathcal{E}_{A(\Delta\times \Delta_2)}(G_{12})}&\leq \|E((f-f_0)\otimes g)\|_{\mathcal{E}_{A(\Delta\times \Delta_2)}(G_{12})} + \|E(f_0\otimes (g-g_0))\|_{\mathcal{E}_{A(\Delta\times \Delta_2)}(G_{12})} \\
        &\leq \|f-f_0\|_{\mathcal{E}_{\Delta_1}(G_1)}\|g\|_{\mathcal{E}_{\Delta_2}(G_2)} + \|f_0\|_{\mathcal{E}_{\Delta_1}(G_1)}\|g-g_0\|_{\mathcal{E}_{\Delta_2}(G_2)} \\
        &< \varepsilon.
    \end{align*}
    Hence, $E(f\otimes g)$ can be arbitrarily well approximated by a function $E(f_0\otimes g_0) \in \mathbf{S}_0(G_{12})$ with respect to the $\mathcal{E}_{A(\Delta_1\times \Delta_2)}(G_{12})$-norm and so $E(f\otimes g) \in \mathcal{E}_{A(\Delta_1\times \Delta_2)}(G_{12})$. Equality \eqref{form: span-heis-tensor} follows from the dense embeddings $\mathbf{S}_0(G_i)\hookrightarrow \mathcal{E}_{\Delta_i}(G_i)$ and Lemma \ref{lem: s0-tensor-approx}.
\end{proof}

We now take a look at the C*-algebras associated with $\Delta_1\times \Delta_2$.
\begin{proposition}\label{prop: twisted-isomorphism}
    There is the $*$-isomorphism $D:C^*(\Delta_1,\varphi_1)\otimes_{\op{sp}}C^*(\Delta_2,\varphi_2)\xrightarrow[]{\sim}C^*(A(\Delta_1\times \Delta_2), \varphi_{12}).$
\end{proposition}
\begin{proof}
In the light of Corollary \ref{cor: spatial-tensor-of-twisted-gcstar}, it suffices to establish the isomorphism $V: C^*(\Delta_1\times \Delta_2,\varphi_1\cdot \varphi_2)\xrightarrow[]{\sim} C^*(A(\Delta_1\times \Delta_2), \varphi_{12}).$ This can be done at the Feichtinger algebra level. For $F\in \mathbf{S}_0(\Delta_1\times \Delta_2)$, we define $VF\in \mathbf{S}_0(A(\Delta_1\times \Delta_2))$ by
\begin{align*}
    (VF)(A(\lambda,\mu)) = F(\lambda,\mu), \qquad \forall (\lambda,\mu)\in \Delta_1\times \Delta_2.
\end{align*}
Due to the fact that $A$ is measure-preserving according to Proposition \ref{prop: preserving-transform},
\begin{align*}
    \rin{L^2(A(\Delta_1\times \Delta_2))}{V(F)}{V(G)} = \rin{L^2(\Delta_1\times \Delta_2)}{F}{G}
\end{align*}
for any $F$ and $G\in \mathbf{S}_0(\Delta_1\times \Delta_2)$ and can be extended to a unitary map on $L^2$. We note that $A(\Delta_1\times \Delta_2)\cong \Delta_1\times \Delta_2$ and $\widehat{\Delta_1}\times \widehat{\Delta_2}\cong \widehat{\Delta_1\times \Delta_2}\cong \widehat{A(\Delta_1\times \Delta_2)}$, using the map $\widehat{A}: \widehat{\Delta_1}\times\widehat{\Delta_2}\to \widehat{A(\Delta_1\times \Delta_2)}$ defined by $\widehat{A}(\hat{\lambda},\hat{\mu})(A(\lambda,\mu))= \hat{\lambda}(\lambda)\cdot \hat{\mu}(\mu)$ for any $(\lambda,\hat{\lambda})\in \Delta_1\times \widehat{\Delta_1}$ and any $(\mu,\hat{\mu})\in \Delta_2\times \widehat{\Delta_2}$. We now note that, for $(z,w)\in \Delta_1\times \Delta_2$,
\begin{align*}
    \pi(A(\lambda,\mu), \widehat{A}(\hat{\lambda},\hat{\mu}))(VF)(A(z,w)) &= F(z-\lambda,w-\mu) \cdot \hat{\lambda}(\lambda)\cdot \hat{\mu}(\mu)\\
    &= ((V \circ \pi((\lambda,\mu),(\hat{\lambda},\hat{\mu}))F)(A(z,w)) .
\end{align*}
From \cite[Theorem 5.1]{Ja18} it follows that $V$ is restricted to the isometric Banach space isomorphism $V:\mathbf{S}_0(\Delta_1\times \Delta_2)\to \mathbf{S}_0(A(\Delta_1\times \Delta_2))$ with
\begin{align*}
    \|VF\|_{\mathbf{S}_0(A(\Delta_1,\Delta_2)), (V^*)^{-1}\Phi} = \|F\|_{\mathbf{S}_0(\Delta_1\times\Delta_2, \Phi)} 
\end{align*}
for any $\Phi\in \mathbf{S}_0(\Delta_1\times \Delta_2)$. We check that this map can indeed be extended to a $*$-homomorphism with respect to the Banach $*$-algebra structure of Proposition \ref{prop: twisted-feich}. For $F\in \mathbf{S}_0(\Delta_1\times \Delta_2)$ and $(\lambda,\mu)\in \Delta_1\times \Delta_2$,
\begin{align*}
    (VF^{*_{\varphi_1\cdot\varphi_2}})(A(\lambda,\mu)) &= F^{*_{\varphi_1\cdot\varphi_2}}(\lambda,\mu) = \overline{F(-\lambda,-\mu)} (\varphi_1\cdot\varphi_2)((\lambda,\mu),(\lambda,\mu) \\
    &\stackrel{\text{Eq. \eqref{form: g12-to-g1-cocycle}}}{=} \overline{F(-\lambda,-\mu)} (\varphi_1\cdot\varphi_2) \varphi_{12}(A(\lambda,\mu)) = (VF)^{*_{\varphi_{12}}}(A(\lambda,\mu)).
\end{align*}
On the other hand, from Equation \eqref{form: g12-to-g1-cocycle} and Proposition \ref{prop: preserving-transform} we infer that, for $G\in \mathbf{S}_0(\Delta_1\times \Delta_2)$,
\begin{align*}
    &V(F*_{\varphi_1\cdot \varphi_2}G)(A(\lambda,\mu))= \int_{\Delta_1\times \Delta_2}F(z,w)G(\lambda-z,\mu-w)(\varphi_1\cdot \varphi_2)((z,w),(\lambda-z,\mu-w))d_{\Delta_1\times \Delta_2}(z,w) \\ 
    &=  \int_{A(\Delta_1\times \Delta_2)}(VF)(A(z,w))(VG)(A(\lambda-z,\mu-w)))
    \varphi_{12}(A(z,w),A(\lambda-z,\mu-w)) d_{A(\Delta_1\times \Delta_2)}(A(z,w)) \\
    &= (V(F)*_{\varphi_{12}}V(G))(A(\lambda,\mu)).
\end{align*}
Therefore $V: \mathbf{S}_0(\Delta_1\times \Delta_2,\varphi_1\cdot \varphi_2)\to \mathbf{S}_0(\Delta_1\times \Delta_2,\varphi_{12})$ is indeed a Banach $*$-algebra isomorphism and can be extended to $V: C^*(\Delta_1\times \Delta_2, \varphi_1\cdot \varphi_2)\to C^*(A(\Delta_1\times \Delta_2),\varphi_{12})$. We define $D:= V\circ E$ where $E$ is the isomorphism in Corollary \ref{cor: spatial-tensor-of-twisted-gcstar}. 
\end{proof}
$\mathcal{E}_{\Delta_1}(G_1)\otimes_{\op{ex}}\mathcal{E}_{\Delta_2}(G_2)$ is a Hilbert $C^*(\Delta_1,\varphi_1)\otimes_{\op{sp}}C^*(\Delta_2,\varphi_2)$-module, while $\mathcal{E}_{A(\Delta_1\times \Delta_2)}(G_{12})$ is a Hilbert $C^*(A(\Delta_1\times\Delta_2),\varphi_{12})$-module. The following theorem relates these them and shows that the external tensor product of Heisenberg modules is also a Heisenberg module.
\begin{theorem}\label{thm: concrete-external-tensor}
    The map \eqref{form: function-tensor} that identifies the elementary tensors in $\mathcal{E}_{\Delta_1}(G_1)\otimes_{\op{ex}}\mathcal{E}_{\Delta_2}(G_2)$ as functions in $L^2(G_{12})$ can be extended to the Hilbert C*-module isomorphism $E: \mathcal{E}_{\Delta_1}(G_1)\otimes_{\op{ex}}\mathcal{E}_{\Delta}(G_2)\xrightarrow[]{\sim} \mathcal{E}_{A(\Delta_1\times \Delta_2)}(G_{12})$. More specifically, $E$ is a $D$-unitary map in the sense of Corollary \ref{cor: unitary-is-adjointable} or, in other words, for $f_1^{(i)}\in \mathcal{E}_{\Delta_1}(G_1)$, $f_2^{(i)}\in \mathcal{E}_{\Delta_2}(G_2)$, and $i=1,...,n$,
\begin{equation}\label{form: module-tensor-isometry}
    \begin{split}
    \lin{A(\Delta_1\times \Delta_2)}{E\left(\sum_{i=1}^nf_1^{(i)}\otimes f_2^{(i)}\right) }{E\left(\sum_{j=1}^nf_1^{(j)}\otimes f_2^{(j)}\right)  }\\
     =D\left(\sum_{i=1}^{n}\sum_{j=1}^n  \lin{\Delta_1}{f_1^{(i)}}{f_1^{(j)}} \otimes \lin{\Delta_2}{f_2^{(i)}}{f_2^{(j)}} \right)
    \end{split}
\end{equation}
where $D: C^*(\Delta_1,\varphi_1)\otimes_{\op{sp}} C^*(\Delta_2,\varphi_2)\xrightarrow{\sim} C^*(A(\Delta_1\times \Delta_2),\varphi_{12})$ is the isomorphism described in Proposition \ref{prop: twisted-isomorphism}.
\end{theorem}
\begin{proof}
   From Theorem \ref{thm: concrete-tensor} we know that $E$ is surjective. In the light of Corollary \ref{cor: unitary-is-adjointable}. We do it for simple tensors it suffices to prove \eqref{form: module-tensor-isometry} is sufficient. From Lemma \eqref{lem: simple-tensor-tfs}, we infer that, for any $f_1\in \mathcal{E}_{\Delta_1}(G_1)$ and any $f_2\in \mathcal{E}_{\Delta_2}(G_2)$,
    \begin{align*}
        \lin{A(\Delta_1\times \Delta_2)}{E(f_1\otimes f_2)}{E(f_1\otimes  f_2)}(A(\lambda,\mu)) &= \rin{L^2(G_{12})}{E(f_1\otimes f_2)}{\pi(A(\lambda,\mu))E(f_1\otimes f_2)}\\
        &=\rin{L^2(G_{12})}{E(f_1\otimes f_2)}{E(\pi(\lambda)f_1 \otimes \pi(\mu)f_2)} \\
        &= \int_{G_1\times G_2}f_1(x)\cdot f(y) \cdot \overline{\pi(\lambda)f_1(x) \pi(\mu)f_2(y)}d\mu_{G_1\times G_2}(x,y) \\
        &= \left(\int_{G_1}f_1(x)\overline{\pi(\lambda)f_1(x)}d\mu_{\Delta_1}(x)\right) \\ &\qquad \cdot \left(\int_{G_2}f_2(y)\overline{\pi(\mu)f_2(y)}d\mu_{\Delta_2}(y) \right) \\
        &= \rin{L^2(G_1)}{f_1}{\pi(\lambda)f_1} \cdot \rin{L^2(G_2)}{f_2}{\pi(\mu)f_2} \\
        &= D\left(\lin{\Delta_1}{f_1}{f_1} \otimes \lin{\Delta_2}{f_2}{f_2} \right)(A(\lambda,\mu)).
    \end{align*}
    The linear extension of this calculus results in \eqref{form: module-tensor-isometry}. 
\end{proof}
\begin{remark}\label{rem: E-corollaries}
    It follows from Corollary \ref{cor: unitary-is-adjointable} that $E: \mathcal{E}_{\Delta_1}(G_1)\otimes_{\op{ex}}\mathcal{E}_{\Delta_2}(G_2)\to \mathcal{E}_{A(\Delta_1\times \Delta_2)}(G_{12})$ is $D$-adjointable. It then follows from Proposition \ref{prop: adjointable-auto-properties} that it is $D$-linear, that is
    \begin{align}\label{form: module-tensor-action}
    E\left(\sum_{i=1}^n (a_1^{(i)}\otimes a_2^{(i)})\cdot \sum_{j=1}^n(f_1^{(j)}\otimes f_2^{(j)}) \right) = D\left(\sum_{i=1}^n a_1^{(i)}\otimes a_2^{(i)} \right) \cdot E\left(\sum_{i=1}^n f_1^{(j)}\otimes f_2^{(j)} \right).
\end{align}
Furthermore, the fact that $E$ is $D$-unitary implies that $\|\mathbf{f}\|_{\op{ex}} = \|E(\mathbf{f})\|_{\mathcal{E}_{A(\Delta_1\times \Delta_2)}(G_{12})}$ for any $\mathbf{f}\in \mathcal{E}_{\Delta_1}(G_1)\otimes_{\op{ex}}\mathcal{E}_{\Delta_2}(G_2)$. In particular, we have from the crossnorm property of the external tensor product that $\|E(f_1\otimes f_2)\|_{\mathcal{E}_{A(\Delta_1\times \Delta_2)}(G_{12})} =\|f_1\|_{\mathcal{E}_{\Delta_1}(G_1)} \|f_2\|_{\mathcal{E}_{\Delta_2}(G_2)}.$
\end{remark}
\begin{remark}
   We can obtain a description of Theorem \ref{thm: concrete-external-tensor} in terms of the right module structures by using a result that is analogous to Proposition \ref{prop: twisted-isomorphism}, namely, that we have an isomorphism $C^*(\Delta_1^{\circ},\overline{\varphi}_1)\otimes_{\op{sp}} C^*(\Delta_2^{\circ},\overline{\varphi}_2) \cong C^*(A(\Delta_1^{\circ}\times \Delta_2^{\circ}),\overline{\varphi}_{12}).$
\end{remark}
We recall that, due to the co-compactness of $\Delta_1$ and $\Delta_2$, both $C^*(\Delta_1^{\circ},\overline{\varphi}_1)$ and $C^*(\Delta_2^{\circ},\overline{\varphi}_2)$ are unital. Therefore $C^*(\Delta_1^{\circ},\overline{\varphi}_1)\otimes_{\op{sp}} C^*(\Delta_2^{\circ},\overline{\varphi}_2)\cong C^*(A(\Delta_1^{\circ}\times \Delta_2^{\circ}), \overline{\varphi}_{12})$ is also unital. Due to Theorem \ref{thm: concrete-external-tensor}, $\mathcal{E}_{A(\Delta_1\times \Delta_2)}(G_{12})$ is a finitely generated projective $C^*(A(\Delta_1\times \Delta_2),\varphi_{12})$-module. 
\begin{corollary}
    The elements $\xi_1^{(1)},...,\xi_1^{(n)}\in \mathcal{E}_{\Delta_1}(G_1)$ and $\eta_2^{(1)},...,\eta_2^{(k)}\in \mathcal{E}_{\Delta_2}(G_2)$ are generators of the respective Heisenberg modules if and only if the system
    \begin{align*}
        \bigcup_{i=1}^{n}\bigcup_{j=1}^k \mathcal{G}\left(E\left(\xi_1^{i}\otimes \eta_2^{(j)}\right),A(\Delta_1\times \Delta_2) \right)
    \end{align*}
    is a multi-window Gabor frame for $L^2(G_{12}).$
\end{corollary}
\begin{proof}
    According to Theorem \ref{thm: concrete-external-tensor}, $E: \mathcal{E}_{\Delta_1}(G_1)\otimes_{\op{ex}}\mathcal{E}_{\Delta_2}(G_2)\to \mathcal{E}_{A(\Delta_1\times \Delta_2)}(G_{12})$ is a $D$-unitary map. Since the $\{\xi_1^{(i)}\otimes \eta_2^{(j)}: i=1,...,n\ \text{and }j=1,...,k\}$ is a generating set of the external tensor product $\mathcal{E}_{\Delta_1}(G_1)\otimes_{\op{ex}}\mathcal{E}_{\Delta_2}(G_2)$, it follows from Proposition \ref{prop: send-gen-to-gen} that $\{E(\xi_1^{(i)}\otimes \eta_2^{(j)}): i=1,...,n\ \text{and }j=1,...,k\}$ is a generating set for $\mathcal{E}_{A(\Delta_1\times \Delta_2)}(G_{12})$ as a Hilbert $C^*(A(\Delta_1\times \Delta_2, \varphi_{12}))$-module. The statement now follows from Theorem \ref{thm: localization-heis} Item (3).
\end{proof}

Finally, we can now state the kernel theorems for the Heisenberg modules by combining Theorem \ref{thm: abstract-kernels} with Theorem \ref{thm: concrete-external-tensor}.

\begin{theorem}\label{thm: heis-kernels}
There is the following commutative diagram
\begin{equation}\label{diag: heis-inner-kernel}
\begin{adjustbox}{max width=\textwidth}
\begin{tikzcd}[row sep=2em,column sep=small]
 & k(\mathcal{N}(\mathbf{S}_0'(G_1),\mathbf{S}_0(G_2))) \arrow[r, hook]  & k(\mathcal{N}(\mathcal{E}_{\Delta_1}'(G_1),\mathcal{E}_{\Delta_2}(G_2))) \arrow[r, hook] & \mathcal{M}(\mathcal{E}_{\Delta_1}
'(G_1),\mathcal{E}_{\Delta_2}(G_2)) \arrow[r, hook]  & \mathcal{HS}(L^2(G_1)', L^2(G_2)) \\
& \mathbf{S}_0(G_1) \otimes_{\pi}\mathbf{S}_0(G_2) \arrow[u, "\cong"] \arrow[r, hook] \arrow[d, "\cong"] & \mathcal{E}_{\Delta_1}(G_1)\otimes_{\pi} \mathcal{E}_{\Delta_2}(G_2) \arrow[u, "\cong"] \arrow[r, hook] & \mathcal{E}_{\Delta_1}(G_1)\otimes_{\op{ex}}\mathcal{E}_{\Delta_2}(G_2) \arrow[u, "\cong"] \arrow[r, hook] \arrow[d, "\cong"] & L^2(G_1)\otimes_{2} L^2(G_2)\arrow[d, "\cong"] \arrow[u, "\cong"] \\
&\mathbf{S}_0(G_{12}) \arrow[rr, hook]  &  & \mathcal{E}_{A(\Delta_1\times\Delta_2)}(G_{12}) \arrow[r, hook] & L^2(G_{12}) 
\end{tikzcd}
\end{adjustbox}
\end{equation}
where all the arrows pointing to the right are embeddings with norm dense ranges. There is also the following commutative diagram for the duals 
\begin{equation}\label{diag: heis-outer-kernel}
\begin{adjustbox}{max width=\textwidth}
\begin{tikzcd}[row sep=2em,column sep=small]
&\mathcal{HS}(L^2(G_1)',L^2(G_2)) \arrow[r, hook]  & \mathcal{B}_{\mathcal{M}}(\mathcal{E}_{\Delta_1}(G_1),\mathcal{E}_{\Delta_2}'(G_2)) \arrow[r, hook] \arrow[d, "\cong"] & \mathcal{B}(\mathcal{E}_{\Delta_1}(G_1),\mathcal{E}_{\Delta_2}'(G_2)) \arrow[r, hook] \arrow[d, "\cong"] & \mathcal{B}(\mathbf{S}_0(G_1),\mathbf{S}_0'(G_2)) \arrow[d, "\cong"] \\
& L^2(G_1)\otimes_2 L^2(G_2) \arrow[u, "\cong"]\arrow[r, hook] \arrow[d, "\cong"] & (\mathcal{E}_{\Delta_1}(G_1)\otimes_{\op{ex}}\mathcal{E}_{\Delta_2}(G_2))' \arrow[r, hook] & (\mathcal{E}_{\Delta_1}(G_1)\otimes_{\pi}\mathcal{E}_{\Delta_2}(G_2))' \arrow[r, hook] & (\mathbf{S}_0(G_1)\otimes_{\pi}\mathbf{S}_0(G_2))' \\
&L^2(G_{12}) \arrow[r, hook] &  \mathcal{E}_{A(\Delta_1\times \Delta_2)}'(G_{12})\arrow[rr, hook] \arrow[u, "\cong"] & & \mathbf{S}_0'(G_{12}) \arrow[u, "\cong"]
\end{tikzcd}
\end{adjustbox}
\end{equation}
where all the arrows pointing to the right are embeddings with weak-$*$ dense ranges.
\end{theorem}
\begin{proof}
 We note that, for $i=1, 2$, $\mathcal{E}_{\Delta_i}(G_i)$ are faithfully localisable from the $\mathbf{S}_0(\Delta_i,\varphi_i)-\mathbf{S}_0(\Delta_i^{\circ},\varphi_i)$ pre-equivalence bimodules $\mathbf{S}_0(G_i)$. We know that $\mathbf{S}_0(G_i)$ as a Banach space has a topology that is finer than that induced by its bimodule structure due to the latter inequality in \eqref{ineq: heis-ineqs}. We also know that $\mathbf{S}_0(\Delta_i,\varphi_i)$, $\mathbf{S}_0(\Delta_i^{\circ},\overline{\varphi}_i)$ and $\mathbf{S}_0(G_i)$ possess the approximation property, according to Theorem \ref{thm: feich-has-ap}. Since $C^*(\Delta_i^{\circ},\overline{\varphi}_i)$ is unital, the Heisenberg modules satisfy the hypotheses of Theorem \ref{thm: proj-tensor-are-preequiv}. The result now follows as a corollary of Theorem \ref{thm: abstract-kernels}.
\end{proof}
Let us denote the duality brackets in $\mathcal{E}_{\Delta_i}(G_i)$ and $\mathcal{E}_{A(\Delta_1\times \Delta_2)}(G_{12})$ and $\mathcal{E}_{\Delta_1}(G_1)\otimes_{\op{ex}}\mathcal{E}_{\Delta_2}(G_2)$ by
\begin{align*}
    \rin{\mathcal{E}_i, \mathcal{E}_i}{\cdot}{\cdot},\ \rin{\mathcal{E}_{12}, \mathcal{E}'_{12}}{\cdot}{\cdot},\ \text{and } \rin{(\mathcal{E}_1\otimes_{\op{ex}} \mathcal{E}_2), (\mathcal{E}_1\otimes_{\op{ex}\mathcal{E}_2})'}{\cdot }{\cdot }
\end{align*}
respectively. Using the function space interpretation of the external tensor product of the Heisenberg modules, we obtain a concrete version of the outer kernel theorem for Heisenberg modules.
\begin{theorem}[Outer Kernel Theorem for Heisenberg Modules]\label{thm: heis-outer-kernel}
For any module continuous map $T\in \mathcal{B}_{\mathcal{M}}(\mathcal{E}_{\Delta_1}(G_1), \mathcal{E}_{\Delta_2}'(G_2))$, there exists a unique \emph{kernel} $\kappa(T)\in \mathcal{E}_{A(\Delta_1\times \Delta_2)}'(G_{12})$ that satisfies
\begin{align}\label{form: explicit-outer-kernel}
\rin{\mathcal{E}_2,\mathcal{E}_2'}{f_2}{Tf_1} = \rin{\mathcal{E}_{12},\mathcal{E}'_{12}}{E(f_1\otimes f_2)}{\kappa(T)}
\end{align}
for any $f_1\in \mathcal{E}_{\Delta_1}(G_1)$ and any $f_2\in \mathcal{E}_{\Delta_2}(G_2).$
\end{theorem}
\begin{proof}
 We know that, via $b(T)$ of Lemma \ref{lem: module-continuous-maps}, $T\in \mathcal{B}_{\mathcal{M}}(Z_1,Z_2')$ becomes a linear functional in $(\mathcal{E}_{\Delta_1}(G_1)\otimes_{\op{ex}}\mathcal{E}_{\Delta_2}(G_2))'$  that satisfies $$\rin{(\mathcal{E}_1\otimes_{\op{ex}} \mathcal{E}_2), (\mathcal{E}_1\otimes_{\op{ex}\mathcal{E}_2})'}{f_1\otimes f_2}{b(T)}= \rin{\mathcal{E}_2,\mathcal{E}_2'}{f_2}{Tf_1}$$ for any $f_1\in \mathcal{E}_{\Delta_1}(G_1)$ and any $f_2\in \mathcal{E}_{\Delta_2}(G_2)$. On the other hand, from Theorem \ref{thm: embedding-dense-range} we know that the formal dual of $E$ in Theorem \ref{thm: concrete-external-tensor} implements the isomorphism $E': \mathcal{E}_{A(\Delta_1\times\Delta_2)}'(G_{12})\xrightarrow[]{\sim} (\mathcal{E}_{\Delta_1}(G_1)\otimes_{\op{ex}}\mathcal{E}_{\Delta_2}(G_2))'$. We define $\kappa(T) \in \mathcal{E}_{A(\Delta_1\times \Delta_2)}'(G_{12})$ to be the pre-image of $b(T)$ under $E'$. It follows that
 \begin{align*}
     \rin{\mathcal{E}_{12},\mathcal{E}_{12}'}{E(f_1\otimes f_2)}{\kappa(T)} &= \rin{(\mathcal{E}_1\otimes_{\op{ex}} \mathcal{E}_2), (\mathcal{E}_1\otimes_{\op{ex}\mathcal{E}_2})'}{f_1\otimes f_2}{E'(\kappa(T))}\\
     &= \rin{(\mathcal{E}_1\otimes_{\op{ex}} \mathcal{E}_2), (\mathcal{E}_1\otimes_{\op{ex}\mathcal{E}_2})'}{f_1\otimes f_2}{b(T)} = \rin{\mathcal{E}_2,\mathcal{E}_2'}{f_2}{Tf_1}.
 \end{align*}
\end{proof}
\begin{theorem}[Inner Kernel Theorem for Heisenberg Modules]\label{thm: heis-inner-kernel}
    Any module continuous map $T\in \mathcal{B}_{\mathcal{M}}(\mathcal{E}_{\Delta_1}(G_1), \mathcal{E}'_{\Delta_2}(G_2))$ with a kernel $\kappa(T)\in \mathcal{E}_{A(\Delta_1\times \Delta_2)}(G_{12})$ can be extended to the map $\widetilde{T}:\mathcal{E}_{\Delta_1}'(G_1)\to \mathcal{E}_{\Delta_2}(G_2)$ where $\widetilde{T}\in \mathcal{M}(\mathcal{E}_{\Delta_1}'(G_1),\mathcal{E}_{\Delta_2}(G_2))$.
\end{theorem}
\begin{proof}
    Let $T\in \mathcal{B}_{\mathcal{M}}(\mathcal{E}_{\Delta_1}(G_1), \mathcal{E}'_{\Delta_2}(G_2))$ with a kernel $\kappa(T)\in \mathcal{E}_{A(\Delta_1\times \Delta_2)}(G_{12})$. It follows from Theorem \ref{thm: concrete-external-tensor} that $\displaystyle \kappa(T) = \lim_{N\to \infty} \sum_{i_N=1}^{K_N} E(h_1^{(i_N)}\otimes h_2^{(i_N)})$ where $$\displaystyle \mathbf{h}=\lim_{N\to \infty}\displaystyle \sum_{i_N=1}^{K_N} h_1^{(i_N)}\otimes h_2^{(i_N)}\in \mathcal{E}_{\Delta_1}(G_1)\otimes_{\op{ex}} \mathcal{E}_{\Delta_2}(G_2).$$ Using the inclusion $\mathcal{E}_{\Delta_1}(G_1)\otimes_{\op{ex}}\mathcal{E}_{\Delta_2}(G_2)\hookrightarrow (\mathcal{E}_{\Delta_1}(G_1)\otimes_{\op{ex}}\mathcal{E}_{\Delta_2}(G_2))'$ (see proof of Lemma \ref{lem: localizable-is-gelfand-triple}), we infer that, for any $f_2\in \mathcal{E}_{\Delta_2}(G_2)$ and $f_1\in \mathcal{E}_{\Delta_1}(G_1)$
    \begin{align*}
        \rin{}{f_2}{Tf_1} &= \rin{}{E(f_1\otimes f_2)}{\kappa(T)} = \lim_{N\to \infty} \sum_{i_N=1}^{K_N} \rin{L^2(G_1)\otimes_2 L^2(G_2)}{E(f_1\otimes f_2)}{E(h_1^{(i_N)}\otimes h_2^{(i_N)})} \\
        &= \lim_{N\to \infty} \sum_{i_N=1}^{K_N} \rin{L^2(G_1)}{f_1}{h_1^{(i_N)}} \cdot \rin{L^2(G_2)}{f_2}{h_2^{(i_N)}} = \rin{L^2(G_2)}{f_2}{ \lim_{N\to \infty} \sum_{i_N=1}^{K_N} \rin{\mathcal{E}_1,\mathcal{E}_1'}{h_1^{(i_N)}}{f_1}\cdot h_2^{(i_N)}  } \\
        &=\rin{\mathcal{E}_2,\mathcal{E}_2'}{f_2}{ \lim_{N\to \infty} \sum_{i_N=1}^{K_N} \rin{\mathcal{E}_1,\mathcal{E}_1'}{h_1^{(i_N)}}{f_1}\cdot h_2^{(i_N)}} = \rin{\mathcal{E}_{2},\mathcal{E}_2'}{f_2}{\mathfrak{R}(\mathbf{h})(f_1)},
    \end{align*}
    where $\mathfrak{R}: \mathcal{E}_{\Delta_1}(G_1)\otimes_{\op{ex}}\mathcal{E}_{\Delta_2}(G_2)\xrightarrow[]{\sim} \mathcal{M}(\mathcal{E}_{\Delta_1}'(G_1),\mathcal{E}_{\Delta_2}'(G_2)$ is the isomorphism of Proposition \ref{prop: ex-tensor-as-operator}. From Lemma \ref{lem: for-external-embed} and the computations above we know that $Tf_1 = \mathfrak{R}(\mathbf{h})(f_1)\in \mathcal{E}_{\Delta_2}(G_2)$ for any $f_1\in \mathcal{E}_{\Delta_1}(G_1).$ Due to the embedding $\mathcal{M}(\mathcal{E}_{\Delta_1}'(G_1), \mathcal{E}_{\Delta_2}(G_2))\hookrightarrow \mathcal{B}_{\mathcal{M}}(\mathcal{E}_{\Delta_1}(G_1),\mathcal{E}_{\Delta_2}'(G_2))$ of Theorem \ref{thm: heis-kernels}, $T = \mathfrak{R}(\mathbf{h})_{|\mathcal{E}_{\Delta_1}(G_1)}$ and $T$ can be extended to $\widetilde{T}:= \mathfrak{R}(\mathbf{h})\in \mathcal{M}(\mathcal{E}_{\Delta_1}'(G_1), \mathcal{E}_{\Delta_2}(G_2))$. 
    \end{proof}
\begin{example}
    Let $\mathbf{h}\in L^2(G_{12})$, $\mathbf{s}\in L^2(\Delta_1\times \Delta_2)$ and $\mathbf{h}\in L^2(G_{12})$ with the following Hilbert space tensor product decomposition: $\displaystyle \mathbf{h} = \lim_{N\to \infty} \sum_{i_N=1}^{K_N}E(h_1^{(i_N)}\otimes h_2^{(i_N)})$. Let us formally define
    \begin{align*}
        T_{\mathbf{s},\mathbf{h}}f_1 = \lim_{N\to\infty}\sum_{i_N=1}^{K_N}\int_{\Delta_1\times \Delta_2} \mathbf{s}(\lambda_1,\lambda_2) \rin{L^2(G_1)}{h_1^{(i_N)}}{\pi(\lambda_1)f_1} \pi^*(\lambda_2)h_2^{(i_N)}
    \end{align*} for any $f_1\in \mathcal{E}_{\Delta_1}(G_1)$. Then, for $f_2\in \mathcal{E}_{2}(G_2)$,
    \begin{align*}
        \rin{\mathcal{E}_2,\mathcal{E}_2'}{f_2}{T_{\mathbf{s},\mathbf{h}}f_1} &:= \rin{L^2(G_2)}{f_2}{T_{\mathbf{s},\mathbf{h}}f_1}\\
        &= \lim_{N\to \infty} \sum_{i_N=1}^{K_N} \int_{\Delta_1\times \Delta_2} \mathbf{s}(\lambda_1,\lambda_2) \rin{L^2(G_1)}{\pi(\lambda_1)f_1}{h_1^{(i_N)}} \rin{L^2(G_2)}{\pi(\lambda_2)f_2}{h_2^{(i)}}d\mu_{\Delta_1\times\Delta_2}(\lambda_1,\lambda_2)\\
        &= \int_{\Delta_1\times\Delta_2}\mathbf{s}(\lambda_1,\lambda_2) \rin{L^2(G_{12})}{\pi(A(\lambda_1,\lambda_2))E(f_1\otimes f_2)}{\mathbf{h}} d\mu_{\Delta_1\times \Delta_2}(\lambda_1,\lambda_2)
    \end{align*}
Using the fact that $E(f_1\otimes f_2)\in \mathcal{E}_{A(\Delta_1\times \Delta_2)}(G_{12})$, and applying the Cauchy-Schwarz inequality to $L^2(\Delta_1\times \Delta_2)$ implies that: 
\begin{align}\label{form: heis-functional-estimate}
    |\rin{\mathcal{E}_2,\mathcal{E}_2'}{f_2}{T_{\mathbf{s},\mathbf{h}}f_1}|\leq \|\mathbf{s}\|_{L^2(\Delta_1\times\Delta_2)} \|E(f_1\otimes f_2)\|_{\mathcal{E}_{A(\Delta_1\times \Delta_2)}(G_{12})} \|\mathbf{h}\|_{L^2(G_{12})}.
\end{align}
Since $\|E(f_1\otimes f_2)\|_{\mathcal{E}_{A(\Delta_1\times \Delta_2)}(G_{12})}= \|f_1\|_{\mathcal{E}_{\Delta_1}(G_1)} \|f_2\|_{\mathcal{E}_{\Delta_2}(G_2)}$ (according to Remark \ref{rem: E-corollaries}), it follows that $T_{\mathbf{s},\mathbf{h}}: \mathcal{E}_{\Delta_1}(G_1)\to \mathcal{E}_{\Delta_2}'(G_2)$ is a bounded linear operator. To get in a position to apply our outer kernel theorem, we still need to show that $T_{\mathbf{s},\mathbf{h}}$ is module continuous, as defined by Definition \ref{def: module-continuous}. However, it follows from the same estimate that resulted in \eqref{form: heis-functional-estimate} that
\begin{align*}
    \left|\sum_{i=1}^N \rin{\mathcal{E}_2,\mathcal{E}_2'}{f_2^{(i)}}{T_{\mathbf{s},\mathbf{h}}f_1^{(i)}}\right| &\leq \|\mathbf{s}\|_{L^2(\Delta_1\times\Delta_2)} \left\|\sum_{i=1}^N E(f_1^{(i)}\otimes f_2^{(i)})\right\|_{\mathcal{E}_{A(\Delta_1\times \Delta_2)}(G_{12})} \|\mathbf{h}\|_{L^2(G_{12})}\\
    &= \|\mathbf{s}\|_{L^2(\Delta_1\times\Delta_2)} \left\|\sum_{i=1}^N f_1^{(i)}\otimes f_2^{(i)}\right\|_{\op{ex}} \|\mathbf{h}\|_{L^2(G_{12})}
\end{align*}
for any $\displaystyle \sum_{i=1}^N f_1^{(i)}\otimes f_2^{(i)}\in \mathcal{E}_{\Delta_1}(G_1)\odot \mathcal{E}_{\Delta_2}(G_2)$ and so $T_{\mathbf{s},\mathbf{h}}\in \mathcal{B}_{\mathcal{M}}(\mathcal{E}_{\Delta_1}(G_1),\mathcal{E}_{\Delta_2}'(G_2))$. On the other hand, for any $\mathbf{f}\in \mathcal{E}_{A(\Delta_1\times \Delta_2)}(G_{12})$, we define
    \begin{align*}
        M_{\mathbf{s},\mathbf{h}}(\mathbf{f}) := \int_{\Delta_1 \times \Delta_2}\mathbf{s}(\lambda_1,\lambda_2) \rin{L^2(G_{12})}{\pi(A(\lambda_1,\lambda_2))\mathbf{f}}{\mathbf{h}}d\mu_{\Delta_1\times \Delta_2}(\lambda_1,\lambda_2).
    \end{align*}
    This integral converges because $\mathcal{G}(\mathbf{f}, \pi(A(\Delta_1\times\Delta_2))$ is a Bessel family in $L^2(G_{12})$, which also gives us $|M_{\mathbf{s},\mathbf{h}}(\mathbf{f})|\leq \|\mathbf{s}\|_{L^2(\Delta_1\times\Delta_2)} \|\mathbf{h}\|_{L^2(G_{12})} \|\mathbf{f}\|_{\mathcal{E}_{A(\Delta_1\times\Delta_2)}(G_{12})}$. Therefore $M_{\mathbf{s},\mathbf{h}}\in \mathcal{E}_{A(\Delta_1\times\Delta_2)}'(G_{12})$. A straightforward calculation and the outer kernel theorem for Heisenberg modules imply that $\kappa(T_{\mathbf{s},\mathbf{h}}) = M_{\mathbf{s},\mathbf{h}}$.
\end{example}
\begin{example}
    Let us consider the Heisenberg module $\mathcal{E}_{\Delta}(G)$ and let $h,k\in \mathcal{E}_{\Delta}(G)$. We define a rank-one operator by $R_{h,k}f=\rin{L^2(G)}{h}{f}k$ and see that $R_{h,k}:\mathcal{E}_{\Delta}(G)\to \mathcal{E}_{\Delta}'(G)$ as 
    \begin{align*}
        \rin{\mathcal{E},\mathcal{E}'}{g}{R_{k.h}f} &:= \lin{L^2(G)}{g}{R_{k,h}f} = \rin{L^2(G)}{f}{h} \rin{L^2(G)}{g}{k} = \rin{L^2(G_{12})}{E(f\otimes g)}{E(h\otimes k)}\\
        &=(\op{tr}_{\Delta}\otimes \op{tr}_{\Delta})(\lin{A(\Delta\times \Delta)}{E(f\otimes g)}{E(h\otimes k)}).
    \end{align*}
    A linear extension results in
    \begin{align*}
        \sum_{i=1}^N\rin{\mathcal{E},\mathcal{E}'}{g^{(i)}}{R_{k.h}f^{(i)}} = (\op{tr}_{\Delta}\otimes \op{tr}_{\Delta})\left(\lin{A(\Delta\times \Delta)}{\sum_{i=1}^nE\left(f^{(i)}\otimes g^{(i)}\right)}{E(h\otimes k)}\right).
    \end{align*}
    Continuity of the traces and the Cauchy-Schwarz inequality for Hilbert C*-modules results in the following estimate:
    \begin{align*}
        \left|\sum_{i=1}^N\rin{\mathcal{E},\mathcal{E}'}{g^{(i)}}{R_{k.h}f^{(i)}}\right|\leq |\Delta| \left\|\sum_{i=1}^NE(f^{(i)}\otimes g^{(i)})\right\|_{\mathcal{E}_{A(\Delta\times \Delta)}(G\times G)}  \cdot \|E(h\otimes k)\|_{\mathcal{E}_{A(\Delta\times\Delta)}(G\times G)}.
    \end{align*}
    As in the previous example, it follows that $R_{h,k}\in \mathcal{B}_{\mathcal{M}}(\mathcal{E}_{\Delta}(G),\mathcal{E}_{\Delta}'(G))$, and it is not hard to see that the kernel is precisely $E(h\otimes k) \in \mathcal{E}_{A(\Delta\times \Delta)}(G\times G)$. The inner kernel theorem indicates that $R_{h,k}$ can be extended and this extension is given by $\mathfrak{R}(h\otimes k)\in \mathcal{M}(\mathcal{E}'_{\Delta}(G), \mathcal{E}_{\Delta}(G))$.
\end{example}
As a final remark we would like to point out two directions that we shall take in our further investigation of the matters dealt with in this report and that we shall address in its sequel. First, we shall provide more examples by studying $\tau$-unitary maps between Heisenberg modules in more detail. Such maps can be constructed by examining how unitary maps at the $L^2$-level interact with closed subgroups of the time-frequency plane. Second, we shall investigate whether the projective tensor product of Heisenberg modules can be identified as function spaces. Establishing such isomorphisms would complete the commutative diagrams of Theorem \ref{thm: heis-kernels} and lead to more kernel theorems involving Heisenberg modules.
\section*{Acknowledgements}
Arvin Lamando was supported by the Computational Research Laboratory, Institute of Mathematics, University of the Philippines Diliman, through a Research Grant in Pure and Applied Mathematics, funded under the General Appropriations Act as an Allotment Order (GAAAO).
\bibliographystyle{myalphaurl-sortbyauthor}
\bibliography{master}
\linespread{1}
\Addresses
\end{document}